\documentclass[a4paper,12pt]{article}

\usepackage[applemac]{inputenc}
\usepackage{amsmath}
\usepackage{slashed}
\usepackage{amscd}\usepackage{amstext}\usepackage{amsbsy}\usepackage{amsopn}\usepackage{amsthm}\usepackage{amsxtra}\usepackage{upref}\usepackage{amsfonts}
\usepackage{amssymb}\usepackage{euscript}
\usepackage{latexsym}\usepackage{graphicx}\usepackage{color}\usepackage[all]{xy}
\usepackage[american]{babel}
\usepackage{tikz}
\usepackage{tikz-3dplot}
\usepackage{mdframed}
\usepackage{geometry}
\usepackage{subcaption}
\usepackage{hyperref}
\usepackage{todonotes}
\usepackage[nocfg]{nomencl}
\makenomenclature

\newtheorem{remark}{Remark}
\newtheorem*{remark*}{Remark}
\newtheorem{definition}{Definition}
\newtheorem{proposition}{Proposition}
\newtheorem{corollary}{Corollary}
\newtheorem{example}{Example}
\newtheorem{theorem}{Theorem}
\newmdtheoremenv{framedtheorem}{Theorem}

\newcommand{\bip}{\mathcal{B}}

\newcommand{\psl}{\mathrm{PSL}_2(\mathbb{Z})}
\newcommand{\aut}{\mathrm{Aut}}
\newcommand{\en}{\mathrm{End}}

\newcommand{\He}{\Gamma_0}
\newcommand{\GL}{\mathrm{GL}}
\newcommand{\PGL}{\mathrm{PGL}}
\newcommand{\SL}{\mathrm{SL}}
\newcommand{\Bp}{\mathcal{B}^+}
\newcommand{\PBp}{\mathrm{P}\mathcal{B}^+}
\newcommand{\La}{\mathcal{L}}
\newcommand{\PLa}{\mathrm{P}\mathcal{L}}
\newcommand{\fix}{\mathrm{Fix}}

\graphicspath{{img/}}

\begin{document}

\title{Cusps, Congruence Groups and Monstrous Dessins}
\author{Valdo Tatitscheff${}^{1,2,3}$ \and Yang-Hui He${}^{2,4,5}$ \and John McKay${}^{6}$}

\date{}
 
\maketitle

\begin{center}
{\small
\begin{tabular}{ll}
${}^{1}$&
Department of Mathematics, Ecole Normale Sup\'erieure, Paris 75005, France\\
${}^{2}$&
Department of Mathematics, City, University of London, EC1V 0HB, UK\\
${}^{3}$&
IRMA, UMR 7501, Universit\'e de Strasbourg et CNRS, \\
& 7, rue Ren\'e Descartes, 67000 Strasbourg, France \\
${}^{4}$&
Merton College, University of Oxford, OX14JD, UK \\
${}^{5}$&
School of Physics, NanKai University, Tianjin, 300071, P.R.~China \\
${}^{6}$
&Department of Mathematics and Statistics,\\
&Concordia University, 1455 de Maisonneuve Blvd.~West,\\
&Montreal, Quebec, H3G 1M8, Canada\\
~\\
~\\
\end{tabular}
}
\url{valdo.tatitscheff@normalesup.org}\\
\url{hey@maths.ox.ac.uk}\\
\url{mckay@encs.concordia.ca}
\end{center}

\vspace{0.5in}

\abstract{We study general properties of the dessins d'enfants associated with the Hecke congruence subgroups $\He(N)$ of the modular group $\psl$. The definition of the $\He(N)$ as the stabilisers of couples of projective lattices in a two-dimensional vector space gives an interpretation of the quotient set $\He(N)\backslash\psl$ as the projective lattices $N$-hyperdistant from a reference one, and hence as the projective line over the ring $\mathbb{Z}/N\mathbb{Z}$. The natural action of $\psl$ on the lattices defines a dessin d'enfant structure, allowing for a combinatorial approach to features of the classical modular curves, such as the torsion points and the cusps. 
We tabulate the dessins d'enfants associated with the $15$ Hecke congruence subgroups of genus zero, which arise in Moonshine for the Monster sporadic group.} 

\newpage

\section*{Introduction and motivations}

\subsection*{Monstrous moonshine}

\def\II{\mathbb{I}}
\def\IM{\mathbb{M}}
\def\IR{\mathbb{R}}

The vast subject of Moonshine began with the third author's observation, initially thought to be outlandish, that
\begin{equation}\label{McKay}
196,884 = 196,883 + 1 \ .
\end{equation}
The number on the left is the linear Fourier coefficient of the Klein $J$-function, and lives in the world of modular forms, while the number on the right comes from the first two irreducible representations of the Monster sporadic group $\IM$, and lives in the world of finite group theory. These two fields are seemingly disparate.

Based on the observation in Equation \ref{McKay} and generalisations of it, Thompson conjectured in \cite{thompson} that there exists a natural graded infinite-dimensional representation $W^\natural=\bigoplus_{n=-1}^\infty W^\natural_n$ of $\IM$, such that $(\dim(W^\natural_n))_n$ is the sequence of Fourrier coefficients of Klein's $J$-function, and Atkin, Fong and Smith verified that such an $\IM$-module exists \cite{smith}. The construction of this module was later given in \cite{FLM} by Frenkel, Lepowsky and Meurman, thus proving Thompson's conjecture. 

The latter had also further suggested to investigate the properties of the graded-traced functions now called {\bf McKay-Thompson series}
\[
T_{\bar{g}}(q) = q^{-1} \sum\limits_{k=0}^\infty {\rm ch}_{W_k^\natural}(\bar{g}) q^k 
=  q^{-1} + 0 + h_1(\bar{g})q + h_2(\bar{g}) q^2 + \ldots \ ,
\]
where ${\rm ch}_{W_k^\natural}(\bar{g})$ denotes the character of the representation $W_k^\natural$ of $\IM$, evaluated on the conjugacy class $\bar{g}$. This ultimately prompted the Monstrous Moonshine conjectures of \cite{conwaynorton}: each McKay-Thompson series $T_{\bar{g}}(q)$ corresponding to a conjugacy class $\bar{g}$ in $\IM$ is, for $q=\exp(2\pi i z)$, the (normalised) generator of a {\em genus zero function field} for a group $G$ between the Hecke group $\Gamma_0(N)$ of level $N$ and its normaliser $\Gamma_0(N)^+$ in $PSL(2,\IR)$, generated by $\He(N)$ and certain Atkin-Lehner involutions \cite{atkin}. Moreover, the level $N$ is a multiple of $n = \mbox{Order}(\bar{g})$, the ratio $N / n = h \in \mathbb{Z}_{>0}$ divides $24$, and $h^2$ divides $N$. 
In particular, for the conjugacy class of the identity the McKay-Thompson series is the Fourier expansion of the $J$-function. The latter generates the function field of the genus zero quotient of the Poincaré half-plane by the modular group $\psl$.

Borcherds proved these conjectures in \cite{borcherds}, using in a central way the monster module constructed by Frenkel, Lepowsky and Meurman. 

There are 194 conjugacy classes (and hence 194 irreducible representations) of $\IM$ (see \cite{atlas}) and due to complex conjugation they give only 172 distinct McKay-Thompson series (which are not independent: linear relations brings the number of independent series down to 163).
Each of these 172 conjugacy classes corresponds to a group $G_{\overline{g}}$ which lies (strictly,for most of them) between $\Gamma_0(N)$ and $\Gamma_0(N)^+$. Precisely 15 correspond to the Hecke groups of our concern (and do not involve Atkin-Lehner involutions). 
 
Each group $G_{\bar{g}}$ is a subgroup of $\mathrm{PSL}_2(\mathbb{R})$, hence it defines a complex surface: the quotient of the upper half-plane $\mathbb{H}$ by $G_{\bar{g}}$. This complex surface is always of genus $0$, has hyperbolic cusps and may have torsion points. The tabulation of the conjugacy classes of the Monster, together with quantities related to them through the moonshine correspondence (such as the number of cusps of the corresponding modular curve), is given in \cite{conwaynorton}.

For more details on the Monstrous Moonshine programme, see the excellent accounts \cite{gannon,Duncan:2014vfa}.

\subsection*{Cusps and exceptional Lie algebras}

The motivation for this work essentially comes from some observations listed in \cite{hemckay}, which let one hope for some links between the three biggest sporadic finite simple groups (the monster group $\mathbb{M}$, the baby monster $\mathbb{B}$ and Fischer's sporadic group $Fi_{24}'$), and the three biggest exceptional Lie algebras ($E_8$, $E_7$ and $E_6$). While the purpose of this paper is not to address such possible correspondences, and is dedicated to the exposition of a new approach to some properties of the Hecke groups from a purely combinatorial point of view, let us nevertheless review briefly those intriguing epiphanies. \\

The most famous observation (that we will leave aside) is known as \textit{McKay's monstrous $E_8$ observation} (see \cite{Conway1985}, \S 14). The conjugacy classes of the monster group are conventionally labeled with a number and a letter, where the number is the order of the elements in this class and the letter, a label which distinguishes the different classes with that order. In particular, there are two conjugacy classes of order $2$, denoted $2A$ and $2B$. Multiplying two elements of the class $2A$ yields an element which is in one of the conjugacy classes $1A, 2A, 2B, 3A, 3C, 4A, 4B, 5A$ or $6A$. The third author noticed a striking correspondence between this sequence and the extended $E_8$ diagram. The same type of phenomenon happens between the elements of the pairs $(\mathbb{B},E_7)$ and $(Fi_{24},E_6)$. \\

The number of \textit{cusps} of the modular curves corresponding to the conjugacy classes in $\mathbb{M}$ is either $1$, $2$, $3$, $4$, $6$ or $8$. The total number of cusps of the modular curves appearing in the monstrous moonshine correspondence for the group $\mathbb{M}$ (respectively, $2\mathbb{B}$, and $3Fi_{24}$ which are subgroups of $\mathbb{M}$) is $360 =3\times120$ (respectively, $448 =2^3\times56$, and $440=2^3(2\times27+1)$). The exceptional Lie algebra $\mathfrak{e_8}$ has $120$ positive roots (respectively, $56$ is the dimension of the smallest fundamental representation of $\mathfrak{e_7}$, and $27$ is the dimension of the adjoint representation of $\mathfrak{e_6}$). Any relationship between sporadic groups and exceptional Lie algebras would be quite amazing, and thus we are eyeing a better understanding of the cusps of those modular surfaces. \\

The \textit{coincidence} that directly motivates this article concerns the Hecke congruence subgroups of $\psl$, which are denoted $\He(N)$. These define special modular curves called the \textit{classical modular curves} (and denoted $X_0(N)$). Those of genus zero all appear in the monstrous moonshine correspondence as linked to conjugacy classes in $\mathbb{M}$. It is known that among the $X_0(N)$, $15$ of them exactly have genus $0$, namely when 
$$N\in\{1,2,3,4,5,6,7,8,9,10,12,13,16,18,25\}=:I_0\ .$$ 

Let $c(N)$ denote the number of cusps of $X_0(N)$. Then:
$$\sum_{N\in I_0} c(N)=56\ ,$$
$$\sum_{N\in I_0} c(N)^2=266=2\times 133\ .$$
The two numbers $56$ and $133$ are respectively the dimensions of the smallest fundamental representation and of the adjoint representation of the exceptional Lie algebra $\mathfrak{e}_7$. The relationship between the set of cusps of the Hecke subgroups of genus zero and $\mathfrak{e_7}$ still remains to be established, if any.

The approach developed in this paper (initially thought as an auxiliairy way to define the cusps of the Hecke groups) yields a nice combinatorial framework to study the classical modular curves. There is no need for complex geometry nor elliptic elements of $\mathrm{PSL}_2(\mathbb{R})$ in order to define and study the cusps of the Hecke groups - complex geometry only appears as one speaks of Hauptmoduln, such as Klein's invariant $J$. If the Lie algebras are supposed to connect with monstrous moonshine through the cusps of the modular curves, this simpler framework may be of some interest.

\section*{Summary and plan}

The cornerstone of what follows is Conway's approach to arithmetic groups in terms of their action on projective lattices in a real vector space \cite{conway}. Because we are mainly following the introduction to these ideas given in \cite{duncan}, moreover presented in details (in the specific framework we are interested in) in Appendix \ref{app}, we get to the heart of the matter as directly as possible. 

Section \ref{sec1} aims at a combinatorial description of the quotient set $\He(N)\backslash\psl$. This set is naturally identified with the set $\mathrm{P}\mathcal{L}_1^N$ of projective lattices $N$-hyperdistant from a reference $L_1$, which is itself in bijection with $\mathbb{P}^1(\mathbb{Z}/N\mathbb{Z})$, the projective line over the ring $\mathbb{Z}/N\mathbb{Z}$. The resulting bijection
\begin{equation}\label{isomintro}
\He(N)\backslash\psl \simeq \mathbb{P}^1(\mathbb{Z}/N\mathbb{Z})
\end{equation}
becomes very interesting as one studies the right action of $\psl$ on $\He(N)\backslash\psl$. The projective line indeed has homogeneous coordinates, in terms of which the right action of $\psl$ takes a pretty guise. The bijection in Equation \ref{isomintro} is elementary and known since long - it appears for example in \cite{Manin_1972}; the derivation given below however has the advantage of being elementary and quite straightforward, the third description of this set as a set of projective lattices being of great help. Conversely, homogeneous coordinates on $\mathbb{P}^1(\mathbb{Z}/N\mathbb{Z})$ provide coordinates on $\mathrm{P}\mathcal{L}_1^N$, which can thus be described in details.

In section \ref{sec2} we first review some general features of Grothendieck's dessins d'enfants, and then investigate some of the properties of the special dessins associated with the $\He(N)$. The set of edges of the latter is naturally in bijection with $\He(N)\backslash\He(1)\simeq\La_1^N\simeq\mathbb{P}^1(\mathbb{Z}/N\mathbb{Z})$. Homogeneous coordinates on the projective line provide an algorithmic way to compute these dessins d'enfants, and hence to understand the structure of the modular curves associated with the $\He(N)$.
 
The number of torsion points, as well as the cusps and their width, are controlled by elementary algebraic equations. These equations are also know since long - they appear for example in \S 1.6 of \cite{shimura} or as Prop. 2.2 in \cite{Manin_1972}, but our approach seems interesting {\em per se}. We compute the Dirichlet $L$-series corresponding to the sequence $(c(N))_{N\geq1}$, and express this series in terms of the Riemann $\zeta$-function. For the sake of completeness, we explain in some details how one goes from our dessins d'enfants associated with the $\He(N)$, to the complex modular curves $X_0(N)$. Since explicit rational parametrisations of the genus zero classical modular curves are known, there are explicit expressions of Bely\u{\i} maps which yield the Hecke dessins d'enfants of genus $0$, and we tabulate them.

Section \ref{sec3} displays, for each of the $15$ Hecke modular groups of genus $0$, a fundamental domain in $\mathbb{H}$, the corresponding dessin d'enfants, and a list of its cusps in terms of projective lattices.

\tableofcontents

\section*{Nomenclature}

\begin{itemize}
	\item Real segments will be written in a standard way: 
	$$[a,b],\ ]a,b[,\ ]a,b]\ \mathrm{or}\ [a,b[\ ,$$ 
	where $a,b\in\mathbb{R}\cup\{\pm\infty\}$, depending on whether they are closed, open, open-closed or closed-open. 
	\item For $M$ and $M$ two integers, $[|M,N|]$ denotes the set of integers between $M$ and $N$, $[|M,N|[$, the set of integers between $M$ and $N$ excluding $N$, ...
	\item The set $\mathrm{Div}(N)$ is the set of positive divisors of a non-zero positive integer $N$. 
	\item For $k,N\in\mathbb{N}$, $k$ divides $N$ is written $k|N$. 
	\item The group of permutations of a set $E$ is denoted $\mathfrak{S}(E)$. 
	\item If $H$ and $G$ are two groups, $H<G$ means that $H$ is a subgroup of $G$.
\end{itemize}

Let now $V$ be a two-dimensional real vector space with basis $(e_1,e_2)$. Lattices in $V$ are by definition the $\mathbb{Z}$-submodules of $V$ isomorphic to $\mathbb{Z}^2$. Since we will also need projective lattices, regular lattices (the ones we just defined) are often referred to as non-projective lattices. Let $$L_1^{np}=\mathbb{Z}\cdot e_1 +\mathbb{Z}\cdot e_2$$ 
be the non-projective lattice generated by the vectors of the basis $(e_1,e_2)$. The set $\mathcal{L}$ of non-projective lattices in $V$ is in bijection with with $\SL_2(\mathbb{Z})\backslash\GL_2^+(\mathbb{R})$. A lattice $L^{np}\in\mathcal{L}$ such that $L^{np}\cap L_1^{np}$ has finite index in both $L^{np}$ and $L_1^{np}$ is said to be commensurable with $L_1^{np}$. 

A projective lattice in $V$ is an equivalence class of lattices in $V$ up to (rational or real) scaling. Let $L_1$ be the projective lattice containing $L_1^{np}$. Commensurability transposes well to projective lattices. The set $\PLa_1$ of projective lattices commensurable with $L_1$ is identified with $\psl\backslash\PGL_2^+(\mathbb{Q})$. 

There exists a symmetric function 
$$\delta:\PLa_1\times\PLa_1\rightarrow\mathbb{N}_{>0}$$ 
called hyperdistance. The right-action of $\psl$ on $\PLa_1$ preserves the hyperdistance. For any $N\in\mathbb{N}_{>0}$, we let $\PLa_1^N\subset\PLa_1$ denote the set of projective lattices $N$-hyperdistant from $L_1$, i.e. the set of projective lattices $L$ such that $\delta(L,L_1)=N$. 

The group $G=\PGL_2^+(\mathbb{Q})$ acts on the right of $\PLa_1$, and the modular group $\psl$ is naturally identified with $\mathrm{Stab}_G(L_1)$. 

As shown in Prop. \ref{representatives}, the set $\psl\backslash\PGL_2^+(\mathbb{Q})$ is identified with the set $\mathcal{M}$ of matrices of the form $\left(\begin{array}{cc} M & b \\ 0 & 1 \end{array}\right)$, for $M\in\mathbb{Q}^*_+$ and $b\in\mathbb{Q}\cap[0,1[$. Following \cite{conway}, we write $L_{M,b}$ to refer to the projective lattice commensurable with $L_1$ corresponding to the class 
$$\psl\cdot\left(\begin{array}{cc} M & b \\ 0 & 1 \end{array}\right)\in\psl\backslash\PGL_2^+(\mathbb{Q})\ .$$
When $b=0$, the label $L_{M,b}$ is shortened to $L_M$.

The Hecke congruence subgroup of level $N$ of the modular group is defined to be $\He(N)=\mathrm{Stab}_G(L_1,L_N)$.

In Appendix \ref{app} more details on this approach to arithmetic groups via their action on lattices are given.

\begin{remark*}
	Note that although $V$ is the real vector space in which we consider (projective) lattices in order to define and study the modular groups of our interest, $V$ also generically denotes the set of vertices of graphs - and we will stick to this conventional notation. What $V$ stands for in what follows is however always clear from the context, hence we hope that this unfortunate notation conflict will not be too much of a discomfort, while reading.
\end{remark*}

\newpage

\section{$\He(N)\backslash\He(1)$ as the projective line $\mathbb{P}^1(\mathbb{Z}/N\mathbb{Z})$}\label{sec1}
The goal of this section is to prove that 
\begin{equation}\label{iso}
\mathbb{P}^1(\mathbb{Z}/N\mathbb{Z})\simeq \mathrm{P}\mathcal{L}_1^N\simeq\He(N)\backslash\He(1) \ ,
\end{equation}
These bijections provide a nice framework to study $\He(N)\backslash\He(1)$: conceptually, because of the definition of $\mathrm{P}\mathcal{L}_1^N$, as well as in practice, since the homogeneous coordinates on $\mathbb{P}^1(\mathbb{Z}/N\mathbb{Z})$ give an explicit description of $\He(N)\backslash\He(1)$. 
We first construct the bijection
$\mathbb{P}^1(\mathbb{Z}/N\mathbb{Z})\leftarrow\mathrm{P}\mathcal{L}_1^N$
and then, the other one: $\mathrm{P}\mathcal{L}_1^N\simeq\He(N)\backslash\He(1)$.
From this one easily computes the index $[\He(1):\He(N)]$.

\subsection{The bijection $\mathrm{P}\mathcal{L}_1^N\rightarrow\mathbb{P}^1(\mathbb{Z}/N\mathbb{Z})$}
Let $\tilde{L}^{np}$ be the non-projective lattice commensurable with $L_1^{np}$ corresponding to some coset 
$$\SL_2(\mathbb{Z})\cdot\left(\begin{array}{cc} a & b \\ c & d \end{array}\right)\in\SL_2(\mathbb{Z})\backslash\GL_2^+(\mathbb{Q})\ .$$
It is a subgroup of $L_1^{np}$ if and only if $a,b,c,d\in\mathbb{Z}$. The index $N=[L_1^{np}:\tilde{L}^{np}]$ equals
$\det\left(\begin{array}{cc} a & b \\ c & d \end{array}\right)=N$.
The order of any element in $L_1^{np}/\tilde{L}^{np}$ divides $N$, hence
$$N\cdot L_1^{np}\leq\tilde{L}^{np}\leq L_1^{np}$$
For all $n\in\mathbb{N}_{>0}$ let
\begin{equation}
\mathrm{red}_n: \left\{\begin{array}{ccl} \tilde{L}^{np} & \rightarrow & (\mathbb{Z}/n\mathbb{Z})^2 \\ (k\ l)_\mathrm{ref} & \mapsto & (\mathrm{red}_n(k),\mathrm{red}_n(l))\end{array}\right. 
\end{equation}
be the map of reduction modulo $n$, where $(k\ l)_\mathrm{ref}$ denotes the coordinate expression of a point in $\tilde{L}^{np}$, in the reference basis. 

\begin{proposition}
The reduction modulo $N$ of the sublattice $\tilde{L}^{np}$ of $L_1^{np}$ of index $N\in\mathbb{N}_{>0}$ is a $\mathbb{Z}/N\mathbb{Z}$-submodule of $(\mathbb{Z}/N\mathbb{Z})^2$, and its cardinality is $N$.
\end{proposition}

\begin{proof}
Let $(f^1,f^2)$ be an oriented basis of $\tilde{L}^{np}$, i.e $\tilde{L}^{np}=\mathbb{Z}\cdot f^1+ \mathbb{Z}\cdot f^2$, where the coordinates of the $f^i$ ($i=1,2$) in the reference basis are $(f^i_1\ f^i_2)_{\mathrm{ref}}$. Let $P=p_1f^1+p_2f^2$ with $p_1,p_2\in\mathbb{Z}$. Then:

$$P=\left(p_1f^1_1\ +p_2f^2_1,\ p_1f^1_2\ +p_2f^2_2\right)_{\mathrm{ref}}$$

Hence $\mathrm{red}_N(P)=\left(\mathrm{red}_N(p_1f^1_1\ +p_2f^2_1),\mathrm{red}_N(\ p_1f^1_2\ +p_2f^2_2)\right)$.
Let now $P,Q\in \tilde{L}^{np}$. One readily sees that
$\mathrm{red}_N(P+Q)=\mathrm{red}_N(P)+\mathrm{red}_N(Q)$
and hence $\mathrm{red}_N(\tilde{L}^{np})$ is an abelian group. Moreover, the $\mathbb{Z}$-module structure on $\tilde{L}^{np}$ induces a $\mathbb{Z}/N\mathbb{Z}$-module structure on $\mathrm{red}_N(\tilde{L}^{np})$. The index condition implies that $\mathrm{red}_N(\tilde{L}^{np})$ has exactly $N$ elements. 
\end{proof}

\begin{definition}
The projective line $\mathbb{P}^1(\mathbb{Z}/N\mathbb{Z})$ is the set of $\mathbb{Z}/N\mathbb{Z}$-submodules of $(\mathbb{Z}/N\mathbb{Z})^2$ which are free and of rank $1$.
\end{definition}

\begin{proposition}
The relation $\sim$ on the pairs $(c,d)\in(\mathbb{Z}/N\mathbb{Z})^2$, such that $(c,d)\sim(c',d')$ if $(c',d')=l\cdot(c,d)$ for some $l\in(\mathbb{Z}/N\mathbb{Z})^\times$, is an equivalence relation. 
\end{proposition}

The projective line $\mathbb{P}^1(\mathbb{Z}/N\mathbb{Z})$ can be equivalently defined as:

$$\{(a,b)\in(\mathbb{Z}/N\mathbb{Z})^2|(\mathbb{Z}/N\mathbb{Z})\cdot a+(\mathbb{Z}/N\mathbb{Z})\cdot b =(\mathbb{Z}/N\mathbb{Z})\}/\sim\ .$$ 

Let $[c:d]$ denote the equivalence class of $(c,d)$, modulo $\sim$. Then:

$$\mathbb{P}^1(\mathbb{Z}/N\mathbb{Z})\simeq\{[a:b]\subset(\mathbb{Z}/N\mathbb{Z})^2/\sim|(\mathbb{Z}/N\mathbb{Z})\cdot a+(\mathbb{Z}/N\mathbb{Z})\cdot b =(\mathbb{Z}/N\mathbb{Z})\}\ ,$$

which makes sense since the constraint in the bracket does not depend on the choice of representatives $(a,b)$ for each class $[a:b]$. If one represents $\mathbb{Z}/N\mathbb{Z}$ as $[|0,N-1|]$, the invertibles are:

$$(\mathbb{Z}/N\mathbb{Z})^\times\simeq \{a\in[|0,N-1|] | \gcd(a,N)=1\}\ .$$

\begin{proposition} The following bijection holds:
$$\mathbb{P}^1(\mathbb{Z}/N\mathbb{Z}) \simeq \{[c:d]|c,d\in[|0,N-1|]^2, \gcd(c,d,N)=1\}\ .$$
\end{proposition}

\begin{proof}
First note that the property $\gcd(c,d,N)=1$ is well-defined modulo $N$, and because invertibles of $\mathbb{Z}/N\mathbb{Z}=[|0,N-1|]$ are the integers coprime with $N$, one sees that it is in fact well defined on the equivalence classes $[c:d]$. Now, note that if $x=\gcd(c,d,N)>1$ then $N/x$ is non-zero and satisfies $(N/x)\cdot c =(N/x)\cdot d=0$, hence the module $\mathbb{Z}/N\mathbb{Z}\cdot(c,d)$ is not free. Thus there is a map:

$$\mathbb{P}^1(\mathbb{Z}/N\mathbb{Z}) \rightarrow \{[c:d]|c,d\in[|0,N-1|]^2, \gcd(c,d,N)=1\} \ .$$

Conversely, given any representative of $[c:d]$ with $c,d\in[|0,N-1|]$ and $\gcd(c,d,N)=1$, the module $\mathbb{Z}/N\mathbb{Z}\cdot(c,d)$ is free (otherwise there would be an $a\neq0$ such that $a\cdot(c,d)=(0,0)$, which would contradict $\gcd(c,d,N)=1$). These two maps are mutually inverse, and that concludes the proof.
\end{proof}

This result is classical and can for example be found as Proposition 2.4 in \cite{Manin_1972}.

\begin{remark}\label{rem:coprime}
Let $[c:d]$ be the equivalence class of a pair $(c,d)\in[|0,N-1|]^2$ such that $\gcd(c,d)=k$ and $\gcd(k,N)=1$ i.e. $k$ is invertible. Now, $\gcd(k^{-1}c,k^{-1}d)=1$, and $[c:d]=[k^{-1}c:k^{-1}d]$. Hence 
$$\mathbb{P}^1(\mathbb{Z}/N\mathbb{Z})\simeq\{[c:d]|c,d\in[|0,N-1|],\ \gcd(c,d)=1\}$$ 
The different representatives $(c',d')\in[|0,N-1|]^2$ of an equivalence class $[c:d]$ such that $\gcd(c,d)=1$ are exactly the bases of the free module which is the point in $\mathbb{P}^1(\mathbb{Z}/N\mathbb{Z})$ under consideration. 
\end{remark}

\begin{definition}
Let $L$ be a projective lattice in $\PLa_1$, $N$-hyperdistant from $L_1$. Among all the non-projective representatives of $L$ some are sub-groups of $L_1^{np}$. Let $L^\mathrm{np}$ be the one for which the index $[L_1^\mathrm{np}:L^\mathrm{np}]$ is minimal (hence equal to $N$ - see Appendix \ref{app}).
\end{definition}

\begin{proposition}
Let $L$ be a projective lattice in $\PLa_1$, $N$-hyperdistant from $L_1$. Then $\mathrm{red}_N(L^\mathrm{np})$ is a free, rank-$1$ sub-module of $(\mathbb{Z}/N\mathbb{Z})^2$. 
\end{proposition}

\begin{proof}
We want to show that $\mathrm{red}(L^\mathrm{np})$ contains some point $(c,d)\in[|0,N-1|]^2$ with $\gcd(c,d)=1$.
As shown in Appendix \ref{app}, any projective lattice commensurable with $L_1$ is an $L_{M,b}$ for some $M\in\mathbb{Q}_+^*$ and $b\in\mathbb{Q}\cap[0,1[$.
Let $\alpha$ be the smallest strictly positive integer such that $\alpha M$ and $\alpha b$ are also integers. Hence $\gcd(\alpha M, \alpha b, \alpha)=1$, $\delta(L_{M,b},L_1)=N=\alpha^2M$, and $\alpha b\in[0,\alpha[\cap\mathbb{Z}$.
\begin{itemize}
\item If $\alpha M=1$, the point $v=(1\ \alpha b)_{\mathrm{ref}}\in L^\mathrm{np}$ works. 
\item If  $\alpha M>1$ and $\gcd(\alpha M,\alpha b)=1$, the point $v=(\alpha M\ \alpha b)_{\mathrm{ref}}\in L^\mathrm{np}$ works.
\item If $\alpha M>1$ and $\gcd(\alpha M,\alpha b)>1$, the point $v=(\alpha M\ \alpha(b+1))_{\mathrm{ref}}\in L^\mathrm{np}$ works. 
\end{itemize}
The coordinates of these $v$ are always in $[|0,N-1|]$, and coprime. By Remark \ref{rem:coprime}, the reduction modulo $N$ of the pair of the coordinates of $v$ in the reference basis is a basis of a free, rank-$1$ sub-module of $(\mathbb{Z}/N\mathbb{Z})^2$. 
\end{proof}

\begin{proposition}\label{Prop:mapinj}
The induced map $\mathrm{P}\mathcal{L}_1^N \rightarrow  \mathbb{P}^1(\mathbb{Z}/N\mathbb{Z}) $ is injective.
\end{proposition}

\begin{proof}
Consider two projective lattices $L,K\in\mathrm{P}\mathcal{L}_1^N$ mapped to the same class $[c:d]\in\mathbb{P}^1(\mathbb{Z}/N\mathbb{Z})$. The set of points in $L^\mathrm{np}$ with coordinates in $[|0,N-1|]$ coincide with the set of points of $K^\mathrm{np}$ with coordinates in $[|0,N-1|]$, hence $L^\mathrm{np}=K^\mathrm{np}$ (since they share the subgroup $N\cdot L_1^{np}$ and coincide on $L_1^{np}/(N\cdot L_1^{np})$), hence $K=L$.
\end{proof}

\begin{example}
The projective lattice $L_2$ (see Appendix \ref{app}) is $2$-hyperdistant from $L_1$, and 
$$L^\mathrm{np}_2=\mathbb{Z}\cdot(2\ 0)+\mathbb{Z}\cdot(0\ 1)$$ 
Even if $2\cdot (L^\mathrm{np}_2)$ is a sublattice of $L_1^{np}$ of index $8$, its projective class is still $L_2$. The reduction $\mathrm{red}_8(2\cdot(L^\mathrm{np}_2))$ is the following submodule of $(\mathbb{Z}/8\mathbb{Z})^2$:
\[
\{(0,0),(0,2),(0,4),(0,6),(4,0),(4,2),(4,4),(4,6)\} \ .
\]
which is obviously not free.
Figure \ref{f:egN=4} illustrates the relationship between $\mathcal{L}_1^4$ and the rank-$1$ free submodules of $\mathbb{Z}/4\mathbb{Z}$.
\end{example}

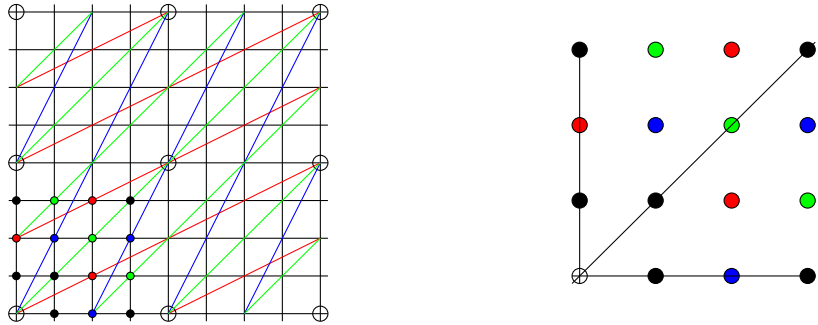
\begin{figure}[h!]
\centering
\begin{subfigure}{0.4\textwidth}
\centering
\begin{tikzpicture}
\draw(0,-0.1)--(0,4.1);
\draw(0.5,-0.1)--(0.5,4.1);
\draw(1,-0.1)--(1,4.1);
\draw(1.5,-0.1)--(1.5,4.1);
\draw(2,-0.1)--(2,4.1);
\draw(2.5,-0.1)--(2.5,4.1);
\draw(3,-0.1)--(3,4.1);
\draw(3.5,-0.1)--(3.5,4.1);
\draw(4,-0.1)--(4,4.1);
\draw(-0.1,0)--(4.1,0);
\draw(-0.1,0.5)--(4.1,0.5);
\draw(-0.1,1)--(4.1,1);
\draw(-0.1,1.5)--(4.1,1.5);
\draw(-0.1,2)--(4.1,2);
\draw(-0.1,2.5)--(4.1,2.5);
\draw(-0.1,3)--(4.1,3);
\draw(-0.1,3.5)--(4.1,3.5);
\draw(-0.1,4)--(4.1,4);
\draw[red] (0,0)--(4,2);
\draw[red] (0,1)--(4,3);
\draw[red] (0,2)--(4,4);
\draw[red] (0,3)--(2,4);
\draw[red] (2,0)--(4,1);
\draw (0,0) circle(0.1);
\draw (0,2) circle(0.1);
\draw (0,4) circle(0.1);
\draw (2,0) circle(0.1);
\draw (2,2) circle(0.1);
\draw (2,4) circle(0.1);
\draw (4,0) circle(0.1);
\draw (4,2) circle(0.1);
\draw (4,4) circle(0.1);
\draw[blue] (0,0)--(2,4);
\draw[blue] (1,0)--(3,4);
\draw[blue] (2,0)--(4,4);
\draw[blue] (3,0)--(4,2);
\draw[blue] (0,2)--(1,4);
\draw[green] (0,3)--(1,4);
\draw[green] (0,2)--(2,4);
\draw[green] (0,1)--(3,4);
\draw[green] (0,0)--(4,4);
\draw[green] (1,0)--(4,3);
\draw[green] (2,0)--(4,2);
\draw[green] (3,0)--(4,1);
\draw[fill=black] (0,0.5) circle(0.05);
\draw[fill=red] (0,1) circle(0.05);
\draw[fill=black] (0,1.5) circle(0.05);
\draw[fill=black] (0.5,0) circle(0.05);
\draw[fill=black] (0.5,0.5) circle(0.05);
\draw[fill=blue] (0.5,1) circle(0.05);
\draw[fill=green] (0.5,1.5) circle(0.05);
\draw[fill=blue] (1,0) circle(0.05);
\draw[fill=red] (1,0.5) circle(0.05);
\draw[fill=green] (1,1) circle(0.05);
\draw[fill=red] (1,1.5) circle(0.05);
\draw[fill=black] (1.5,0) circle(0.05);
\draw[fill=green] (1.5,0.5) circle(0.05);
\draw[fill=blue] (1.5,1) circle(0.05);
\draw[fill=black] (1.5,1.5) circle(0.05);
\end{tikzpicture}
\end{subfigure}
\begin{subfigure}{0.4\textwidth}
\centering
\begin{tikzpicture}
\draw (0,0) circle(0.1);
\draw[fill=black] (0,1) circle(0.1);
\draw[fill=red] (0,2) circle(0.1);
\draw[fill=black] (0,3) circle(0.1);
\draw[fill=black] (1,0) circle(0.1);
\draw[fill=black] (1,1) circle(0.1);
\draw[fill=blue] (1,2) circle(0.1);
\draw[fill=green] (1,3) circle(0.1);
\draw[fill=blue] (2,0) circle(0.1);
\draw[fill=red] (2,1) circle(0.1);
\draw[fill=green] (2,2) circle(0.1);
\draw[fill=red] (2,3) circle(0.1);
\draw[fill=black] (3,0) circle(0.1);
\draw[fill=green] (3,1) circle(0.1);
\draw[fill=blue] (3,2) circle(0.1);
\draw[fill=black] (3,3) circle(0.1);
\draw (-0.1,0)--(3.1,0);
\draw (0,-0.1)--(0,3.1);
\draw (-0.1,-0.1)--(3.1,3.1);
\end{tikzpicture}
\end{subfigure}
\caption{\sf{
The underlying black lattice on the left is $L_1^{np}$, and three of its sublattices of index $4$ are shown as the intersection points of $L_1^{np}$ with, respectively, the red, blue and green lines. On the right, one sees in the $(\mathbb{Z}/4\mathbb{Z})^2$-plane, six of its free submodules of rank $1$: the three lines in black, corresponding to the coordinates $[0:1]$, $[1:0]$ and $[1:1]$, as well as the red ([1:2]), green ([3:1]) and blue ([2:1]) lines which are the images of the corresponding sublattices on the left. Those six submodules are in fact all the free submodules of $(\mathbb{Z}/4\mathbb{Z})^2$ of rank $1$ (see next section).}
\label{f:egN=4}}
\end{figure}

\begin{proposition}
Let $c,d\in[|0,N-1|]$ be two coprime numbers. The free module corresponding to the class $[c:d]\in\mathbb{P}^1(\mathbb{Z}/N\mathbb{Z})$ defines a unique non-projective sublattice $L^\mathrm{np}_{[c:d]}$ of $L_1^{np}$ of index $N$. Its projectivisation $L_{[c:d]}$ is in $\PLa_1^N$, and the image of $L_{[c:d]}$ under the map of Proposition \ref{Prop:mapinj} is $[c:d]$.
\end{proposition}

\begin{proof}
Since $\gcd(c,d)=1$, there exist $a,b\in\mathbb{Z}$ such that $ad-bc=1$. Consider the map
$$[c:d]\rightarrow \L^\mathrm{np}_{[c:d]}=\SL_2(\mathbb{Z})\cdot\left(\begin{array}{cc} Na & Nb \\ c & d \end{array}\right)$$
The lattice $L^\mathrm{np}_{[c:d]}$ is obviously a sublattice of $L_1^{np}$ of index $N$. Now since $ad-bc=1$, the minimal $\alpha\in\mathbb{Q}_{>0}^*$ such that $\alpha a$, $\alpha b$, $\alpha c$ and $\alpha d$ are integers is $1$, hence the projectivisation $L_{[c:d]}$ of $\L^\mathrm{np}_{[c:d]}$ is $N$-hyperdistant from $L_1$. It is easy to see that the map above is the reciprocal of the one of Proposition \ref{Prop:mapinj}. 
\end{proof}

We have proved the following. 

\begin{theorem}\label{thm1}
The set $\PLa_1^N$ is in bijection with the projective line $\mathbb{P}^1(\mathbb{Z}/N\mathbb{Z})$. Moreover, if $L\in\PLa_1^N$ corresponds to some $[c:d]\in\mathbb{P}^1(\mathbb{Z}/N\mathbb{Z})$ with $c,d\in[|0,N-1|]$ coprime numbers, the class of $L$ in $\psl\backslash\PGL_2^+(\mathbb{Q})$ is:
$$\psl\cdot \left(\begin{array}{cc} Na & Nb \\ c & d \end{array}\right)$$
for some $a,b,c,d\in\mathbb{Z}$ such that $ad-bc=1$.
\end{theorem}

\begin{proposition}
Let $L$ be a projective lattice $N$-hyperdistant from $L_1$. Then
$$(L^\mathrm{np})/(N\cdot L_1^\mathrm{np})\simeq(\mathbb{Z}/N\mathbb{Z})\simeq L_1^\mathrm{np}/(L^\mathrm{np})\ .$$
\end{proposition}

\begin{proof}
We have shown that there are $a,b\in[|0,N-1|]$ coprimes, such that $(a\ b)\in L^\mathrm{np}$. Let $c,d\in\mathbb{Z}$ such that $ad-bc=1$.
\begin{enumerate}

\item Let $k\in\mathbb{Z}$ such that $k\cdot (a\ b)\in(N\cdot L_1^\mathrm{np})$. Then $N|ka$ and $N|kb$ hence $N|k(ad-bc)$ which proves 

$$(\L^\mathrm{np})/(N\cdot L_1^\mathrm{np})\simeq(\mathbb{Z}/N\mathbb{Z}) \ .$$

Now let us take $p\in L$. Then $\mathrm{red}_N(p)=m\cdot (a\ b)$ for some $m\in\mathbb{Z}/N\mathbb{Z}$ hence $L=\mathbb{Z}(a\ b)+N\cdot L_1$.
\item Consider the vector $(c\ d)\in L_1^\mathrm{np}$, and $k\in\mathbb{Z}$ such that $k\cdot (c\ d)\in L^\mathrm{np}$, that is, $k(c\ d)=k'(a\ b)+(l_1N,l_2N)$ for $k',l_1,l_2\in\mathbb{Z}$. Then $N|(kc-k'a)$ and $N|(kd-k'b)$ hence $N$ divides $-b(kc-k'a)+a(kd-k'd)=k(ad-bc)=k$, and thus:

$$(\mathbb{Z}/N\mathbb{Z})\simeq L_1^\mathrm{np}/(\L^\mathrm{np}) \ .$$

\end{enumerate}
\end{proof}

\subsection{The bijection $\mathrm{P}\mathcal{L}_1^N\simeq \He(N)\backslash\He(1)$}
\begin{proposition}
Let $N\in\mathbb{N}_{>0}$. The right-action of $\He(1)$ on $\PLa_1$ fixes the set $\PLa_1^N$. Moreover, the projective lattice $L_N\cdot M$ depends solely on the class of $M$ in $\He(N)\backslash\He(1)$.
\end{proposition}

\begin{proof}
The projective determinant is invariant under the right-action of $\psl=\He(1)$, hence $\delta(L_N\cdot M,L_1\cdot M)=\delta(L_N\cdot M,L_1)=\delta(L_N,L_1)=N$.
Let $M,M'\in\He(1)$ such that $L_N\cdot M'=L_N\cdot M$. Then

$$L_N\cdot M'M^{-1}=L_N \ ,$$

hence by definition of $\He(N)$, $M'M^{-1}\in\He(N)$, that is, $M'=AM$ with $A\in\He(N)$.

\end{proof}

The cardinality of $\mathrm{P}\mathcal{L}_1^N$ is thus an upper-bound for the index of $\He(N)$ in $\He(1)$, hence Theorem \ref{thm1} implies that $[\He(1):\He(N)]<\infty$ for all $N\in\mathbb{N}_{>0}$.
Let $\{\beta_i\}_{i\in I}$ be a set of representatives for the elements of $\He(N)\backslash\He(1)$ (one for each class). Then
$$\He(1)=\bigcup_{i\in I} \He(N)\cdot\beta_i$$

\begin{remark}\label{rem5}
Let $c,d\in[|0,N-1|]$ be two coprime numbers. Theorem \ref{thm1} shows that:
\begin{equation}
[c:d]\cdot M=[c:d]\cdot\left(\begin{array}{cc} k & l \\ m & n \end {array}\right)=\left(\begin{array}{cc} Na & Nb \\ c & d \end {array}\right)\left(\begin{array}{cc} k & l \\ m & n \end {array}\right)
\end{equation}
\begin{equation}
[c:d]\cdot M=\left(\begin{array}{cc} N(ak+bm) & N(al+bn) \\ ck+dm & cl+dn \end {array}\right)=[ck+dm:cl+dn]
\end{equation}
which yields a very explicit formula for the action of $\He(1)$ on $\mathrm{P}\mathcal{L}_1^N$.
\end{remark}

\begin{proposition}
The right-action of $\He(1)$ on $\mathrm{P}\mathcal{L}_1^N$ is transitive, and the bijection:

$$\mathrm{P}\mathcal{L}_1^N\simeq\He(N)\backslash\He(1)$$

holds. Note that the projective lattice where $[0:1]\in\mathrm{P}\mathcal{L}_1^N$ corresponds to the class $\He(N)\cdot 1$.
\end{proposition}

\begin{proof}
Since for all $c,d\in[|0,N-1|]$ such that $\gcd(c,d)=1$, there exists $a,b\in\mathbb{Z}$ such that $ad-bc=1$, and since
$[c:d]=[0:1]\cdot\left(\begin{array}{cc} a & b \\ c & d \end {array}\right)$,
the action is transitive. By definition, $\mathrm{Stab}_{\He(1)}([0:1])=\He(N)$.
\end{proof}

\subsection{Sets of representatives for $\mathbb{P}^1(\mathbb{Z}/N\mathbb{Z})$}

Let us wrap-up what we have done, and show how to assign a single element in $[|0,N-1|]^2$ to a rank-$1$ free $\mathbb{Z}/N\mathbb{Z}$-submodule of $(\mathbb{Z}/N\mathbb{Z})^2$. 

Let $a,b\in[|0,N-1|]$ such that $\gcd(a,b)=1$. To the class $[a:b]$ corresponds a rank-$1$ free $\mathbb{Z}/N\mathbb{Z}$-submodule of $(\mathbb{Z}/N\mathbb{Z})^2$, but this map is many-to-one in general. The possible bases of this module are indeed the elements of the orbit:
\begin{equation}
(\mathbb{Z}/N\mathbb{Z})^\times\cdot(a,b) \ .
\end{equation} 

One may agree on some conventions to choose one representative for each class. One way to do it is as follows.
\begin{itemize}
\item Consider the set $\tilde{D}$ of orbits of the action of $(\mathbb{Z}/N\mathbb{Z})^\times$ on $\mathbb{Z}/N\mathbb{Z}\simeq[|0,N-1|]$. Let also $D$ be the set containing the smallest element of each orbit.
\item For each $d\in D$, consider the stabiliser $G=\mathrm{Stab}_{(\mathbb{Z}/N\mathbb{Z})^\times}\{d\}$. Let $\tilde{C_d}$ be the set of orbits of the action of $G$ on the set of elements in $\mathbb{Z}/N\mathbb{Z}$ which are coprime with $d$. Let $C_d$ be the set containing the smallest element of each orbit.
\end{itemize}

For all $a,b\in[|0,N-1|]$, there is a pair $(c,d)\in C_d\times D$ in the $(\mathbb{Z}/N\mathbb{Z})^\times$-orbit of $(a,b)$. 

\begin{example}
	Let $N=6$, in which case:
	
	$$(\mathbb{Z}/6\mathbb{Z})^\times=\{1,5\}\ .$$
	
	Its orbits when acting on $\mathbb{Z}/6\mathbb{Z}$ are
	
	$$\{\{0\},\{1,5\},\{2,4\},\{3\}\} \ ,$$
	
	thus $D=\{0,1,2,3\}$, and 
	
	$$\left\{\begin{array}{l}
	\mathrm{Stab}_{(\mathbb{Z}/6\mathbb{Z})^\times}\{0\}=(\mathbb{Z}/6\mathbb{Z})^\times\\
	\mathrm{Stab}_{(\mathbb{Z}/6\mathbb{Z})^\times}\{1\}=\{1\}\\
	\mathrm{Stab}_{(\mathbb{Z}/6\mathbb{Z})^\times}\{2\}=\{1\}\\
	\mathrm{Stab}_{(\mathbb{Z}/6\mathbb{Z})^\times}\{3\}=(\mathbb{Z}/6\mathbb{Z})^\times
	\end{array}\right. \ , $$
	
	hence
	
	$$\left\{\begin{array}{l}
	C_0=\{1\}\\
	C_1=\{0,1,2,3,4,5\}\\
	C_2=\{1,3,5\}\\
	C_3=\{1,2\}\\
	\end{array}\right. \ .$$
	
	A set of representatives for $\mathbb{P}^1(\mathbb{Z}/6\mathbb{Z})$ is:
	$$\{(1,0),(0,1),(1,1),(2,1),(3,1),(4,1),(5,1),(1,2),(3,2),(5,2),(1,3),(2,3)\}\subset(\mathbb{Z}/N\mathbb{Z})^2$$
	
	From this we know that the set $\mathrm{P}\mathcal{L}_1^6$ is exactly:
	
	$$\mathrm{P}\mathcal{L}_1^6=\{L_{1/6};L_{6};L_{1/6,1/6};L_{2/3,1/3};L_{3/2,1/2};L_{2/3,2/3};L_{1/6,5/6};$$
	$$L_{1/6,1/3};L_{3/2};L_{1/6,2/3};L_{1/6,1/2};L_{2/3}\}\ ,$$
	
	where:
	$$\left.\begin{array}{ll}
	L_{1/6}=\psl\cdot\left(\begin{array}{cc} 0 & -6 \\ 1 & 0 \end{array}\right), & 
	L_{6}=\psl\cdot\left(\begin{array}{cc} 6 & 0 \\ 0 & 1 \end{array}\right), \\
	L_{1/6,1/6}=\psl\cdot\left(\begin{array}{cc} 0 & -6 \\ 1 & 1 \end{array}\right), &
	L_{2/3,1/3}=\psl\cdot\left(\begin{array}{cc} 6 & 0 \\ 2 & 1 \end{array}\right), \\
	L_{3/2,1/2}=\psl\cdot\left(\begin{array}{cc} 6 & 0 \\ 3 & 1 \end{array}\right), & 
	L_{2/3,2/3}=\psl\cdot\left(\begin{array}{cc} 6 & 0 \\ 4 & 1 \end{array}\right), \\ 
	L_{1/6,5/6}=\psl\cdot\left(\begin{array}{cc} 6 & 0 \\ 5 & 1 \end{array}\right), &
	L_{1/6,1/3}=\psl\cdot\left(\begin{array}{cc} 0 & -6 \\ 1 & 2 \end{array}\right), \\ 
	L_{3/2}=\psl\cdot\left(\begin{array}{cc} -6 & -6 \\ 3 & 2 \end{array}\right), & 
	L_{1/6,2/3}=\psl\cdot\left(\begin{array}{cc} -12 & -6 \\ 5 & 2 \end{array}\right), \\ 
	L_{1/6,1/2}=\psl\cdot\left(\begin{array}{cc} 0 & -6 \\ 1 & 3 \end{array}\right), & 
	L_{2/3}=\psl\cdot\left(\begin{array}{cc} 6 & 6 \\ 2 & 3 \end{array}\right) \ .
	\end{array}\right.$$
	
\end{example}

Of course the procedure we are following here is a pure convention, and one is free to choose any other representatives one likes more. The next proposition shows for instance a set of representatives for the elements of $\mathbb{P}^1(\mathbb{Z}/N\mathbb{Z})$ in the case $N=p^\alpha$ with $p$ prime, which does not coincide with the one one would obtain with the recipe of above.

\begin{proposition}\label{prop:palpha}
When $N=p^\alpha$ with $p$ prime and $\alpha\in\mathbb{N}_{>0}$, a set of representatives in $[|0,p^\alpha-1|]^2$ for $\mathbb{P}^1(\mathbb{Z}/p^\alpha\mathbb{Z})$ is given by
\[
\left\{\begin{array}{rcl}(a,1) & & a=0,1,2...,p^\alpha-1 \ ; \\ (1,b) & &  b=pk,\ k\in[|0,p^{\alpha-1}-1|] \ . \end{array}\right.
\]
Hence $|\mathbb{P}^1(\mathbb{Z}/p^\alpha\mathbb{Z})|=(p+1)p^{\alpha-1}$.
\end{proposition}

\begin{proof}
Let $(x,y)\in(\mathbb{Z}/(p^\alpha)\mathbb{Z})^2$. If $y$ is invertible, let $a=y^{-1}x$, and then $(x,y)\sim(a,1)$. If not, then $y=pl$ for some $l\in\mathbb{N}$ but since $\gcd(x,y)=1$, it implies that $x$ is invertible, and thus that $(x,y)\sim(1,b)$ for $b=x^{-1}y$. These representatives are never equivalent, and we have $p^{\alpha}+p^{\alpha-1}=(p+1)p^{\alpha-1}$ of them. 
\end{proof}

\begin{remark}
The relationship with the set of representatives one would have got after using the recipe explained above, is as follows. The representatives of the form $(a,1)$ for $a\in[|0,p^\alpha-1|]$ are obtained in both procedures, while for $(1,b)$ with $b$ not invertible, one can write $b=b'k^\beta$ with $b'$ invertible. Then $[1:b]=[(b')^{-1}:k^\beta]$, and $((b')^{-1},k^\beta)$ is the representative of the class $[1:b]$ that one would have got out of the first method. When $N=p^\alpha$, it turns out that the choice of representatives given in \ref{prop:palpha} is often convenient.
\end{remark}

\subsection{The Index Formula}

\begin{proposition}\label{chinese}
Let $M,N\in\mathbb{N}$ be coprime numbers. Then
$$\left.\begin{array}{ccc} \mathbb{P}^1(\mathbb{Z}/MN\mathbb{Z}) & \rightarrow & \mathbb{P}^1(\mathbb{Z}/M\mathbb{Z})\times\mathbb{P}^1(\mathbb{Z}/N\mathbb{Z}) \\
(c,d) & \mapsto & ((c_1,d_1),(c_2,d_2))
\end{array}\right.$$
is a bijection. The pair $(c,d)$ is a representative of the point $[c:d]$, and the $c_i$ or $d_i$ are the reductions of $c$ and $d$ modulo $M$ and $N$, respectively.
\end{proposition}

Proposition \ref{chinese} is none other than the Chinese remainder theorem, and from Proposition \ref{prop:palpha} we know that:

$$|\mathbb{P}^1(\mathbb{Z}/(p^\alpha)\mathbb{Z})|=(p+1)p^{\alpha-1}\ ,$$

hence the following holds.

\begin{proposition}
$$
|\mathbb{P}^1(\mathbb{Z}/N\mathbb{Z})|=N\prod_{p|N}(1+\frac{1}{p})
\ . 
$$
\end{proposition}

\begin{proof}
Write $N=\prod_{i}p_i^{\alpha_i}$. We have 
$$|\mathbb{P}^1(\mathbb{Z}/N\mathbb{Z})|=\prod_{i}|\mathbb{P}^1(\mathbb{Z}/(p_i^{\alpha_i})\mathbb{Z})|=\prod_{i}(p_i+1)p_i^{\alpha_i-1}=N\prod_{p|N}(1+\frac{1}{p})$$
\end{proof}

\section{Dessins d'enfants and analytical modular curves}\label{sec2}

The isomorphism $\psl\simeq(\mathbb{Z}/2\mathbb{Z})\star(\mathbb{Z}/3\mathbb{Z})$ induces a structure of bipartite fat graph on each of the sets $\He(N)\backslash\He(1)$. These bipartite fat graphs are called \textit{dessin d'enfants}, after Grothendieck's famous {\it Esquisse d'une Programme} \cite{groth} (see \cite{Schneps2011,Schneps1994}). 
This additional structure on $\He(N)\backslash\He(1)$ gives an efficient and purely group-theoretic definition of the cusps of the classical modular curves $X_0(N)$. The parametrisation of the $\He(N)\backslash\He(1)$ of the last section is of good use in order to understand and handle these dessins d'enfants.

\subsection{Generalities}
We begin with some rudiments and notation. For a more detailed introduction to dessins d'enfants, we refer to the book \cite{joneswolfart}.
\subsubsection{Fat graphs}

\begin{definition}
A {\bf fat graph} (or ribbon graph) is a connected simple graph $\Gamma=(V,E)$ together with a cyclic orientation of $E_v$ for every $v\in V$, where $E_v\subset E$ is the set of edges incident to $v$.
\end{definition}

\begin{remark}
If $\Gamma=(V,E)$ is a fat graph then for $e\in E$ an edge and $v\in V$ one of the two ends of $e$, it makes sense to speak of ``the edge directly after $e$ with respect to $v$". 
\end{remark}

An edge $e$ of a graph $\Gamma=(V,E)$ is oriented if one of its two ends has been chosen to be its source, and the other one, its target.

\begin{definition}
Let $\Gamma=(V,E)$ be a fat graph with $V$ and $E$ finite. A face of length $k$ is a cycle of oriented edges $(\vec e_1,...,\vec e_k)$ such that: 
\begin{itemize}
\item for all $i\in\mathbb{Z}/k\mathbb{Z}$, the source of $\vec e_{i+1}$ is the target of $\vec e_i$,
\item $e_{i+1}$ is the edge directly after $e_i$ with respect to the target of $\vec e_i$. 
\end{itemize}
\end{definition}

The faces form a partition of the set of oriented edges.

\subsubsection{Cusps and genus}

A fat graph gives rise to topological surfaces as follows. Let $\Gamma=(V,E)$ be a connected fat graph with finitely many edges and vertices. For each face $F$ of $\Gamma$, let us glue the boundary of a topological $2$-cell to $F$, following the cyclic orientation of the latter. 
The resulting topological oriented surface $\tilde S_\Gamma$ is thus obtained as a cell complex whose $1$-skeleton is $\Gamma$, and is compact and connected.
Instead of gluing copies of a disk one can also glue copies of once-punctured disks $\mathbb{D}^2\backslash\{0\}$, and that yields another topological connected surface $S_\Gamma$ which can be obtained from $\tilde S_\Gamma$ by removing one point in the interior of each of faces in $\tilde S_\Gamma$. 
These points are called {\bf cusps}. 

\begin{definition}
The genus $g(\Gamma)$ of a connected fat graph $\Gamma$ is the genus of the closed surface $\tilde S_\Gamma$. More intrinsically, $g(\Gamma)$ is half the rank of the first homology group of the chain complex
$$\xymatrix{0 \ar[r] & \mathbb{Z}^{\mathrm{faces}} \ar[r]^d & \mathbb{Z}^{E} \ar[r]^\delta & \mathbb{Z}^{V}\ar[r] & 0}$$
\end{definition}

\subsubsection{Bipartite fat graphs}

\begin{definition}
A bipartite graph is a graph $(V,E)$ for which the set $V$ is written as the disjoint union of a set $V_W$ of white vertices and a set $V_B$ of black vertices:

$$V=V_W\amalg V_B\ ,$$

and such that each edge $e\in E$ has one ``white" end and one ``black" end. A bipartite fat graph is a bipartite graph endowed with a fat structure.
\end{definition}

\begin{figure}[h!]
\centering
\begin{tikzpicture}
\draw[fill=black] (0,0) circle (0.1);
\draw (0,1) circle (0.1);
\draw (0,-1) circle (0.1);
\draw (0,0.5) circle(0.5);
\draw (0,0)--(0,-0.9);
\draw (0,-0.5) node{$2$};
\draw (-0.5,0.5) node{$1$};
\draw (0.5,0.5) node{$3$};
\draw (0.3,-0.2) node{$b$};
\draw (0,1.3) node{$w_1$};
\draw (0,-1.3) node{$w_2$};
\end{tikzpicture}

\caption{{\sf A bipartite graph drawn on a plane, whose counterclockwise orientation induces a fat structure on the graph. For example, the cyclic ordering of the edges around the black vertex is $1<2<3<1$. 
This fat graph has two faces, corresponding to the cycles $(1_{w_1\rightarrow b},2_{b\rightarrow w_2},2_{w_2\rightarrow b},3_{b\rightarrow w_1})$ (exterior face) and $(3_{w_1\rightarrow b},1_{b\rightarrow w_1})$ (inner face). 
Moreover $\tilde S_\Gamma\simeq\mathbb{S}^2$ and $S_\Gamma$ has two cusps.}\label{firstdessin}
}
\end{figure}
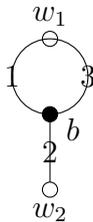

A bipartite fat structure on a graph $\Gamma=(V,E)$ singles out two permutations of its edges: a permutation $x\in\mathfrak{S}(E)$ which sends each edge to the next one with respect to its white end, and a permutation $y\in\mathfrak{S}(E)$ which sends each edge to the next one with respect to its black end.

\item Conversely, let $x$ and $y$ be two permutations of a set $E$, written as a product of disjoint circles. Let us define a bipartite fat graph $\Gamma=(V_W\amalg V_B,E)$ by setting $V_W$ to be the set of cycles in $x$, $V_B$ the set of cycles in $y$, and where an edge $e\in E$ links the only two vertices (one in $V_W$, one in $V_B$) that it ``belongs" to. This construction is easily seen to be the inverse of the one of above.

This reasoning shows that any bipartite fat graph can be drawn on an oriented surface, in such a way that the fat structure at each vertex coincides with the counterclockwise orientation of the surface. We will follow this convention in the sequel. There might be crossings among the edges, whener the genus of the surface on which the graph is frawn is smaller that the genus of the graph. Note that the group generated by $x$ and $y$ acts transitively on $E$ if and only if $\Gamma$ is connected.

\begin{remark}
The cycles of $(xy)$ (in our notations, the permutation group $\mathfrak{S}(E)$ acts on the right of $E$, hence $x$ acts first, and them $y$) are in one-to-one correspondence with the faces of $\Gamma$, as follows. 
Let $e$ be an edge of $\Gamma$. 
The two possible orientations for $e$ are denoted $e_{w\rightarrow b}$ and $e_{b\rightarrow w}$.
Then a cycle $(e_1,...,e_k)$ with $e_i\in E$ for $i\in[|1,k|]$ corresponds to the face 
$(e_{1,b\rightarrow w},f_{1,w\rightarrow b},e_{2,b\rightarrow w},f_{2,w\rightarrow b}...,e_{k,b\rightarrow w},f_{k,w\rightarrow b})$ of $\Gamma$, where each $f_i$ labels the edge directly after $e_i$ with respect to the white end of the latter.
\end{remark}

\subsubsection{Algebraic Bipartite Maps and Dessins d'Enfants}

\begin{definition}
An algebraic bipartite map (ABM) is a quadruple 

$$\bip=(G,x,y,E)\ ,$$ 

where $E$ is a set, $x,y\in\mathfrak{S}(E)$, and such that $G=\langle x,y\rangle$ acts transitively on the right of $E$. The group $G$ is called the {\bf monodromy group} (or {\bf cartographic group}) of $\bip$.
If $E$ is a finite set, then $\bip$ is called a {\bf dessin d'enfant} ({\bf dessin}, for short). 
The type of the ABM is the triple $(a,b,c)$ where $a$ (resp. $b$, $c$) is the order of $x$ (resp. $y$, $xy$).
\end{definition}

\begin{definition}
A morphism of ABMs $(G,x,y,E)\rightarrow(G',x',y',E')$ is a pair

$$\left\{\begin{array}{c}f:E\rightarrow E'\\
\phi:G\rightarrow G'\end{array}\right\} \ ,$$

where $f$ is a morphism of sets, and $\phi$ a morphism of groups satisfying $\phi(x)=x'$ and $\phi(y)=y'$, and such that for all $e\in E$ and $g\in G$:

$$f(e)\cdot\phi(g)=f(e\cdot g)\ .$$ 

A morphism of dessins is a morphism of ABMs between two dessins.
\end{definition} 

For instance, the dessin d'enfant corresponding to the bipartite fat graph shown in Figure \ref{firstdessin} is $(\mathfrak{S}_3,(13),(123),\{1,2,3\})$. 

\subsubsection{Automorphism group}
\begin{definition}
Let $\bip=(G,x,y,E)$ be an ABM. The automorphism group $\aut(\bip)$ of $\bip$ is the centraliser of $G$ in $\mathfrak{S}(E)$, that is, the group of permutations that commute with $x$ and $y$. We will consider $\aut(\bip)$ as acting on the left of $E$.
\end{definition}

\begin{example}
Consider for example the dessin d'enfant of type $(2,3,2)$ in Figure \ref{f:egAut}. First,
$$xy=(14)(26)(35)$$
and $G\simeq\mathfrak{S}_3\simeq\aut(\bip)$. The automorphism group is generated by the rotation of order $2$ around the central vertex, represented by the permutation $(14)(36)(25)$, and the rotation of order $3$ around the black vertices, represented by $(123)(654)$. The topological surface $S_\bip$ is a sphere with three cusps. Note here the difference between the monodromy group and the automorphism group. The latter is a group of symmetrie and is {\bf not} generated (while the monodromy group is) by the local cyclic order around the vertices.

\begin{figure}[h!]
\centering
\begin{tikzpicture}
\draw (0,1) circle (0.1);
\draw (0,0) circle (0.1);
\draw (0,-1) circle (0.1);
\draw[fill=black!90] (-1,0) circle (0.1);
\draw[fill=black!90] (1,0) circle (0.1);
\draw (-1,0)--(0,1) node[midway]{$3$};
\draw (-1,0)--(0,0) node[midway]{$2$};
\draw (-1,0)--(0,-1) node[midway]{$1$};
\draw (1,0)--(0,1) node[midway]{$4$};
\draw (1,0)--(0,0) node[midway]{$5$};
\draw (1,0)--(0,-1) node[midway]{$6$};
\end{tikzpicture}
{\caption{\sf 
The dessin $\bip=(G=\langle x,y\rangle,x=(16)(34)(25),y=(123)(456),[|1,6|])$}
\label{f:egAut}}
\end{figure}
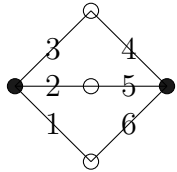
\end{example}

The action of a group $G$ on a set $E$ is said to be 
\begin{description}
\item[transitive] if $\forall e,e'\in E$,  there is at least one $g\in G$ such that $e'=e\cdot g$;
\item[semi-regular] if $\forall e,e'\in E$, there is at most one $g\in G$ such that $e'=e\cdot g$;
\item[regular] if $\forall e,e'\in E$, there is exactly one $g\in G$ such that $e'=e\cdot g$.
\end{description}

The following two results on automorphisms of ABMs correspond respectively to Theorem 2.1 and Corollary 2.1 in \cite{joneswolfart}, where the proof of these statements can be found.

\begin{proposition}
Let $\bip=(G,x,y,E)$ be an ABM. Then
\begin{enumerate}
\item $\aut(\bip)$ acts semi-regularly on $E$;
\item $\aut(\bip)$ acts regularly on $E$ if and only if $G$ does, and in that case $G\simeq \aut(\bip)$.
\end{enumerate}
\end{proposition}

In the latter case one says that the ABM is {\bf regular}.

\begin{proposition}
Let $(G,x,y,E)$ be an ABM, and $e\in E$. Let $G_e=\mathrm{Stab}_G(e)$. Then
$$\aut(\bip)\simeq N_G(G_e)/G_e \ , $$
where $N_G(G_e)$ is the normaliser of $G_e$ in $G$. 
\end{proposition}

\subsubsection{Quotient of ABMs}
Let $\bip=(G,x,y,E)$ be an ABM, and let $H<\aut(\bip)$. The left-quotient of $\bip$ by $H$ is another ABM denoted $H\backslash\bip$ together with a morphism
$$\bip\rightarrow H\backslash\bip$$
The quotient $H\backslash\bip=(G',x',y',E')$ is constructed as follows.
\begin{enumerate}
\item $E'$ is the set of equivalence classes $H\backslash E$.
\item  The permutation $x'$ of $E'$ is the one satisfying $[e]\cdot x'=[e\cdot x]$ for all $e\in E$.
\item  Similarly, $y'$ is the permutation of $E'$ such that $[e]\cdot y'=[e\cdot y]$ for all $e\in E$.
\item One sets $G'=\langle x',y'\rangle$.
\end{enumerate}

\begin{remark}
If $\bip$ is of type $(p,q,r)$ then $H\backslash\bip$ is of type $(p',q',r')$ where $p'|p,\ q'|q,\ r'|r$.
\end{remark}

\begin{example}
Consider the dessin d'enfant given in Figure \ref{f:egAut}:
$$\bip=(G=\left<x,y\right>,x=(16)(34)(25),y=(123)(456),[|1,6|])$$
Let $H=\langle(16)(35)(24)\rangle<\aut(\bip)$. The quotient $\bip\rightarrow H\backslash\bip$ is given schematically in Figure \ref{f:egQuot}, and
\[
H\backslash\bip=(\mathfrak{S}_3,(13),(123),[|1,3|]) \ .
\]
\begin{figure}[h!]
\centering
\begin{subfigure}{0.2\textwidth}
\begin{tikzpicture}
\draw (0,1) circle (0.1);
\draw (0,0) circle (0.1);
\draw (0,-1) circle (0.1);
\draw[fill=black!90] (-1,0) circle (0.1);
\draw[fill=black!90] (1,0) circle (0.1);
\draw[blue] (-1,0)--(0,1) node[midway]{$3$};
\draw[red] (-1,0)--(0,0) node[midway]{$2$};
\draw[green] (-1,0)--(0,-1) node[midway]{$1$};
\draw[green] (1,0)--(0,1) node[midway]{$4$};
\draw[red] (1,0)--(0,0) node[midway]{$5$};
\draw[blue] (1,0)--(0,-1) node[midway]{$6$};
\end{tikzpicture}
\end{subfigure}
\begin{subfigure}{0.1\textwidth}
\centering
\begin{tikzpicture}
\draw[->] (-3,0)--(1,0);
\end{tikzpicture}
\end{subfigure}
\begin{subfigure}{0.4\textwidth}
\centering
\begin{tikzpicture}
\draw[fill=black] (0,0) circle (0.1);
\draw (0,1) circle (0.1);
\draw (0,-1) circle (0.1);
\draw (0,0.5) circle(0.5);
\draw (0,0)--(0,-0.9);
\draw[red] (0,-0.5) node{$2$};
\draw[green] (-0.5,0.5) node{$1$};
\draw[blue] (0.5,0.5) node{$3$};
\end{tikzpicture}
\end{subfigure}
\caption{
{\sf An example of a quotient of dessin d'enfants}
\label{f:egQuot}}
\end{figure}
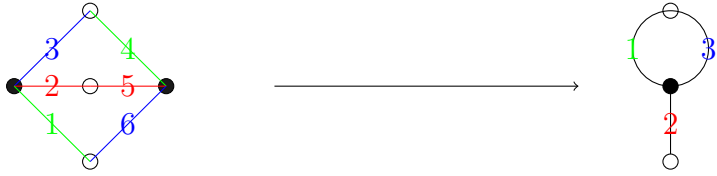
\end{example}

\subsubsection{$\He(1)$ and the universal ABM of type $(2,3,\infty)$}

Recall the standard presentation of $\psl=\He(1)$
\begin{equation}
\He(1)\simeq\psl\simeq C_2\star C_3=\langle S,U|S^2=U^3=1\rangle \ .
\end{equation}

\begin{definition}
The universal ABM of type $(2,3,\infty)$ is the ABM
\[
\bip_\infty=(\psl,S,U,\psl)
\]
\[
S:=\left[\begin{array}{cc} 0 & -1 \\ 1 & 0 \end{array}\right],\ U:=\left[\begin{array}{cc} 0 & -1 \\ 1 & -1 \end{array}\right] 
\]
where the action of the group $\psl$ on the set $\psl$ corresponds to the group right-multiplication. This ABM is regular, all its white vertices are two-valent and all its black vertices are three-valent. It is universal in the sense that any ABM of type $(2,3,c)$ for some $c\in\mathbb{N}_{>0}$ is isomorphic to a quotient $H\backslash\bip_\infty$ for some $H<\aut(\bip_\infty)=\psl$.
\end{definition}

Part of the corresponding bipartite fat graph is shown in Figure \ref{f:tri}. It is easily obtained from the universal trivalent tree by the replacement of each vertex of the tree by a black vertex, and the addition of a white vertex in the middle of each edge.

\subsubsection{Projective bases of $L_1$}

The set of all oriented projective bases that generate $L_1$ can be identified with the set of edges of $\bip_\infty$ via the map
$$\left[\begin{array}{c} f^1 \\ f^2 \end{array}\right]\rightarrow \left[\begin{array}{cc} f^1_1 & f^1_2 \\ f^2_1 & f^2_2 \end{array}\right]$$
where the $f^i_j$ are the coordinates of the vector $f^i$ in the reference basis. 

Since $\psl\simeq \langle S\rangle \star \langle U\rangle$, the set of edges of $\bip_\infty$ corresponds to the words $e, eS$, and $e\cdot S^{k}U^{l_1}SU^{l_2}...S^{k_n}U^{l_n}S^{k'}$ for integers $n\geq0$, $k,k'\in[|0,1|]$ and $l_1,...,l_n\in[|1,2|]$. Here $e$ is a conventional ``origin" associated with the identity matrix in $\psl$.

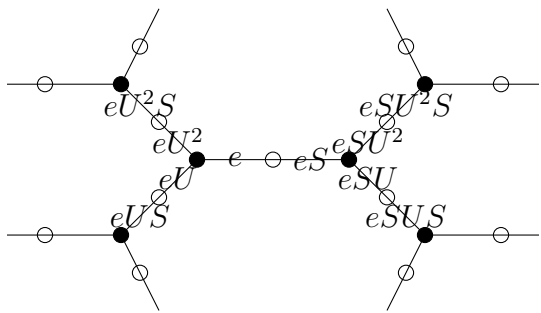
\begin{figure}[h!]
\centering
\begin{tikzpicture}[scale=1]
\draw[fill=black] (-1,0) circle (0.1);
\draw[fill=black] (1,0) circle (0.1);
\draw[fill=black] (-2,1) circle (0.1);
\draw[fill=black] (-2,-1) circle (0.1);
\draw[fill=black] (2,-1) circle (0.1);
\draw[fill=black] (2,1) circle (0.1);
\draw (-1,0)--(1,0) node[pos=0.25]{$e$} node[pos=0.75]{$eS$};
\draw (-1,0)--(-2,-1) node[pos=0.25]{$eU$} node[pos=0.75]{$eUS$};
\draw (-1,0)--(-2,1) node[pos=0.25]{$eU^2$} node[pos=0.75]{$eU^2S$};
\draw (1,0)--(2,1) node[pos=0.25]{$eSU^2$} node[pos=0.75]{$eSU^2S$};
\draw (1,0)--(2,-1) node[pos=0.25]{$eSU$} node[pos=0.75]{$eSUS$};
\draw (-2,1)--(-1.5,2);
\draw (-2,1)--(-3.5,1);
\draw (-2,-1)--(-1.5,-2);
\draw (-2,-1)--(-3.5,-1);
\draw (2,1)--(1.5,2);
\draw (2,1)--(3.5,1);
\draw (2,-1)--(1.5,-2);
\draw (2,-1)--(3.5,-1);
\draw (3,1) circle (0.1);
\draw (3,-1) circle (0.1);
\draw (-3,1) circle (0.1);
\draw (-3,-1) circle (0.1);
\draw (0,0) circle (0.1);
\draw (-1.5,0.5) circle (0.1);
\draw (-1.5,-0.5) circle (0.1);
\draw (1.5,0.5) circle (0.1);
\draw (1.5,-0.5) circle (0.1);
\draw (-1.75,1.5) circle (0.1);
\draw (-1.75,-1.5) circle (0.1);
\draw (1.75,1.5) circle (0.1);
\draw (1.75,-1.5) circle (0.1);
\end{tikzpicture}
\caption{{\sf A part of $\bip_\infty$, with reference edge $e$.}
\label{f:tri}}
\end{figure}

The map above associates $\mathrm{Id}\in\psl$ to the reference projective basis in $\mathbb{P}V$. Any other projective basis that generate $L_1\in\psl\backslash\PGL_2^+(\mathbb{Q})$ is thus identified with the corresponding word of $S$'s and $U$'s. Note that
$$\left[\begin{array}{cc} f^1_1 & f^1_2 \\ f^2_1 & f^2_2 \end{array}\right]\cdot S=\left[\begin{array}{cc} f^1_2 & -f^1_1 \\ f^2_2 & -f^2_1 \end{array}\right]$$ and
$$\left[\begin{array}{cc} f^1_1 & f^1_2 \\ f^2_1 & f^2_2 \end{array}\right]\cdot U=\left[\begin{array}{cc} f^1_2 & -(f^1_1+f^1_2) \\ f^2_2 & -(f^2_1+f^2_2) \end{array}\right]$$
This right-action of $\psl$ on $\psl\backslash\PGL_2^+(\mathbb{Q})$ describes global projective linear transformations of $V$ that preserve $L_1$. 

However, if one considers a projective lattice $N$-hyperdistant from $L_1$, it is a priori not preserved by such a projective linear transformation in $\psl$, but instead is mapped to another projective lattice $N$-hyperdistant from $L_1$ (since the right action of $\psl$ preserves the hyperdistance). The dessins d'enfants which correspond to the Hecke groups $\He(N)$ contains this data quite efficiently.

\subsection{Definition of the dessins $\bip_{0,N}$}\label{subbon}

Let $N\in\mathbb{N}_{>0}$. Since $\He(N)<\psl=\aut(\bip_\infty)$, there is a quotient dessin:
\begin{equation}
\bip_{0,N}=\He(N)\backslash\bip_\infty=\left(G_{0,N},x_{0,N},y_{0,N},E_{0,N}=\He(N)\backslash\He(1)\right)\ .
\end{equation}
We will soon see that if $N\geq2$ those dessins are of type $(a,b,c)$ with:
\begin{equation}
a=2, \ b=3, \ c=N \ .
\end{equation} 
Of course the case $N=1$ corresponds to the trivial dessin with $E$ a singleton. 

Let $X_0(N)=\tilde{S}_{\bip_{0,N}}$ (resp. $Y_0(N)=S_{\bip_{0,N}}$) be the closed topological surface (resp. the topological surface with cusps) associated with $\bip_{0,N}$. The groups $\He(N)$ inherit a genus and a set of cusps from their corresponding dessin.

\subsubsection{Canonical morphisms}

Let $N,d\in\mathbb{N}_{>0}$ with $d$ dividing $N$. Since $\He(N)\leq\He(d)\leq\He(1)$ one has the following.

\begin{proposition}\label{prop:Can}
There is a canonically defined morphism
$$(f,\phi)_{N,d}:\bip_{0,N}\rightarrow\bip_{0,d}.$$
\end{proposition}

\begin{proof} Since $\He(N)\leq\He(d)$ are subgroups of finite index in $\He(1)$, the group $\He(N)$ has also finite index in $\He(d)$. Let $I_{N,d}=[\He(d):\He(N)]$. One has:
$$\He(1)=\coprod_{j=1}^{I_{d,1}} \He(N)\cdot\beta_j\ ,$$
$$\He(d)=\coprod_{i=1}^{I_{N,d}} \He(N)\cdot\alpha_i\ .$$
This yields 
$$\He(1)=\coprod_{i,j} \He(N)\cdot (\alpha_i\beta_j)\ ,$$
hence $$\He(N)\backslash\He(1)\simeq\He(d)\backslash\He(1)\times\He(N)\backslash\He(d)\ .$$

Let $f$ be the projection $\He(N)\backslash\He(1)\rightarrow\He(d)\backslash\He(1)$, and

$$\phi:\langle x_{0,N},y_{0,N}\rangle\leq\mathfrak{S}(E_{0,N})\rightarrow\langle x_{0,d},y_{0,d}\rangle\leq\mathfrak{S}(E_{0,d})$$

the group morphism with domain the group generated by $(x_{0,N},y_{0,N})$ and target the group generated by $(x_{0,d},y_{0,d})$. It is defined by $\phi(x_{0,N})=x_{0,d}$ and $\phi(y_{0,N})=y_{0,d}$.

Let $e\in\psl=E(\bip_\infty)$. By definition of the quotient, $(\He(N)\cdot e)\cdot x_{0,N}=\He(N)\cdot (e\cdot x)$, and $(\He(d)\cdot e)\cdot x_{0,d}=\He(d)\cdot (e\cdot x)$ where $x=x_{0,1}$. Subsequently:
$$f((\He(N)\cdot e)\cdot x_{0,N})=f(\He(N)\cdot (e\cdot x))=\He(d)\cdot (e\cdot x)=(\He(d)\cdot e)\cdot x_{0,d}\ .$$
The same reasoning holds for the y's hence
$$(f,\phi):\bip_{0,N}\rightarrow\bip_{0,d}$$
is a morphism of dessin d'enfants.
\end{proof}

\begin{example}
The morphism $(f,\phi)_{6,2}:\bip_{0,6}\rightarrow\bip_{0,2}$ satisfies $f^{-1}_{6,2}(\{1\})=\{1,7,6,12\}$, $f^{-1}_{6,2}(\{2\})=\{2,4,8,10\}$ and $f^{-1}_{6,2}(\{3\})=\{3,5,9,11\}$. It is pictured in Figure \ref{f:egMorph}.
\begin{figure}[h!]
\begin{subfigure}{.4\textwidth}
\centering
\begin{tikzpicture}
\draw[fill=black] (0,0) circle (0.1);
\draw[fill=black] (0,-2) circle (0.1);
\draw (0,1) circle (0.1);
\draw (-0.5,-2.5) circle (0.1);
\draw (0.5,-2.5) circle (0.1);
\draw (0,-1) circle (0.1);
\draw (0,-3) circle (0.1);
\draw (0,-4) circle (0.1);
\draw[fill=black] (1,-3) circle (0.1);
\draw[fill=black] (-1,-3) circle (0.1);
\draw[red] (0,1) arc(90:270:0.5);
\draw[green] (0,1) arc(90:-90:0.5);
\draw[green] (0,-2)--(-0.5,-2.5);
\draw[red] (-0.5,-2.5)--(-1,-3);
\draw[red] (0,-2)--(0.5,-2.5);
\draw[green] (0.5,-2.5)--(1,-3);
\draw[blue] (0,0)--(0,-0.9);
\draw[blue] (0,-1.1)--(0,-2);
\draw[green] (-1,-3)--(0,-3);
\draw[red] (1,-3)--(0,-3);
\draw[blue] (-1,-3) arc(180:360:1);
\draw (-0.5,0.5) node{$1$};
\draw (0.5,0.5) node{$3$};
\draw (0,-0.5) node{$2$};
\draw (0,-1.5) node{$4$};
\draw (-0.25,-2.25) node{$5$};
\draw (0.25,-2.25) node{$6$};
\draw (-0.75,-2.75) node{$7$};
\draw (-0.75,-3.7) node{$8$};
\draw (0.75,-3.7) node{$10$};
\draw (-0.5,-3) node{$9$};
\draw (0.5,-3) node{$12$};
\draw (0.75,-2.75) node{$11$};
\end{tikzpicture}
\end{subfigure}
\begin{subfigure}{.2\textwidth}
\centering
\begin{tikzpicture}
\draw (-2,0)--(2,0);
\draw (1.9,0.1)--(2,0);
\draw (1.9,-0.1)--(2,0);
\draw (0,0.5) node{$(f,\phi)_{6,2}$};
\end{tikzpicture}
\end{subfigure}
\begin{subfigure}{.4\textwidth}
\centering
\begin{tikzpicture}
\draw[fill=black] (0,0) circle (0.1);
\draw (0,1) circle (0.1);
\draw (0,-1) circle (0.1);
\draw[red] (0,1) arc(90:270:0.5);
\draw[green] (0,1) arc(90:-90:0.5);
\draw[blue] (0,0)--(0,-0.9);
\draw (-0.5,0.5) node{$1$};
\draw (0.5,0.5) node{$3$};
\draw (0,-0.5) node{$2$};
\end{tikzpicture}
\end{subfigure}
\caption{
{\sf The canonical morphism $(f,\phi)_{6,2}:\bip_{0,6}\rightarrow\bip_{0,2}$, illustrating Proposition \ref{prop:Can}}.
\label{f:egMorph}
}
\end{figure}
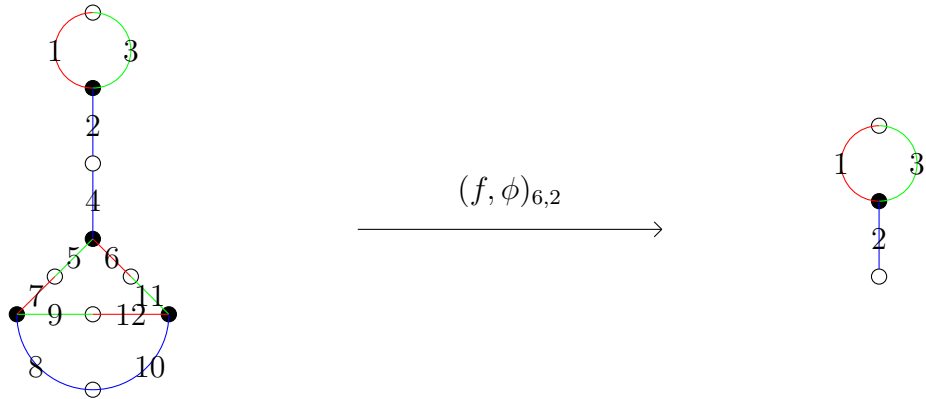
\end{example}

\subsubsection{Naming the edges}
Theorem \ref{thm1} implies that one can choose representatives of the elements of $\He(N)\backslash\He(1)$ as a set of pairs of coprime numbers in $[|0,N-1|]$.

When $N=p^{\alpha}$ with $p$ a prime number and $\alpha\in\mathbb{N}_{>0}$, we have seen that
$$\{(a,1),(1,b)|\ a\in\mathbb{Z}/p^\alpha\mathbb{Z},\ b\in p\mathbb{Z}/p^\alpha\mathbb{Z}\}$$
conveniently represents the points of $\mathbb{P}^1(\mathbb{Z}/p^\alpha\mathbb{Z})$. 

As already emphasized, there is in general no natural choice of representatives. However, it is easy to construct such a set of representatives, since Remark \ref{rem5} implies that:
$$[c:d]\cdot x_{0,N}=[d:-c]$$
$$[c:d]\cdot y_{0,N}=[d:-(c+d)]$$
Hence one can built $\bip_{0,N}$ edge by edge, in a very hands-on way. 

\begin{example}
Let us draw $\bip_{0,11}$. We could use the special set of representatives listed above since $11$ is prime, however, we will construct the dessin directly to illustrate the general case. Let us start with the projective lattice $L_{11}$ (which corresponds to $[0:1]$), and compute:

$$\left.\begin{array}{ccc} [0:1]x_{0,11}=[1:0] & [0:1]y_{0,11}=[1:-1]=[10:1] & [0:1]y^2_{0,11}=[1:0] \\
\ [10:1]x_{0,11}=[1:1] & [1:1]y_{0,11}=[1:-2]=[5:1] & [1:1]y^2_{0,11}=[1:-6]=[9:1] \\
\ [5:1]x_{0,11}=[2:1] & [2:1]y_{0,11}=[7:1] & [2:1]y^2_{0,11}=[4:1] \\
\ [9:1]x_{0,11}=[6:1] & [6:1]y_{0,11}=[3:1] & [6:1]y^2_{0,11}=[8:1] \end{array}\right.$$

This is enough to completely determine $\bip_{0,11}$: it has two faces, corresponding to the cycles $([1:0])$ and $([a:1])_{a\in[|0,10|]}$, and its genus is $1$. The corresponding bipartite fat graph is given in Figure \ref{f:egEdges}.

\begin{figure}[h!]
\centering
\begin{subfigure}{0.5\textwidth}
\begin{tikzpicture}
\draw[fill=black] (0,0) circle (0.1);
\draw[fill=black] (0,-2) circle (0.1);
\draw (0,1) circle (0.1);
\draw (0,-1) circle (0.1);
\draw[fill=black] (1,-3) circle (0.1);
\draw[fill=black] (-1,-3) circle (0.1);
\draw (0,0.5) circle(0.5);
\draw (-0.55,0.5) node{$[0:1]$};
\draw (0.5,-2.5) circle(0.1);
\draw (-0.5,-2.5) circle(0.1);
\draw (0,-2)--(-0.94,-2.94);
\draw (0,-2)--(0.94,-2.94);
\draw (0,0)--(0,-0.9);
\draw (0,-1.1)--(0,-2);
\draw (-1,-3)--(0,-4);
\draw (1,-3)--(0.1,-3.9);
\draw (0,-4) arc (-135:45:0.71);
\draw (-0.1,-4.1) arc (-50:-230:0.71);
\draw (-1.1,-4) circle (0.1);
\draw (1,-4) circle (0.1);

\draw (1.8,-1.95)--(2.2,-1.95);
\draw (1.8,-2.05)--(2.2,-2.05);

\draw[fill=black] (4,0) circle (0.1);
\draw[fill=black] (4,-2) circle (0.1);
\draw (4,1) circle (0.1);
\draw (4,-1) circle (0.1);
\draw[fill=black] (5,-3) circle (0.1);
\draw[fill=black] (3,-3) circle (0.1);
\draw (4,0.5) circle(0.5);
\draw (3.45,0.5) node{$[0:1]$};
\draw (4.5,-2.5) circle(0.1);
\draw (3.5,-2.5) circle(0.1);
\draw (4,-2)--(3.06,-2.94);
\draw (4,-2)--(4.94,-2.94);
\draw (4,0)--(4,-0.9);
\draw (4,-1.1)--(4,-2);
\draw (5,-3)--(4,-4);
\draw (3,-3)--(3.9,-3.9);
\draw (4,-4) arc (-45:-225:0.71);
\draw (4.1,-4.1) arc (-130:45:0.71);
\draw (3,-4) circle (0.1);
\draw (5.1,-4) circle (0.1);

\end{tikzpicture}
\end{subfigure}
\begin{subfigure}{0.3\textwidth}
\begin{tikzpicture}
\draw[dashed] (0,0)--(4,0);
\draw[dashed] (4,0)--(4,4);
\draw[dashed] (4,4)--(0,4);
\draw[dashed] (0,4)--(0,0);
\draw[fill=black] (1,1) circle (0.1);
\draw[fill=black] (2,2) circle (0.1);
\draw[fill=black] (3,3) circle (0.1);
\draw[fill=black] (1,3) circle (0.1);
\draw (0,2) circle (0.1);
\draw (4,2) circle (0.1);
\draw (2,0) circle (0.1);
\draw (2,4) circle (0.1);
\draw (1.5,1.5) circle (0.1);
\draw (2.5,2.5) circle (0.1);
\draw (0.5,3.5) circle (0.1);
\draw (1.5,2.5) circle (0.1);
\draw (0,2)--(2,0);
\draw (4,2)--(2,4);
\draw (1,1)--(3,3);
\draw (2,2)--(1,3);
\draw (0.75,3.25) circle(0.35);
\end{tikzpicture}
\end{subfigure}
\caption{
{\sf The bipartite fat graph corresponding to $\bip_{0,11}$.
}
\label{f:egEdges}}
\end{figure}
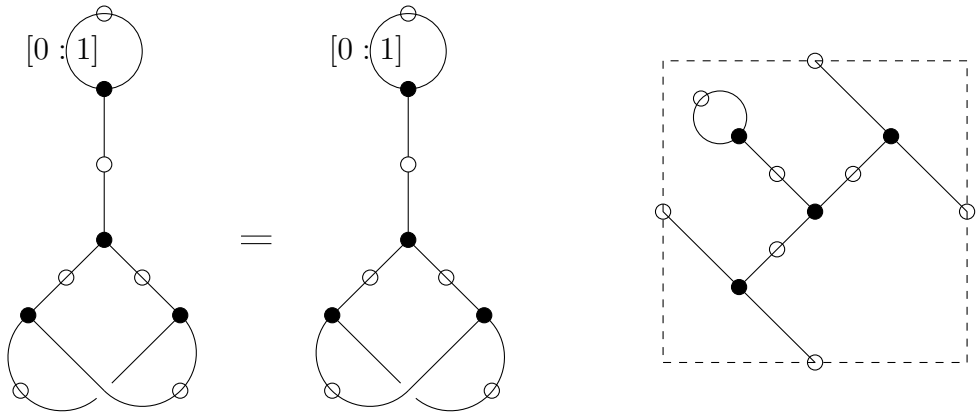
\end{example}


\subsubsection{Interpretation of the Hecke dessins in terms of lattices} 
We know that the set of edges of the universal bipartite map $\bip_\infty$ of type $(2,3,\infty)$ is the set of projective bases for the projective lattice $L_1$. Choose a projective lattice $N$-hyperdistant from $L_1$, say, $L_N$. It corresponds to the following coset in $\psl\backslash\PGL_2^+(\mathbb{Q})$:
$$\psl\cdot\left(\begin{array}{cc} N & 0 \\ 0 & 1 \end{array}\right)\ .$$
Under a projective linear transformation of the vector space $V$ preserving $L_1$ (i.e., under the right multiplication by a matrix in $\psl$), $L_N$ is mapped to a projective lattice $N$-hyperdistant from $L_1$, which is a priori different from $L_N$.
 
Since any matrix in $\psl$ can be written as a product of $S$'s and $U$'s, the dessin d'enfant corresponding to $\He(N)$ describes how these ``elementary'' projective transformations act on the set $\La_1^N$:
\begin{itemize}
\item there is a bijection between the set of edges in $\bip_{0,N}$ and the set $\La_1^N$,
\item if one right-multiplies the class corresponding to a projective lattice by S (resp. U), one obtains the class corresponding to the projective lattice associated with the edge directly after the one we started with, with respect to the white (resp. black) end of the latter.
\end{itemize}
In the next subsection we study the cusps of the $\bip_{0,N}$, i.e the cycles of the permutation $y_{0,N}x_{0,N}$. In terms of projective lattices, a cusp is a cycle for the projective transformation $US$ acting on $\La_1^N$.


\subsection{Torsion points, cusps and genus of the $\bip_{0,N}$}\label{subcusps}

\subsubsection{Torsion points of order $2$}

\begin{definition}
	The torsion points of order $2$ in $\bip_{0,N}$ are the one-valent white vertices of $\bip_{0,N}$. 
\end{definition}

Let $c,d\in[|0,N-1|]$ be coprimes, and such that $[c:d]$ corresponds to the edge terminating at such a torsion point of order $2$. Since the latter is one-valent, we know that: 

$$[c:d]\cdot x_{0,N}=[d:-c]=[c:d]\ ,$$

hence there exists $k\in[|0,N-1|]$ satisfying $\gcd(k,N)=1$, and such that $c=kd$ and $d=-kc$ in $\mathbb{Z}/N\mathbb{Z}$. This implies $-(c,d)=k^2(c,d)$, and $k^2=-1$ (using Bezout's identity). Therefore, if $-1$ is not a quadratic residue modulo $N$, there cannot be any white vertex of valency one in $\bip_{0,N}$. One can refine this analysis into an actual counting of the number of torsion points of order $2$ in $\bip_{0,N}$, as follows.

\paragraph{The case $N=p^\alpha$ with $p>2$}
Consider the case $N=p^\alpha$ where $p>2$ is a prime number, and $\alpha\in\mathbb{N}_{>0}$. In that case, the representatives $(c,d)\in[|0,N-1|]^2$ of the points of $\mathbb{P}^1(\mathbb{Z}/N\mathbb{Z})$ have at least one coordinate which is coprime with $p^\alpha$, since $\gcd(c,d,p^\alpha)=1$, hence the representative of any edge can be chosen of the form $(c,1)$, with $c\in\mathbb{Z}/p^\alpha\mathbb{Z}$, as already explained above.

Let us assume that the white end of the edge $[c:1]$ is of one-valent. Right-multiplication by $x_{0,p^\alpha}$ yields $[1:-c]$, which has to be the same point as $[c:1]$, because of the assumption on the valency of the white end. Hence $c^2=-1$ in $\mathbb{Z}/p^\alpha\mathbb{Z}$, which implies that the order of $c$in the group $(\mathbb{Z}/p^\alpha\mathbb{Z})^\times$ is $4$.

It is a classical result that:
$$(\mathbb{Z}/p^\alpha\mathbb{Z})^\times\simeq(\mathbb{Z}/(p-1)p^{\alpha-1}\mathbb{Z})\ ,$$
 though not canonically. Anyways, since this group is cyclic, the equation $x^2=-1$ has exactly two solutions if and only if $4|(p-1)p^{\alpha-1}$, that is, if and only if $p\equiv 1 [4]$.

\paragraph{The case $N=2^\alpha$}
Now consider the case $p=2$, and $\alpha\in\mathbb{N}_{>0}$. 
\begin{itemize}
\item If $\alpha=1$, the group $(\mathbb{Z}/2\mathbb{Z})^\times$ is trivial. Hence the equation $x^2=-1=1$ has $x=1$ as unique solution. 
\item Assume nom that $\alpha>1$. The invertibles in $\mathbb{Z}/2^\alpha\mathbb{Z}\simeq[|0,2^\alpha-1|]$ are the odd numbers. A square root of $-1$ hence corresponds to a solution of the equation
$$(2y+1)^2=l2^\alpha-1\ ,$$
for some $y,l\in\mathbb{Z}$. This is equivalent to $4k^2+4k+1=l2^{\alpha}-1$, and hence to $2k^2+2k=l2^{\alpha-1}-1$. This equation has no solution in $\mathbb{Z}$ since we assumed that $\alpha>1$.
\end{itemize}

\paragraph{General $N$}
Let $N\in\mathbb{N}_{>0}$, and decompose $N$ in prime factors: $N=\prod_i p_i^{\alpha_i}$. The Chinese remainder theorem states that 
$$\mathbb{Z}/N\mathbb{Z}\simeq\prod_i\mathbb{Z}/p_i^{\alpha_i}\mathbb{Z}\ ,$$
hence $x\in(\mathbb{Z}/N\mathbb{Z})$ satisfies $x^2=-1$ if and only if $(\mathrm{red}_{p_i^{\alpha_i}}(x))^2=-1$ in $\mathbb{Z}/p_i^{\alpha_i}\mathbb{Z}$ for all $i$. Conversely, remember that for coprimes $M$ and $N$, one has $E_{0,MN}=E_{0,M}\times E_{0,N}$, hence any tuple $(x_i)$ such that for all $i$, $x_i^2\equiv-1$ mod. $p_i^{\alpha_i}$, corresponds to a solution of the equation $x^2=-1$ in $(\mathbb{Z}/N\mathbb{Z})$. We have then proved the following:

\begin{proposition}
Let $N\in\mathbb{N}_{>0}$, and decompose it in prime factors:
 
$$N=2^a\times\prod_{i=1}^n p_i^{\alpha_i},\ \forall i\in[|1,n|],\ \alpha_i>0\ .$$ 

Then $\bip_{0,N}$ has torsion points of order $2$ if and only if $a\leq1$ and for all $i\in[|1,n|]$, $p_i\equiv1$ modulo 4, and $\alpha_i=1$. In that case there are exactly $2^n$ different solutions to the equation $x^2=-1$ in $(\mathbb{Z}/N\mathbb{Z})$, or equivalently, $\bip_{0,N}$ has exactly $2^n$ torsion points of order $2$. 
\end{proposition}

\subsubsection{Torsion points of order $3$}

\begin{definition}
	The torsion points of order $3$ in $\bip_{0,N}$ are the one-valent black vertices of $\bip_{0,N}$. 
\end{definition}

Let $c,d\in[|0,N-1|]$ coprimes such that $[c:d]$ is the edge terminating at such a torsion point of order $3$. Since the latter is a one-valent vertex, we know that:

$$[c:d]\cdot y_{0,N}=[d:-(c+d)]=[c:d]$$

hence there exists $k\in[|0,N-1|]$ satisfying $\gcd(k,N)=1$, and such that $c=kd$ and $d=-k(c+d)$ in $(\mathbb{Z}/N\mathbb{Z})$. This implies that $(k^2+k+1)c=0$, together with $(k^2+k+1)d=0$, and again thanks to Bezout's identity: $k^2+k+1=0$.
Multiplying both sides of $k^2+k+1=0$ by $k$ yields $k^3=1$, but in general, $k^3=1$ does not imply $k^2+k+1=0$ in $\mathbb{Z}/N\mathbb{Z}$. However, we're going to be looking at the third roots of $1$, and among them, which ones are solutions of $k^2+k+1=0$. \\

Let $N\in\mathbb{N}_{>0}$ be a power of a prime: $N=p^\alpha$. There cannot be any solution of the equation $k^2+k+1=0$ in $\mathbb{Z}/p^\alpha\mathbb{Z}$ if this ring does not admit any third roots of $1$, and because we know the cyclic structure of the group of invertibles in $\mathbb{Z}/p^\alpha\mathbb{Z}$, we can conclude that the prime $p$ has to be either $3$ or congruent to $1$ modulo $3$. Let $k$ be such that $k^3=1$, and set $a=k^2+k+1\in\mathbb{Z}/p^\alpha\mathbb{Z}$. Multiplying both sides with $k$ yields $ka=k^3+(a-1)=a$, hence $(k-1)a=0$.

\paragraph{The case $N=p^{\alpha}$ for $p>3$}

Let $p>3$ be a prime, and let $N=p^\alpha$, with $\alpha\geq1$. 
Since $N>3$, $k=1$ is not a solution of $k^2+k+1=0$, hence one must consider the other third roots of $1$, if any. Suppose that $p\equiv1\ \mathrm{mod.}\ 3$. Then, from the structure of the group $(\mathbb{Z}/p^\alpha\mathbb{Z})^\times$, we know that there are two third roots of $1$ which are not $1$. Let $k$ be such a root. Then:
\begin{itemize}
\item Either $(k-1)$ is invertible, in which case $a=k^2+k+1$ has to be zero, since $(k-1)a=0$.
\item Otherwise, $(k-1)$ is not invertible, i.e. $k=pk+1$. Then $a=k^2+k+1=p^2k^2+3pk+3$. We assumed that $p>3$, thus $a$ is invertible, and subsequently $k=1$, which contradicts our initial hypothesis.
\end{itemize}
Hence the non-trivial third roots of $1$ satisfy $k^2+k+1=0$.

\paragraph{The case $N=3^{\alpha}$}

For $\alpha=1$ the trivial case $k=1$ is the only solution of $k^2+k+1=0$, and we assume now that $\alpha>1$. 

One can check that $k_1=(1+3^{\alpha-1})$ and $k_2=(1-3^{\alpha-1})$ are the two non-trivial third roots of $1$, and that $k_1^2+k_1+1=k_2^2+k_2+1=3$. Hence if $\alpha>1$ the equation $k^2+k+1=0$ has no solution on $\mathbb{Z}/3^\alpha\mathbb{Z}$.

\paragraph{General $N$}
Eventually, consider any $N\in\mathbb{N}_{>0}$, and decompose $N$ in prime factors: $N=\prod_i p_i^{\alpha_i}$. The Chinese remainder theorem states that 
$$\mathbb{Z}/N\mathbb{Z}\simeq\prod_i\mathbb{Z}/p_i^{\alpha_i}\mathbb{Z}\ ,$$
hence $x\in\mathbb{Z}/N\mathbb{Z}$ satisfies $x^2+x+1=0$ if and only if $\mathrm{red}_{p_i^{\alpha_i}}(x)$ satisfies this equation in $\mathbb{Z}/p_i^{\alpha_i}\mathbb{Z}$ for all $i$. Conversely, remember that for coprimes $M$ and $N$, one has $E_{0,MN}=E_{0,M}\times E_{0,N}$ hence any tuple $(x_i)$ such that for all $i$, $x_i^2+x_i+1=0$ mod. $p_i^{\alpha_i}$, corresponds to a solution in $(\mathbb{Z}/N\mathbb{Z})$. Hence one has the following

\begin{proposition}
Let $N\in\mathbb{N}_{>0}$, and decompose it in prime factors 
$$N=3^a\times\prod_{i=1}^n p_i^{\alpha_i},\ \forall i\in[|1,n|],\ \alpha_i>0\ .$$ 
Then $\bip_{0,N}$ has torsion points of order $3$ if and only if $a\leq1$ and for all $i\in[|1,n|]$, $p_i\equiv1$ modulo 3. In that case there are exactly $2^n$ different solutions in $\mathbb{Z}/N\mathbb{Z}$ to the equation $x^2+x+1=0$, and equivalently, $\bip_{0,N}$ has exactly $2^n$ torsion points of order $3$. 
\end{proposition}

\subsubsection{Description of the cusps and their width}

Recall that the cusps of $\bip_{0,N}$ are the cycles of the permutation $y_{0,N}x_{0,N}$. Let $\mathcal{C}(N)$ be the set of cusps in $\bip_{0,N}$.

Note that
$y_{0,1}x_{0,1}=\left[\begin{array}{cc} 1 & 0 \\ 1 & 1 \end{array}\right]$,
which implies
$[c:d]\cdot(y_{0,N}x_{0,N})=[c+d:d]$.
The case $N=8$ is partly studied as an example in Example \ref{f:B08} below.

\begin{example}\label{f:B08}
Let $N=8$, and choose the set of representatives $\{(c,1),\ c\in\mathbb{Z}/8\mathbb{Z}\}\cup\{(1,c),\ c\in2\mathbb{Z}/8\mathbb{Z}\}$ for the homogeneous coordinates on $\mathbb{P}^1(\mathbb{Z}/8\mathbb{Z})$. Consider the projective lattice $8$-hyperdistant from $L_1$ corresponding to the homogeneous coordinate $[1:6]$. The right-action of $US$ yields the projective lattice corresponding to 
$$[6+1:6]=[-1:6]=[1:-6]=[1:2] \ ,$$ 
and $[1:2]\cdot US=[1:6]$. 
Hence the ``central'' cusp in $\He(8)$ is the cycle of projective lattices corresponding to $([1:6],[1:2])$. One can compute that they are the projective lattices $(L_{1/8,3/4},L_{1/8,1/4})$.

\begin{figure}[h!]
\centering
\begin{tikzpicture}[scale=0.7]
\draw[fill=black] (0,-2) circle (0.1);
\draw[fill=black] (0,-4) circle (0.1);
\draw[fill=black] (0,-8) circle (0.1);
\draw[fill=black] (0,-6) circle (0.1);
\draw (0,-10) circle (0.1);
\draw (0,0) circle (0.1);
\draw (0,-3) circle (0.1);
\draw (1,-5) circle (0.1);
\draw (-1,-5) circle (0.1);
\draw (0,-7) circle (0.1);
\draw (0,-10) circle (0.1);
\draw (0,-1) circle(1);
\draw (0,-5) circle(1);
\draw (0,-9) circle(1);
\draw (0,-2)--(0,-4);
\draw (0,-6)--(0,-8);
\draw (-1.1,-1) node{$[0:1]$};
\draw (0.1,-3.5) node{$[1:1]$};
\draw (-1.1,-5.5) node{$[2:1]$};
\draw (0.1,-7.5) node{$[3:1]$};
\draw (1.1,-9) node{$[4:1]$};
\draw (0.1,-6.5) node{$[5:1]$};
\draw (1.1,-4.5) node{$[6:1]$};
\draw (0.1,-2.5) node{$[7:1]$};
\draw (1.1,-1) node{$[1:0]$};
\draw (-1.1,-4.5) node{$[1:6]$};
\draw (-1.1,-9) node{$[1:4]$};
\draw (1.1,-5.5) node{$[1:2]$};
\end{tikzpicture}
\caption{{\sf The dessin d'enfant $\bip_{0,8}$.}
}
\end{figure}
\end{example}

\begin{definition}
The width function
$$w:\mathcal{C}(N)\rightarrow\mathbb{N}$$
associates to each cusp $c\in\mathcal{C}(N)$ the length of the corresponding cycle in the decomposition of $y_{0,N}x_{0,N}$ in disjoint cycles. 
\end{definition}

\begin{proposition}\label{prop:sumW}
$$\sum_{c\in\mathcal{C}(N)} w(c)=|\mathbb{P}^1(\mathbb{Z}/N\mathbb{Z})|=N\prod_{p|N}(1+\frac{1}{p}) \ , $$
where the last product runs over the prime numbers dividing $N$. 
\end{proposition}

\begin{proof}
The sum of the length of all the cycles in the decomposition of the permutation $y_{0,N}x_{0,N}$ is the cardinality of $E_{0,N}$ which has already been shown to be $|\mathbb{P}^1(\mathbb{Z}/N\mathbb{Z})|$.
\end{proof}

\begin{definition}\label{prop:w2}
Let $N>1$ be an integer. Then $\bip_{0,N}$ has two special cusps denoted $c_\infty$ and $c_0$, with width $w(c_\infty)=1$ and $w(c_0)=N$. 
\end{definition}

\begin{proof}
The cusp $c_\infty$ is the singleton $\{[1:0]\}$, which is a cusp of width $1$ since
$$[1:0]\cdot \left[\begin{array}{cc} 1 & 0 \\ 1 & 1 \end{array}\right]=[1:0] \ .$$

Let now $c_0$ be the cusp defined as the one containing the edge $[0:1]$. Since
$$[0:1]\cdot \left[\begin{array}{cc} 1 & 0 \\ 1 & 1 \end{array}\right]^k=[k:1] \ ,$$
and since $[N:1]=[0:1]$, the cusp $c_0$ is the cycle $([0:1],[1:1],...,[N-1:1])$ and has width $N$.
\end{proof}

Proposition \ref{prop:sumW} and definition \ref{prop:w2} imply 

\begin{corollary}
Let $N=p$ a prime number. Then $\mathcal{C}(p)=\{c_0,c_\infty\}$.
\end{corollary}

\begin{proposition}\label{prop:desccusps}
Let $N=p^\alpha$ with $p$ prime. Then:
$$|\mathcal{C}(p^\alpha)|=\sum_{0\leq k\leq\alpha} \phi(\gcd(p^k,p^{\alpha-k}))$$
where $\phi$ is Euler's totient function. Moreover, there is an explicit description of $\mathcal{C}(p^\alpha)$.
\end{proposition}

\begin{proof}
The points of $\mathbb{P}^1(\mathbb{Z}/p^\alpha\mathbb{Z})$ which are neither $[1:0]$ or of the form $[a:1]$ with $a\in\mathbb{Z}/p^\alpha\mathbb{Z}$ can be written 
$[1:p^k\beta]$
with $\gcd(\beta,p)=1$ and $0<k<\alpha$. 
Then:
$$[1:p^k\beta]\cdot\left[\begin{array}{cc} 1 & 0 \\ w & 1 \end{array}\right]=[1+wp^k\beta:p^k\beta]=[1:\frac{p^k\beta}{1+wp^k\beta}]\ .$$
The smallest $w$ such that 
\begin{equation}\label{eq:w}
\frac{p^k\beta}{1+wp^k\beta}=p^k\beta\ (\Leftrightarrow wp^{2k}\beta^2=0)
\end{equation}
holds in $\mathbb{Z}/p^\alpha\mathbb{Z}$, is the width of the cusp containing the projective lattice $[1:p^k\beta]$.

Now, since $\beta$ and $p$ are coprimes:
\begin{itemize}
\item either $2k\leq\alpha$, and then the smallest $w$ for which eq. \ref{eq:w} holds is $p^{\alpha-2k}$,
\item or $2k>\alpha$ then the smallest $w$ for which eq. \ref{eq:w} holds is $1$. 
\end{itemize}

In any case, $0\leq\beta\leq p^{\alpha-k}$, and since $\beta$ and $p$ are coprimes, there are $\phi(p^{\alpha-k})$ different possible $\beta$'s, hence: 
\begin{itemize}
\item either $2k\leq\alpha$, then there are 
$$\frac{\phi(p^{\alpha-k})}{p^{\alpha-2k}}=\frac{p^{\alpha-k-1(p-1)}}{p^{\alpha-2k}}=p^{k-1}(p+1)=\phi(\gcd(p^k,p^{\alpha-k}))$$
cusps containing points of the form $[1:p^k\beta]$, all of width $p^{\alpha-2k}$,
\item or $2k>\alpha$, and then there are $\phi(p^{\alpha-k})=\phi(\gcd(p^k,p^{\alpha-k}))$ cusps of width $1$. 
\end{itemize}

Eventually, the points of the form $[a:1]$ correspond to the special case $k=0$. They form a unique cusp $c_0$ of width $p^\alpha$. The point $[1:0]$ corresponds to $k=p^\alpha$ and form the cusp $c_\infty$ of width $1$. 

Hence we found that:
\begin{itemize}
\item for each value of $k$ in $[|0,\lfloor \alpha/2 \rfloor|]$, there are $\phi(\gcd(p^k,p^{\alpha-k}))$ cusps of width $p^{\alpha-2k}$,
\item for each value of $k$ in $[|\lfloor \alpha/2 \rfloor+1,\alpha|]$, there are $\phi(\gcd(p^k,p^{\alpha-k}))$ cusps of width $1$.
\end{itemize}
\end{proof}

\begin{proposition}\label{prop:cuspmult}
Let $M,N$ be two coprime integers. Then
$$\mathcal{C}(MN)=\mathcal{C}(M)\times\mathcal{C}(N) \ . $$
\end{proposition}

\begin{proof}
Consider the two canonical morphisms
$$\xymatrix{
\bip_{0,NM} \ar[d]^{(f,\phi)} \ar[r]^{(g,\chi)} & \bip_{0,N}\\
\bip_{0,M} } \ .
$$
From the very definition of these morphisms, the image of a cycle of $y_{0,MN}x_{0,MN}$ by $(f,\phi)$ (resp. $(g,\chi)$) can only be a cycle of $y_{0,M}x_{0,M}$ (resp. $y_{0,N}x_{0,N}$), possibly of smaller width, but in that case the width of the image divides the width of the original cycle.

Hence the width of the cusp containing the edge $(e_M,e_N)\in E_{0,M}\times E_{0,N}$ is a multiple of $w(e_M)w(e_N)$, where $w(e_M)$ (respectively $w(e_N)$) is the width of the cusp in $\bip_{0,M}$ (resp., $\bip_{0,N}$) containing $e_M$ (respectively $e_N$). 
One actually knows even more, since Proposition \ref{prop:sumW} and the equality
$$E_{0,MN}=E_{0,M}\times E_{0,N}$$ 
force it to be exactly $w(e_M)w(e_N)$. That concludes the proof.
\end{proof}

Proposition \ref{prop:cuspmult} can be rephrased as the statement that the ``width number'' function, which associates $|\mathcal{C}(N)|$ to $N\in\mathbb{N}_{>0}$, is multiplicative. Moreover, the reasonning in the proof of Proposition \ref{prop:cuspmult} together with Proposition \ref{prop:desccusps} show the following:

\begin{corollary}
Let $N>1$ be an integer. Then 

$$w:\mathcal{C}(N)\rightarrow \mathrm{Div}(N) \ ,$$

where $\mathrm{Div}(N)$ is the set of divisors of $N$. 
\end{corollary}

Note that this function is onto if and only if $N$ is square-free. Putting all together, we have proved the following result.

\begin{theorem}\label{thm:c}
Let $N>2$ be an integer. Then
$$|\mathcal{C}(N)|=\sum_{d|N}\phi(\gcd(d,\frac{N}{d}))\ .$$
To each $d$ dividing $N$, there correspond $\phi(\gcd(d,\frac{N}{d}))$ cusps. Writing $N=\prod_ip_i^{\alpha_i}$ and $d=\prod_ip_i^{\beta_i}$, the width of such a cusp is:
$$w(c_{d,k})=\prod_{i}\max(1,p_i^{\alpha_i-2\beta_i})\ .$$
\end{theorem}

\subsubsection{L-series of the cusps}

In this section we will denote $c(\cdot)$ the cusp number function $|\mathcal{C}(\cdot)|:\mathbb{N}_{>0}\rightarrow\mathbb{N}_{>0}$.

\begin{proposition}
Let $p\in\mathbb{N}$ be a prime number and let $\alpha\in\mathbb{N}_{>0}$. Then:
$$c(p^{2\alpha+1})=2p^\alpha\ .$$
If moreover $\alpha\geq1$, 
$$c(p^{2\alpha})=p^{\alpha-1}(p+1)\ .$$
\end{proposition}

\begin{proof}
From Theorem \ref{thm:c}, and for $k\in\mathbb{N}$
$$c(p^k)=\sum\limits_{0\leq l\leq k}\phi(\gcd(p^l,p^{k-l})\ .$$
Since $\phi(p^i)=p^{i-1}(p-1)$, one can write:
$$c(p^{2\alpha+1})=2\sum\limits_{0\leq k\leq \alpha} p^{k-1}(p-1)
=2+2(p-1)\sum\limits_{0\leq k\leq \alpha-1}p^{k}
=2+2(p-1)\frac{p^\alpha-1}{p-1}
=2p^\alpha.$$
Mutatis mutandis,
$$c(p^{2\alpha})=2\sum\limits_{0\leq k\leq \alpha-1} p^{k-1}(p-1)+p^{\alpha-1}(p-1)
=2+2(p^{\alpha-1})-2+p^{\alpha-1}(p-1)
=p^{\alpha-1}(p+1).$$
\end{proof}

\begin{definition}
The formal L-series associated to the function $c(n)$ is 
$$L(c,s)=\sum_{n\geq 0}\frac{c(n)}{n^s} \ . $$
Since $c(n)$ is multiplicative, one can write $L(c,s)$ as an Euler product
$$L(c,s)=\prod_{p\ \mathrm{prime}} L_p(c,s)=\prod_{p\ \mathrm{prime}} \sum_{\alpha\geq 0}\frac{c(p^\alpha)}{p^{\alpha s}}\ .$$
\end{definition}

For all prime $p$, and all $s\in\mathbb{C}$ such that
$$|s|>\frac{1}{2}+\frac{e}{2\ln2}  \ , $$ 
the series $L_p(c,s)$ converges absolutely. 

\begin{proposition}\label{prop:L}
Let $s\in\mathbb{C}$ satisfying the latter bound. One can rearrange $L_p(c,s)$ as:
$$L_p(c,s)=L_p(c,s)_e+L_p(c,s)_o=\sum_{\alpha\geq 0}\frac{c(p^{2\alpha})}{p^{2\alpha s}}+\sum_{\alpha\geq 0}\frac{c(p^{2\alpha+1})}{p^{(2\alpha+1) s}} \ ,$$
and the computation yields:
$$L_p(c,s)_e=\frac{p^s+p^{-s}}{p^s-p^{1-s}} \ , \qquad L_p(c,s)_o=\frac{2}{p^s-p^{1-s}}\ .$$
\end{proposition}

\begin{proof}
One readily computes:
\begin{align*}
L_p(c,s)_e & =1+(1+\frac{1}{p})\sum_{\alpha\geq1}\frac{p^\alpha}{p^{2\alpha s}}=1+(1+\frac{1}{p})\sum_{\alpha\geq1}(p^{1-2s})^\alpha
\\
&=1+(1+\frac{1}{p})(\frac{p^{1-2s}}{1-p^{1-2s}})=\frac{1-p^{1-2s}+p^{1-2s}+p^{-2s}}{1-p^{1-2s}}
=\frac{p^s+p^{-s}}{p^s-p^{1-s}} \ ,
\end{align*}
and similarly
$$L_p(c,s)_o=\frac{2}{p^s}\sum_{\alpha\geq0}p^{\alpha(1-2s)}=\frac{2}{p^s}\frac{1}{1-p^{1-2s}}=
\frac{2}{p^s-p^{1-s}} \ .$$
\end{proof}

\begin{corollary}
The series $L(c,s)$ can be expressed in terms of Riemann's $\zeta$-function:
$$L(c,s)=\zeta(2s-1)(\frac{\zeta(s)}{\zeta(2s)})^2 \ .$$
\end{corollary}

\begin{proof}
From Proposition \ref{prop:L}, for a given prime $p$, one has:
$$L_p(c,s)=L_p(c,s)_e+L_p(c,s)_o=\frac{p^s+p^{-s}}{p^s-p^{1-s}}+\frac{2}{p^s-p^{1-s}}
=\frac{(p^{s/2}+p^{-s/2})^2}{p^s-p^{1-s}}=\frac{(1+p^{-s})^2}{1-p^{1-2s}} \ .$$
Hence
$$L(c,s)=\prod_{p\ \mathrm{prime}}\frac{(1+p^{-s})^2}{1-p^{1-2s}}$$
Now
$\zeta(s)=\displaystyle{\prod_{p\ \mathrm{prime}}\frac{1}{1-p^{-s}}}$ and $\displaystyle{\frac{\zeta(s)}{\zeta(2s)}=\prod_{p\ \mathrm{prime}} (1+p^{-s})}$.
\end{proof}

\def\belyi{Bely\u{\i}\ }

\subsection{Complex structures and Bely\u{\i} maps}
In this subsection we discuss analytical aspects of the dessins, realised explicitly as preimages of so-called \belyi maps. We follow closely the presentation of \cite{joneswolfart}.


\subsubsection{The triangle group $\psl\simeq\Delta(2,3,\infty)$ and its action on $\mathbb{H}$}
The {\bf triangle group} of type $(a,b,c)$ is the group with presentation:
\begin{equation}
\Delta(a,b,c)=\langle X,Y,Z | X^a=Y^b=Z^c=XYZ=1\rangle \ .
\end{equation}
The modular group corresponds to the special case 
$$\Delta(2,3,\infty)=\langle S,U,Z | S^2=U^3=SUZ=1\rangle\simeq\psl\ .$$
Consider a hyperbolic triangle $T$ with internal angles $\pi/2$, $\pi/3$ and $0$ (for example, the triangle in the hyperbolic plane $\mathbb{H}$ with vertices $i$, $e^{i\pi/3}$ and $\infty$). Then, the group generated by the rotations through $2\pi/2$, $2\pi/3$ and $0$ about the vertices of this triangle is $\Delta(2,3,\infty)$. Let the {\bf extended triangle group} $\Delta[2,3,\infty]$ be the group generated by the reflections with respect to the sides of $T$. The half-plane $\mathbb{H}$ is tessellated by the images of $T$ under $\Delta[2,3,\infty]$, and the group $\Delta(2,3,\infty)$ is the subgroup of order $2$ in $\Delta[2,3,\infty]$ consisting of the transformations which preserve the orientation.

Consider the following graph embedded in $\mathbb{H}$: let there be a white (respectively, black and red) vertex at each image of $i$ (respectively, $e^{i\pi/3}$ and $\infty$) under $\Delta(2,3,\infty)$, and an edge for each image of the sides of $T$ under the same group. Now, remove the red vertices and all edges incident to them; this yields a bipartite graph embedded in $\mathbb{H}$. The counterclockwise orientation on $\mathbb{H}$ induces a fat structure on the graph, and the corresponding ABM is $\bip_\infty(2,3,\infty)$. By construction, $\psl$ is the group generated by the rotations about the vertices of the hyperbolic triangle $T$, and hence naturally appears as the automorphism group of this ABM:
\begin{equation}
\mathrm{Aut}(\bip_\infty(2,3,\infty))\simeq \Delta(2,3,\infty)\simeq \psl \ .
\end{equation}
Since this ABM is regular (because the automorphism group is transitive, for example), $\psl$ is also the cartographic group of $\bip_\infty(2,3,\infty)$.

\subsubsection{Complex structure on the surfaces corresponding to the $\bip_{0,N}$}

Let $N\in\mathbb{N}_{>0}$. Recall that the dessin d'enfant $\bip_{0,N}$ is the quotient $\He(N)\backslash\bip_\infty(2,3,\infty)$, and that it comes with topological surfaces $X_0(N)$ and $Y_0(N)$. 

The embedding $\bip_\infty(2,3,\infty)\hookrightarrow\mathbb{H}$ induces a complex structure on $X_0(N)$ and $Y_0(N)$, as explained pedagogically in \cite{joneswolfart}. As shown in Figures \ref{f:egFund9} and \ref{f:egFund10},each $\bip_{0,N}$ corresponds to a fundamental domain for the action of $\He(N)$ on $\mathbb{H}$.

The complex structure on the surfaces $X_0(N)$ and $Y_0(N)$ may have torsion (or orbifold) points of order $2$ and $3$. In terms of Fuchsian groups, each of these torsion points corresponds to an equivalence class of fixed points for some elliptic transformations in $\He(N)$. In the bipartite fat graphs, the torsion points of order $2$ correspond to the $1$-valent white vertices and the torsion points of order $3$, to the $1$-valent black vertices. Recall that we have computed their number for each $N$ in subsection \ref{subcusps}.

\begin{figure}[t!h!]
\begin{subfigure}{0.7\textwidth}
\centering
\includegraphics[scale=0.4]{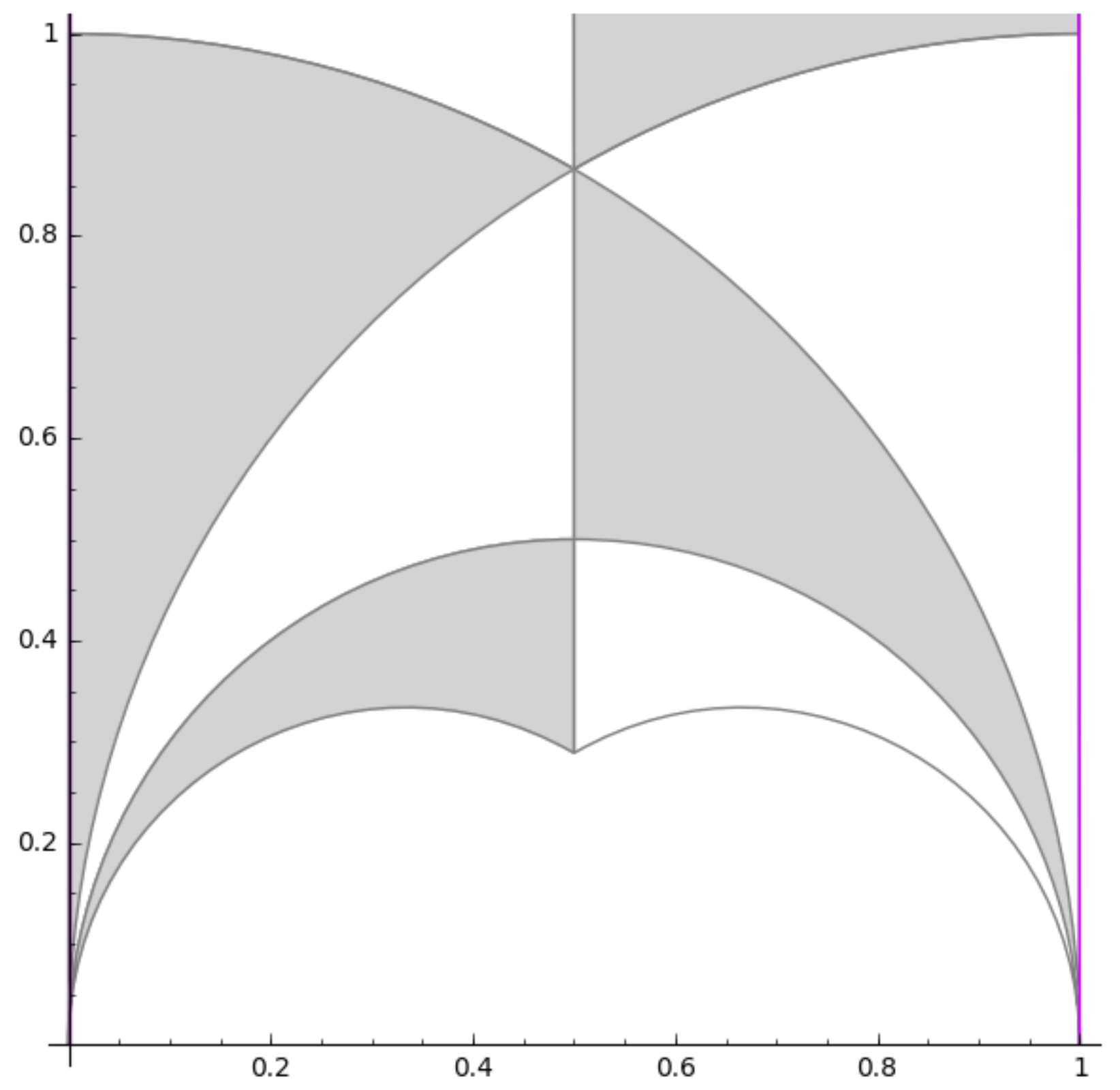}
\end{subfigure}
\begin{subfigure}{.2\textwidth}
\begin{tikzpicture}[scale=1.5]
\draw[fill=black] (0,0) circle (0.1);
\draw[fill=black] (0,-2) circle (0.1);
\draw (0,1) circle (0.1);
\draw (0,-1) circle (0.1);
\draw (0,0.5) circle(0.5);
\draw (0,0)--(0,-0.9);
\draw (0,-1.1)--(0,-2);
\end{tikzpicture}
\end{subfigure}
\caption{
{\sf Fundamental domains for $\He(3)$ (on the left) and $\bip_{0,3}$ (on the right).}
\label{f:egFund9}
}
\end{figure}

\begin{figure}[t!h!]\label{fig10}
\begin{subfigure}{0.6\textwidth}
\centering
\includegraphics[scale=0.4]{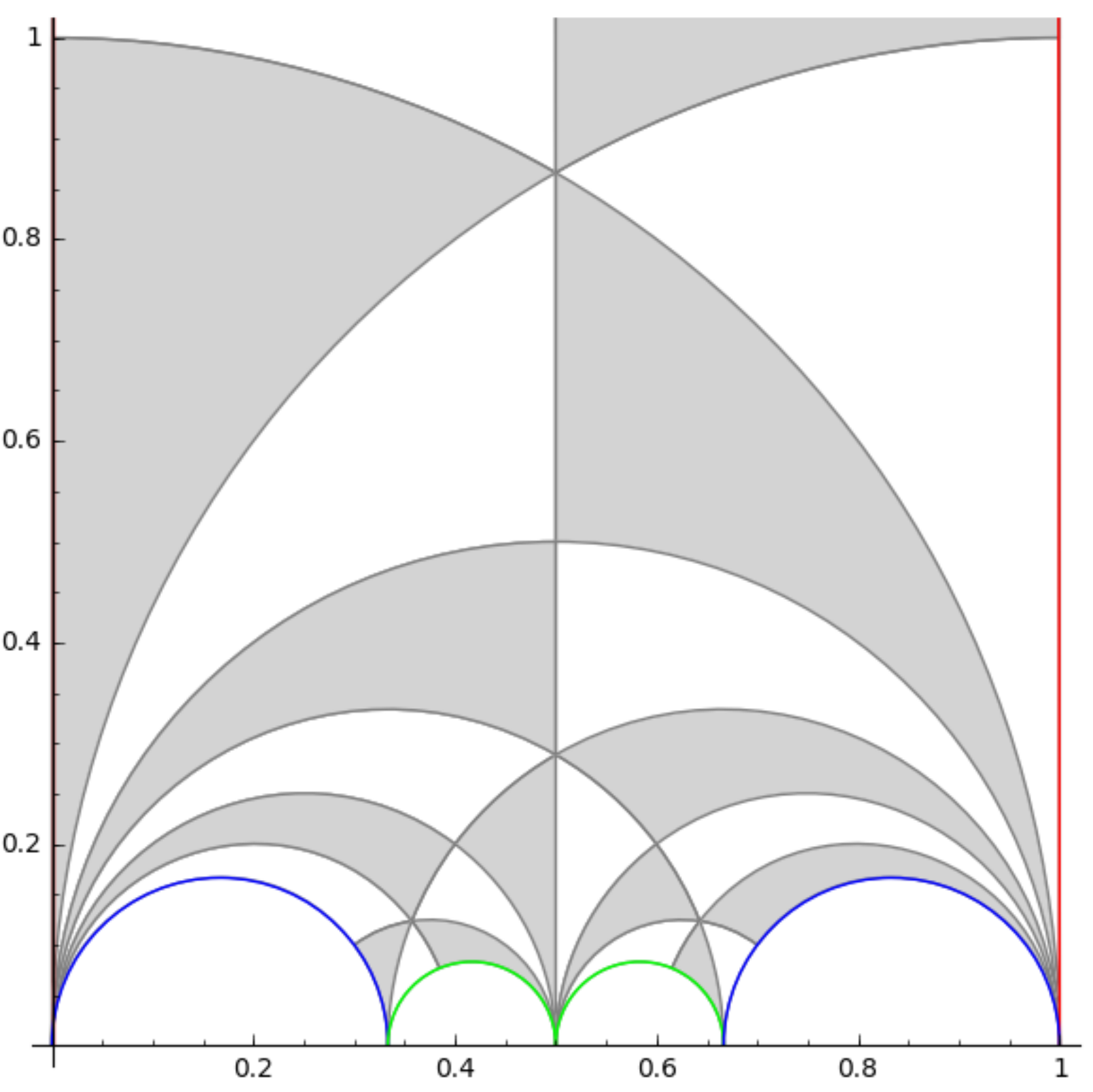}
\end{subfigure}
\begin{subfigure}{0.3\textwidth}
\centering
\begin{tikzpicture}[scale=1.3]
\draw[fill=black] (0,0) circle (0.1);
\draw[fill=black] (0,-2) circle (0.1);
\draw (0,1) circle (0.1);
\draw (-0.5,-2.5) circle (0.1);
\draw (0.5,-2.5) circle (0.1);
\draw (0,-1) circle (0.1);
\draw (0,-3) circle (0.1);
\draw (0,-4) circle (0.1);
\draw[fill=black] (1,-3) circle (0.1);
\draw[fill=black] (-1,-3) circle (0.1);
\draw (0,0.5) circle(0.5);
\draw (0,-2)--(-1,-3);
\draw (0,-2)--(1,-3);
\draw (0,0)--(0,-0.9);
\draw (0,-1.1)--(0,-2);
\draw (-1,-3)--(1,-3);
\draw (-1,-3) arc(180:360:1);
\end{tikzpicture}
\end{subfigure}
\caption{
{\sf Fundamental domains for $\He(6)$ (on the left) and $\bip_{0,6}$ (on the right).}
\label{f:egFund10}
}
\end{figure}

\begin{example}
On the left-hand-side of Figures \ref{f:egFund9} and \ref{f:egFund10}, fundamental domains for $\He(3)$ and $\He(6)$ obtained with SAGE \cite{sage} are displayed.
In these fundamental domains, the images of the triangle $T$ under $\psl$ are in white, those of $T'$, in grey. Colors on the edges label the identifications through which one recovers the topology of the quotient surface $\He(N)\backslash\mathbb{H}$, but those leading to torsion points (for example, the identification of the two lowermost edges of the fundamental domain shown for $\He(3)$ is implicit). The corresponding dessins d'enfants (resp. $\bip_{0,3}$ and $\bip_{0,6}$) are drawn on the right-hand side of Fig. \ref{f:egFund9} and \ref{f:egFund10}.

\end{example}

\subsubsection{Bely\u{\i}'s Theorem and dessins d'enfants}

The whole theory of dessins d'enfants, and the reason they are related to some number-theoretic questions, relies on the following key theorem \cite{belyi}. Let $X$ be a compact Riemann surface. It is a deep and fundamental result that $X$ is biholomorphic to the analytic set of the complex points of a smooth algebraic curve, in a complex projective space $\mathbb{P}^n(\mathbb{C})$ for some $n\in\mathbb{N}_{>0}$. 

Let $K$ be a subfield of $\mathbb{C}$. A smooth algebraic curve has {\bf a model over $K$} if the underlying analytic variety is isomorphic to the zero locus of a finite set of polynomials with coefficients in $K$, in some affine or projective complex space. Recall that $\overline{\mathbb{Q}}$ is the algebraic closure of the field of rational numbers $\mathbb{Q}$ (equivalently, the field of algebraic numbers).

\begin{theorem}[Bely\u{\i}, 1979]
The compact Riemann surface $X$ has a model over $\overline{\mathbb{Q}}$ if and only if there exists a non-constant holomorphic function $\beta:X\rightarrow\mathbb{P}^1(\mathbb{C})$ which ramifies over at most three points (which can be chosen to be $0,1$, and $\infty$, by considering the action of $\mathrm{PSL}_2(\mathbb{C})$ on $\mathbb{P}^1(\mathbb{C})$ by automorphisms). 
\end{theorem}

Such a map $\beta$ is called Bely\u{\i} map. The preimage under $\beta$ of the real segment $[0,1]\in\mathbb{P}^1(\mathbb{C})$ is a bipartite graph embedded in $X$. Since $X$ has a complex structure it is oriented, and this orientation defines a fat structure on this bipartite graph. 
Hence, any Bely\u{\i} map defines a dessin d'enfant.

Conversely, a dessin d'enfant canonically defines a compact surface on which its underlying graph is embedded. The exact way the dessin d'enfant defines a complex structure on this compact surface is a slight generalisation of the two last paragraphs (again, see \cite{joneswolfart} for more details on this implication). Given a dessin d'enfant, there is always a corresponding Bely\u{\i} map for which the white vertices (respectively, black vertices) are the preimages of 0 (resp., 1) and the edges, the preimages of the segment $[0,1]$. The explicit expression of this Bely\u{\i} map is in general difficult to derive. One can however motivate its existence as follows.

The dessins d'enfants of type $(2,3,c)$ yield complex structures built from the hyperbolic triangle $T$ of type $(2,3,\infty)$. This (open) triangle is conformally equivalent to $\mathbb{H}$ as a consequence of Riemann's open mapping theorem, and a stronger version of the latter even implies that the biholomorphism
$J:T\rightarrow\mathbb{H}$
can be extended continuously to the boundary $\partial T$ of $T$, and chosen in such a way that the vertices of $T$ are mapped to $0,1$ and $\infty$. Schwarz's reflection principle then asserts that one can extend $J$ to its image $T'$ under the reflection through the edge $[i,e^{i\pi/3}]$, which yields a map
$J:(T\cup T')\rightarrow\mathbb{P}^1(\mathbb{C})$. This map $J$ is usually called Klein's function or Klein's J-invariant. In what follows we will denote this J-invariant $J_{0,1}$, in order to avoid confusion.  
Successive applications of the reflection principle indeed further extend $J$ to:
\begin{equation}\label{J}
J:G\backslash\mathbb{H}\rightarrow\mathbb{P}^1(\mathbb{C}) \ ,
\end{equation}
for any subgroup $G<\psl$, and this $J$ can even be continued on the compactification of $G\backslash\mathbb{H}$ (with some help from the removable singularity theorem).

In the end, any subgroup of $G<\psl$ of finite index (e.g., a $\He(N)$) gives rise to a complex surface with cusps, with Fuchsian model $G\backslash\mathbb{H}$. This surface can be compactified by adding a point at each cusp, and comes with a Bely\u{\i} map obtained from Klein's invariant $J_{0,1}$ through the reflection principle. Bely\u{\i}'s theorem then states that the algebraic curve defined by such a dessin d'enfant always has a model over a number field.

In fact, it is a classical result that the algebraic curves $Y_0(N)$ and $X_0(N)$, for $N\in\mathbb{N}$, have a model over $\mathbb{Q}$, even if their defining equation over $\mathbb{Q}$ is in general hard to derive. The complete projective algebraic curve $X_0(N)$ corresponding to $\He(N)$ is usually called the \emph{(compact) classical modular curve}. 
It satisfies a polynomial equation with rational coefficients 
$$\Phi(x,y)=0$$
such that $(x,y)=(J(\tau),J(N\tau))$ is a point of the curve, with $J$ the \emph{usual} Klein's function.

\subsubsection{Genus formula}

For each dessin $\bip_{0,N}$ we have a complete description of the set of torsion points of order $2$ and $3$, the set of cusps, and their width. Moreover we know that the map
$$X_0(N)\rightarrow \mathbb{P}^1(\mathbb{C})$$
induced by reflection principle on Klein's invariant $J_{0,1}$ is the \belyi map corresponding to this dessin $\bip_{0,N}$. This map ramifies at the vertices and the cusps. The ramification order is the valency for a vertex (or the width for a cusp). We now have enough data to compute the genus of $X_0(N)$ for all $N\in\mathbb{N}_{>0}$ using Riemann-Hurwitz formula. This gives the following.

\begin{theorem}\label{thm4}
Let $N\in\mathbb{N}_{>0}$, $|E_{0,N}|=\displaystyle{N\prod_{p|N}(1+p^{-1})}$, and $\nu_2(N)$ (resp., $\nu_3(N)$) the number of torsion points of order $2$ (resp., $3$) of the dessin $\bip_{0,N}$. Let $c_w(N)$ be the number of cusps of width $w$ in $\bip_{0,N}$. Then: 
$$\chi(X_0(N))=2|E_{0,N}|-\frac{1}{2}(|E_{0,N}|-\nu_2(N))-\frac{2}{3}(|E_{0,N}|-\nu_3(N))-\sum_{w\geq 1}c_w(N)(w-1) \ , $$
where $\chi(X_0(N))$ is the Euler characteristic of $X_0(N)$.
\end{theorem}

\begin{corollary}
Let $p$ be a prime number, and $g(p)$ the genus of $X_0(p)$. Then:
$$\frac{p-13}{12}\leq g(p) \leq \frac{p+1}{12} \ . $$
\end{corollary}

\begin{proof}
For $p$ prime, $0\leq \nu_2(p)\leq 2$ and $0\leq \nu_3(p)\leq 2$, and there are only two cusps: $c_0$ of width $p$ and $c_1$ of width $1$. Hence, applying Theorem \ref{thm4} one gets: 
$$\frac{1}{6}(11-p)\leq \chi(X_0(p))\leq \frac{1}{6}(25-p) \ ,$$
and $\chi(X_0(p))=2-2g(X_0(p))$ gives the desired bounds. 
\end{proof}

Some interesting properties of the sequence of genera of the classical modular curves, like bounds, modularity properties and densities are investigated in \cite{genera}.

\subsubsection{Moduli problem of level-$N$ structures on elliptic curves}

The classical modular curves $Y_0(N)$ are known to solve a moduli problem. 

Let $E$ be an elliptic curve over a perfect field $k$, typically the field of rational numbers $\mathbb{Q}$, and let $N\in\mathbb{N}_{>0}$. A cyclic subgroup of $E$ of order $N$ is a Zariski-closed subset $S$ of $E$ such that $S(\bar{k})$ is a cyclic subgroup of $E(\bar{k})$ of order $N$, where $\bar{k}$ is the algebraic closure of $k$. Consider the pairs $(E,S)$ up to isomorphism, where
$$f:(E,S)\rightarrow (E',S')$$
is an isomorphism if $f:E\rightarrow E'$ an isomorphism such that $f(S)=S'$. It is the moduli problem we are interested in, and has the modular curve $Y_0(N)$ as solutions. See \cite{milne} for a more detailed discussion.

Over the complex numbers, the elliptic curves correspond to the projective lattices in the complex plane. A projective lattice $L_1$ in a two dimensional real vector space, and modulo $\psl$, corresponds to an elliptic curve $E_1$. The projective lattices $N$-hyperdistant from $L_1$ correspond in turn to the cyclic subgroups of $E$ of order $N$. Hence the dessins d'enfants describe the part of the structure of the moduli spaces $Y_0(N)$ which concerns the \emph{cyclic subgroups of order $N$}, while the complex surface associated with those dessins brings the \emph{moduli space of complex structures} to the picture.

\subsection{Hauptmoduln and \belyi maps}
We now focus on the special class of Hecke congruence subgroups that appear in the Monstrous Moonshine correspondence \cite{conwaynorton}, namely, those of genus $0$.

\subsubsection{Hauptmoduln for genus zero algebraic curves}
Let $\tilde{X}$ be an analytic projective irreducible curve embedded in some $\mathbb{P}^n(\mathbb{C})$. 
There exists a non-singular model $X$ of $\tilde{X}$ which has the same field of meromorphic functions as $X$: $$\mathcal{M}_X=\mathcal{M}_{\tilde{X}}\ .$$ 
It is a classical result (see for example \cite{fulton}) that \emph{over the complex numbers}, the following holds:
\[
\mbox{
$\tilde{X}$ is rational $\qquad \Leftrightarrow \qquad$
$X$ is biholomorphic to $P^1(\mathbb{C})$,
}
\]
that is, over $\mathbb{C}$ the curve $\tilde{X}$ is rational if and only if $X$ has genus zero. 
Whenever it is the case, $\tilde{X}$ can be parametrised by a rational function of a single variable (which lives on $\mathbb{P}^1$). 

\begin{definition}\label{haupt}
Let $X$ be a genus-zero Riemann surface. 
Its field of meromorphic functions $\mathcal{M}_X$ is the field of rational fractions in a single meromorphic function over $X$. 
Such a function is called a {\bf Hauptmodul} (or principal modulus) for $X$.
\end{definition}


\subsubsection{Bely\u{\i} maps and replication Formul\ae\ for $J$}

From now on, for all $N\in\mathbb{N}$ let $X_0(N)$ be the analytic complex curve $X_0(N)(\mathbb{C})$.

\paragraph{Choice of a coordinate}

Exactly $15$ classical modular curves $X_0(N)$ are of genus zero (and hence rational): those corresponding to
$$N\in\{1,2,3,4,5,6,7,8,9,10,12,13,16,18,25\}\ .$$
For all $N$ in this set, there is a conformal isomorphism (whose existence is provided by the open-mapping theorem):
$$J_{0,N}:X_0(N)\rightarrow \mathbb{P}^1(\mathbb{C})\ .$$
The function $J_{0,N}$ is a Hauptmodul for the curve $X_0(N)$:
$$\left.\begin{array}{ccc} X_0(N) & \rightarrow & \mathbb{P}^1(C)\\
\tau & \mapsto & J_{0,N}(\tau)=:t \end{array}\right.$$

The analytic curve $X_0(N)$ is the (smoothened compactification of the) quotient of the upper-half plane under the action of $\He(N)$. Every choice of fundamental domain for $\He(N)$ in $\mathbb{H}$ defines a chart on $X_0(N)$, on which $J_{0,N}$ can be explicitly expressed in terms of Dedekind's $\eta$ (see table 3 of \cite{conwaynorton}).

\paragraph{Bely\u{\i} maps}

Recall that Klein's invariant $J_{0,1}$ defines a branched cover 
$$\beta_{0,N}:X_0(N)\rightarrow\mathbb{P}^1(\mathbb{C})$$ 
which ramifies only over $1$, $0$ and $\infty$, with ramification order dividing $2$ over $1$, and $3$ over $0$. Since $X_0(N)$ is biholomorphic to $\mathbb{P}^1(\mathbb{C})$, this map can be expressed as:
$$\beta_{0,N}:\mathbb{P}^1(\mathbb{C})\rightarrow\mathbb{P}^1(\mathbb{C})\ .$$ 
It is a Bely\u{\i} map which corresponds to the dessin d'enfant $\bip_{0,N}$. 

This map $\beta_{0,N}$ is a rational function of $t$, by definition of a Hauptmodul. Note that:
\begin{enumerate}
\item The set of preimages of $0$ is the set of black vertices of the corresponding dessin. The multiplicity of a root is the valence of the corresponding vertex. The set of poles is the set of faces, and the multiplicity of a pole is the width of the corresponding cusp. The set of preimages of $1$ is the set of white vertices, and the multiplicity of a preimage of $1$ is the valence of the corresponding vertex. 
\item Let $N\in\mathbb{N}_{>0}$. Since the dessin d'enfant $\bip_{0,N}$ is of type $(2,3,N)$, the multiplicities of the roots of the numerator of $\beta_{0,N}$ are either one or three, and the multiplicities of the preimagesof $1$, either one or two. 
\item The classical modular curves $X_0(N)$ all have a model over $\mathbb{Q}$: they can be defined by a polynomial equation
$$\Phi_N(X,Y)=0$$
where $\Phi_N\in\mathbb{Q}[X,Y]$, and such that $(J(\tau),J(N\tau))$ is a point of $X_0(N)$. The function $J(\tau)$ is  is a rational fraction of the Hauptmodul $t=J_{0,N}(\tau)$, with rational (equivalently, integer) coefficients: $J(\tau)=\beta_{0,N}(t)$. 
\end{enumerate}

\begin{figure}[h!]
\centering
\begin{tabular}{|c|c|}
\hline
$N$ & $J(\tau)=\beta_{0,N}(t)$ (Bely\u{\i} map) \\
\hline\hline
$1$ & $t$ \\
\hline 
$2$ & $\frac{(t+256)^3}{1728t^2}$ \\
\hline 
$3$ & $\frac{27(t+9)^3(t+1)}{1728t^3}$ \\
\hline 
$4$ & $\frac{16(t^2+16t+16)^3}{1728(t+1)t^4}$ \\
\hline 
$5$ & $\frac{(t^2+250t+3125)^3}{1728t^5}$ \\
\hline 
$6$ & $\frac{(2t+3)^3(8t^3+252t^2+486t+243)^3}{1728t^6(8t+9)^3(t+1)^2}$ \\
\hline 
$7$ & $\frac{(t^2+13t+49)(t^2+245t+2401)^3}{1728t^7}$ \\
\hline 
$8$ & $\frac{4(t^4+64t^3+320t^2+512t+256)^3}{1728t^8(t+2)^2(t+1)}$ \\
\hline 
$9$ & $\frac{3(t+3)^3(t^3+81t^2+243t+243)^3}{1728t^9(t^2+3t+3)}$ \\
\hline 
$10$ & $ \frac{(t^6+260t^5+6400t^4+64000t^3+320000t^2+800000t+800000)^3}{1728t^{10}(t+5)^5(t+4)^2}$ \\
\hline 
$12$ & $ \frac{(3t^6+252t^5+1464t^4+3456t^3+4032t^2+2304t+512)^3(3t^2+12t+8)^3}{1728t^{12}(3t+4)^4(t+2)^3(t+1)^3(3t+2)}$ \\
\hline 
$13$ & $ \frac{(t^4+247t^3+3380t^2+15379t+28561)^3(t^2+5t+13)}{1728t^{13}}$ \\
\hline 
$16$ & $ \frac{2(t^8+128t^7+1408t^6+6656t^5+17664t^4+28672t^3+28672t^2+16384t+4096)^3}{1728t^{16}(t+2)^4(t^2+2t+2)(t+1)}$ \\
\hline 
$18$ & $ \frac{(t^9+9t^8+270t^7+1728t^6+5832t^5+13122t^4+21870t^3+26244t^2+19683t+6561)^3(t^3+3t^2+9t+9)^3}{1728t^{18}(t+3)^9(t^2+3t+3)^2(t^2+3)^2(t+1)}$ \\
\hline 
$25$ & $\frac{(t^{10}+250t^9+4375t^8+35000t^7+178125t^6+631250t^5+1640625t^4+3125000t^3+4296875t^2+3906250t+1953125)^3}{1728t^{25}(t^4+5t^3+15t^2+25t+25)}$ \\
\hline 
\end{tabular}
\caption{These explicit expressions of the Bely\u{\i} maps have been obtained from those in \cite{hoeij}, through $\mathrm{PSL}_2(\mathbb{C})$-rotations of $t$. It is an interesting fact that the the coefficients appearing in these $\beta_{0,N}$ (for these choices of $t$) are not only integer, but positive integers. We are not aware of any explanation of this fact in the litterature.}
\end{figure}

\paragraph{Divisibility relations}

Let $N\in\mathbb{N}$, and let $d$ be a divisor of $N$. Since there exists a canonical projection:
$$\bip_{0,N}\rightarrow\bip_{0,k}\ ,$$
the map $J_{0,k}$ defines a function on $X_0(N)$ (through the reflection principle). Again by definition of a Hauptmodul, the induced $J_{0,k}$ is rational fraction of $J_{0,N}$. 

Let $k|N$ and $l|k$. The following diagram commutes.
$$\xymatrix{
& X_{0}(N) \ar[dr]^{\beta_{N,k}} \ar[dl]_{\beta_{N,l}} &  \\
X_{0}(l) \ar[dr]_{\beta_{l,1}} & \ & X_{0}(k) \ar[dl]^{\beta_{k,1}} \\
 & X_0(1) &  }
$$


\section{Genus zero Hecke groups}\label{sec3}

We tabulate fundamental domains, the dessins and the cusps (as cycles of projective lattices) of the 15 genus zero Hecke subgroups, which are exactly the Hecke subgroups of $\psl$ appearing in the moonshine correspondence. 

In the simplest cases (up to $N=9$), on each edge in the dessin we write the name of the corresponding projective lattice (following the rules described in Appendix \ref{app}). For $N>9$ we only write the name of a single projective lattice in each cusp directly on the graph, to avoid being too cumbersome. However, the knowledge of where this projective lattice sits on the graph together with the tabulation of the cusps on the side is enough to keep track of which lattice corresponds to which edge in the dessin. We leave to the reader to check and get familiar with this general rule in the $9$ first cases, where the correspondence edge/projective lattice is completely explicit.

\subsection*{$\He(1)$}

The index in $\psl$ is $1$.

\begin{figure}[h!]
\begin{subfigure}{.6\textwidth}
\centering
\includegraphics[scale=0.35]{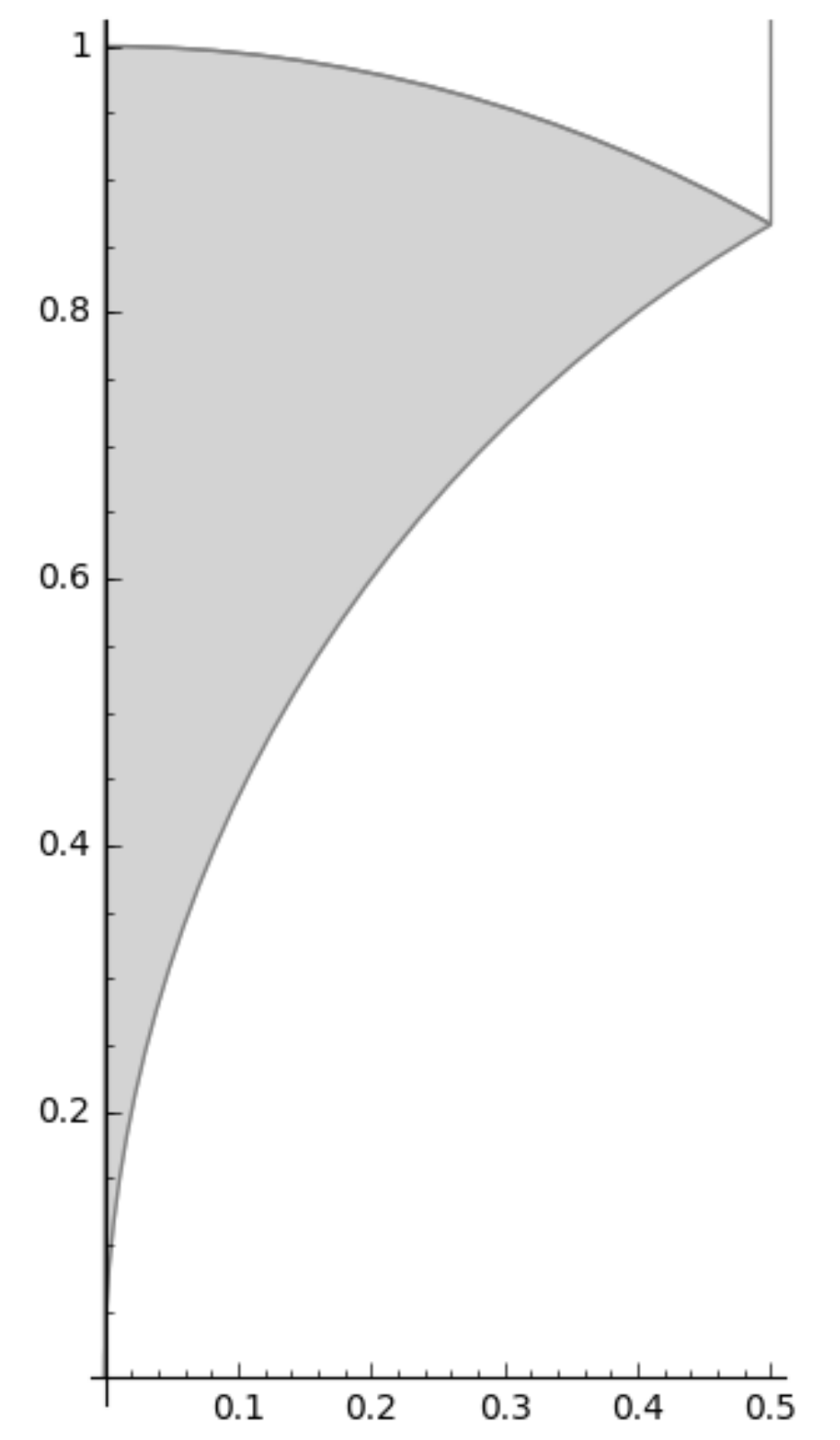}
\end{subfigure}
\begin{subfigure}{.3\textwidth}
\begin{tikzpicture}[scale=1.5]
\draw[fill=black] (0,0) circle (0.1);
\draw (0,1) circle (0.1);
\draw (0,0) arc(270:92:0.5);
\node (c) at (-0.5,0.5) {$L_1$};
\end{tikzpicture}
\end{subfigure}
\end{figure}

\begin{figure}[h!]
\centering
\begin{tabular}{|c|c|c|}
\hline
Cusp & Representative & Width \\
\hline
\hline
$\{[0:1]\}=\{L_1\}$ & $\infty$ & $1$ \\
\hline
\end{tabular}
\end{figure}

\clearpage

\subsection*{$\He(2)$}

The index in $\psl$ is $3$.

\begin{figure}[h!]
\begin{subfigure}{.6\textwidth}
\centering
\includegraphics[scale=0.35]{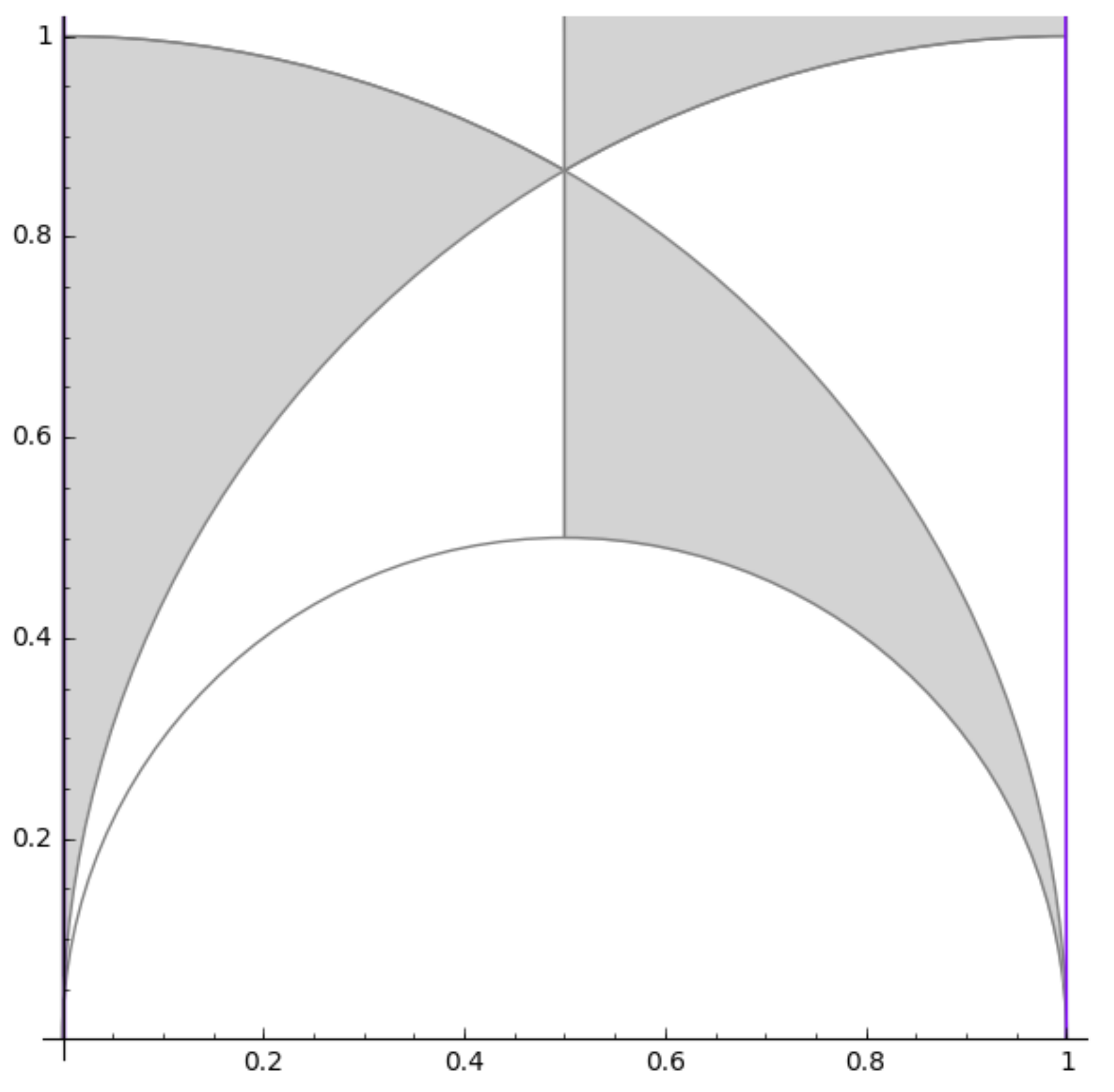}
\end{subfigure}
\begin{subfigure}{.3\textwidth}
\begin{tikzpicture}[scale=1.5]
\draw[fill=black] (0,0) circle (0.1);
\draw (0,1) circle (0.1);
\draw (0,-1) circle (0.1);
\draw (0,0.5) circle(0.5);
\draw (0,0)--(0,-0.9);
\node (c) at (-0.5,0.5) {$L_2$};
\node (c) at (0.5,0.5) {$L_{1/2}$};
\node (c) at (0,-0.5) {$L_{1/2,1/2}$};
\end{tikzpicture}
\end{subfigure}
\end{figure}

\begin{figure}[h!]
\centering
\begin{tabular}{|c|c|c|}
\hline
Cusp & Representative & Width \\
\hline
\hline
$([0:1],[1:1])=(L_{2},L_{1/2,1/2}) $ & $0$ & $2$ \\
\hline
$([1:0])=(L_{1/2})$ & $\infty$ & $1$ \\
\hline
\end{tabular}
\end{figure}

\subsection*{$\He(3)$}

The index in $\psl$ is $4$.

\begin{figure}[h!]
\begin{subfigure}{.6\textwidth}
\centering
\includegraphics[scale=0.35]{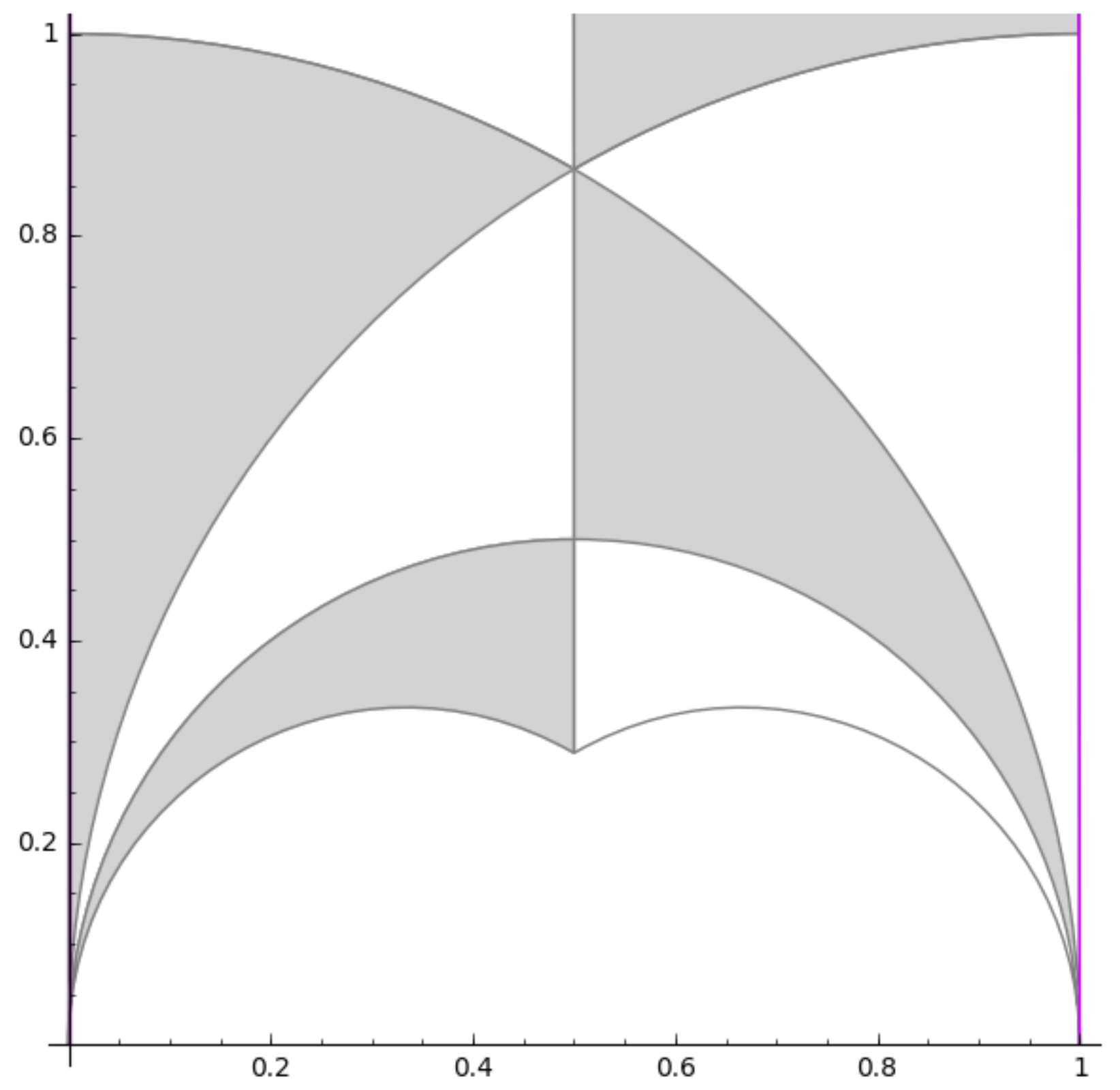}
\end{subfigure}
\begin{subfigure}{.3\textwidth}
\begin{tikzpicture}[scale=1.5]
\draw[fill=black] (0,0) circle (0.1);
\draw[fill=black] (0,-2) circle (0.1);
\draw (0,1) circle (0.1);
\draw (0,-1) circle (0.1);
\draw (0,0.5) circle(0.5);
\draw (0,0)--(0,-0.9);
\draw (0,-1.1)--(0,-2);
\node (c) at (-0.5,0.5) {$L_3$};
\node (c) at (0.5,0.5) {$L_{1/3}$};
\node (c) at (0,-0.5) {$L_{1/3,2/3}$};
\node (c) at (0,-1.5) {$L_{1/3,1/3}$};
\end{tikzpicture}
\end{subfigure}
\end{figure}

\begin{figure}[h!]
\centering
\begin{tabular}{|c|c|c|}
\hline
Cusp & Representative & Width \\
\hline
\hline
$([0:1],[1:1],[2:1])=(L_3,L_{1/3,1/3};L_{1/3,2/3})$ & $0$ & $3$ \\
\hline
$([1:0])=(L_{1/3})$ & $\infty$ & $1$ \\
\hline
\end{tabular}
\end{figure}

\clearpage

\subsection*{$\He(4)$}

The index in $\psl$ is $6$.

\begin{figure}[h!]
\begin{subfigure}{.6\textwidth}
\centering
\includegraphics[scale=0.35]{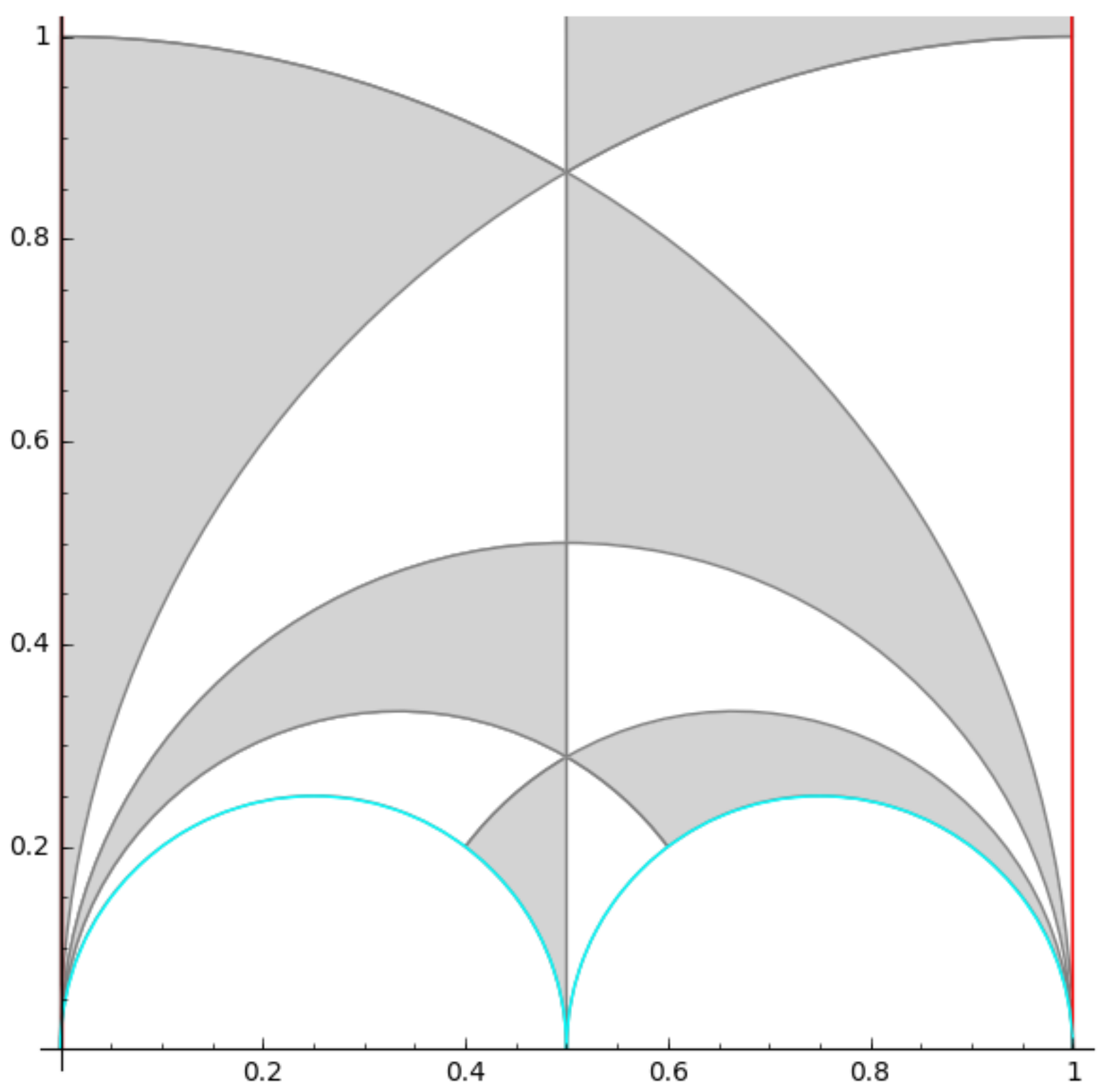}
\end{subfigure}
\begin{subfigure}{.3\textwidth}
\begin{tikzpicture}[scale=1.5]
\draw[fill=black] (0,0) circle (0.1);
\draw[fill=black] (0,-2) circle (0.1);
\draw (0,1) circle (0.1);
\draw (0,-1) circle (0.1);
\draw (0,-3) circle (0.1);
\draw (0,0.5) circle(0.5);
\draw (0,-2.5) circle(0.5);
\draw (0,0)--(0,-0.9);
\draw (0,-1.1)--(0,-2);
\node (c) at (-0.5,0.5) {$L_4$};
\node (c) at (0.5,0.5) {$L_{1/4}$};
\node (c) at (0,-1.5) {$L_{1/4,1/4}$};
\node (c) at (0.5,-2.5) {$L_{1,1/2}$};
\node (c) at (0,-0.5) {$L_{1/4,3/4}$};
\node (c) at (-0.5,-2.5) {$L_{1/4,1/2}$};
\end{tikzpicture}
\end{subfigure}
\end{figure}

\begin{figure}[h!]
\centering
\begin{tabular}{|c|c|c|}
\hline
Cusp & Representative & Width \\
\hline
\hline
$([0:1],[1:1],[2:1],[3:1])=(L_4,L_{1/4,1/4},L_{1,1/2},L_{1/4,3/4})$ & $0$ & $4$ \\
\hline
$[1:2]=L_{1/4,1/2}$ & $1/2$ & $1$ \\
\hline
$[1:0]=L_{1/4}$ & $\infty$ & $1$ \\
\hline
\end{tabular}
\end{figure}

\subsection*{$\He(5)$}

The index in $\psl$ is $6$.

\begin{figure}[h!]
\begin{subfigure}{.6\textwidth}
\centering
\includegraphics[scale=0.35]{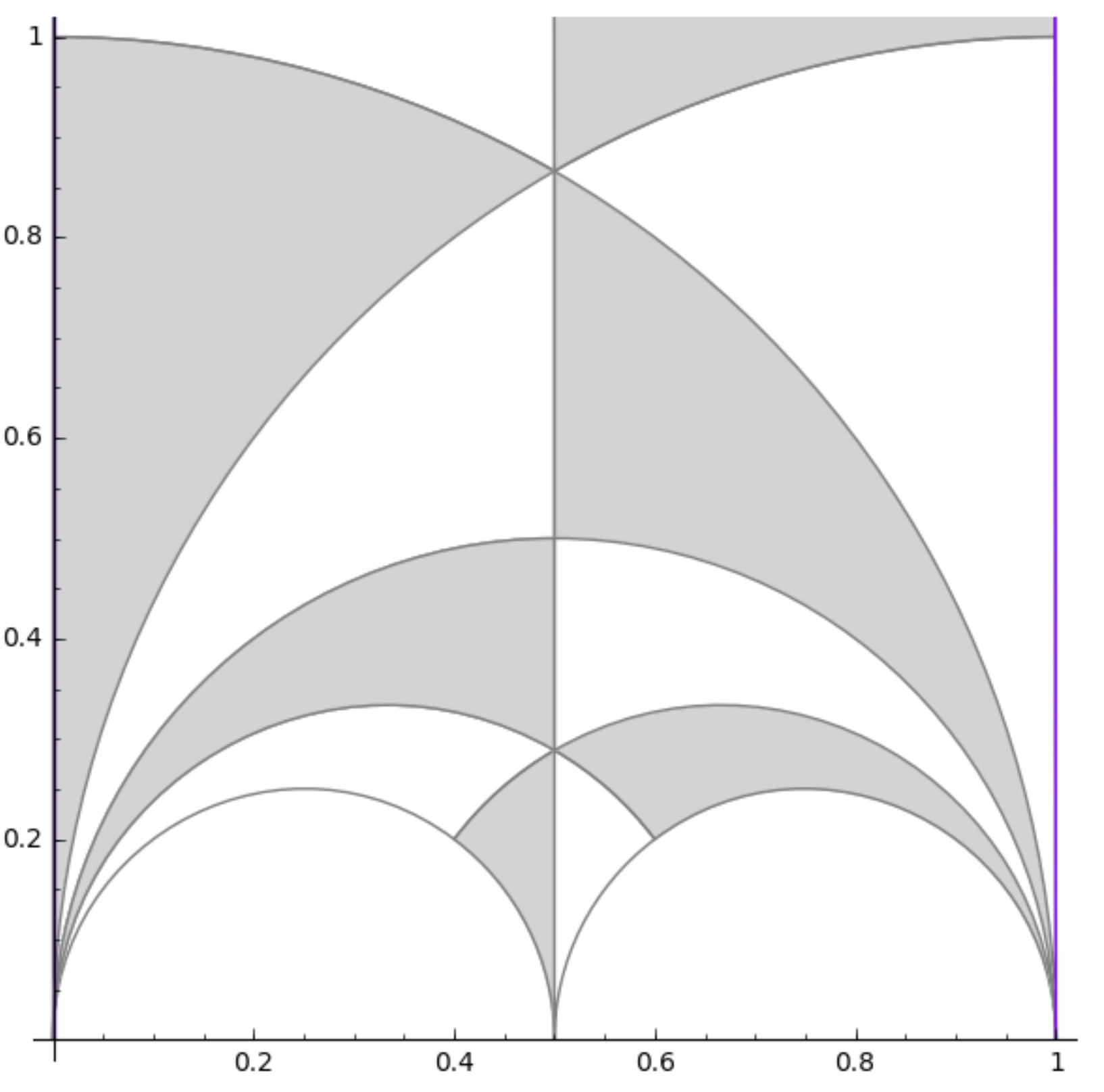}
\end{subfigure}
\begin{subfigure}{.3\textwidth}
\begin{tikzpicture}[scale=1.5]
\draw[fill=black] (0,0) circle (0.1);
\draw[fill=black] (0,-2) circle (0.1);
\draw (0,1) circle (0.1);
\draw (0,-1) circle (0.1);
\draw (1,-3) circle (0.1);
\draw (-1,-3) circle (0.1);
\draw (0,0.5) circle(0.5);
\draw (0,-2)--(-0.94,-2.94);
\draw (0,-2)--(0.94,-2.94);
\draw (0,0)--(0,-0.9);
\draw (0,-1.1)--(0,-2);
\node (c) at (-0.5,0.5) {$L_5$};
\node (c) at (0.5,0.5) {$L_{1/5}$};
\node (c) at (0,-1.5) {$L_{1/5,1/5}$};
\node (c) at (-0.6,-2.5) {$L_{1/5,3/5}$};
\node (c) at (0.6,-2.5) {$L_{1/5,2/5}$};
\node (c) at (0,-0.5) {$L_{1/5,4/5}$};
\end{tikzpicture}
\end{subfigure}
\end{figure}

\begin{figure}[h!]
\centering
\begin{tabular}{|c|c|c|}
\hline
Cusp & Representative & Width \\
\hline
\hline
$([0:1],...,[4:1])=(L_5,L_{1/5,1/5},L_{1/5,3/5},L_{1/5,2/5},L_{1/5,4/5})$ & $0$ & $5$ \\
\hline
$([1:0])=(L_{1/5})$ & $\infty$ & $1$ \\
\hline
\end{tabular}
\end{figure}

\clearpage

\subsection*{$\He(6)$}

The index in $\psl$ is $12$.

\begin{figure}[h!]
\begin{subfigure}{.6\textwidth}
\centering
\includegraphics[scale=0.45]{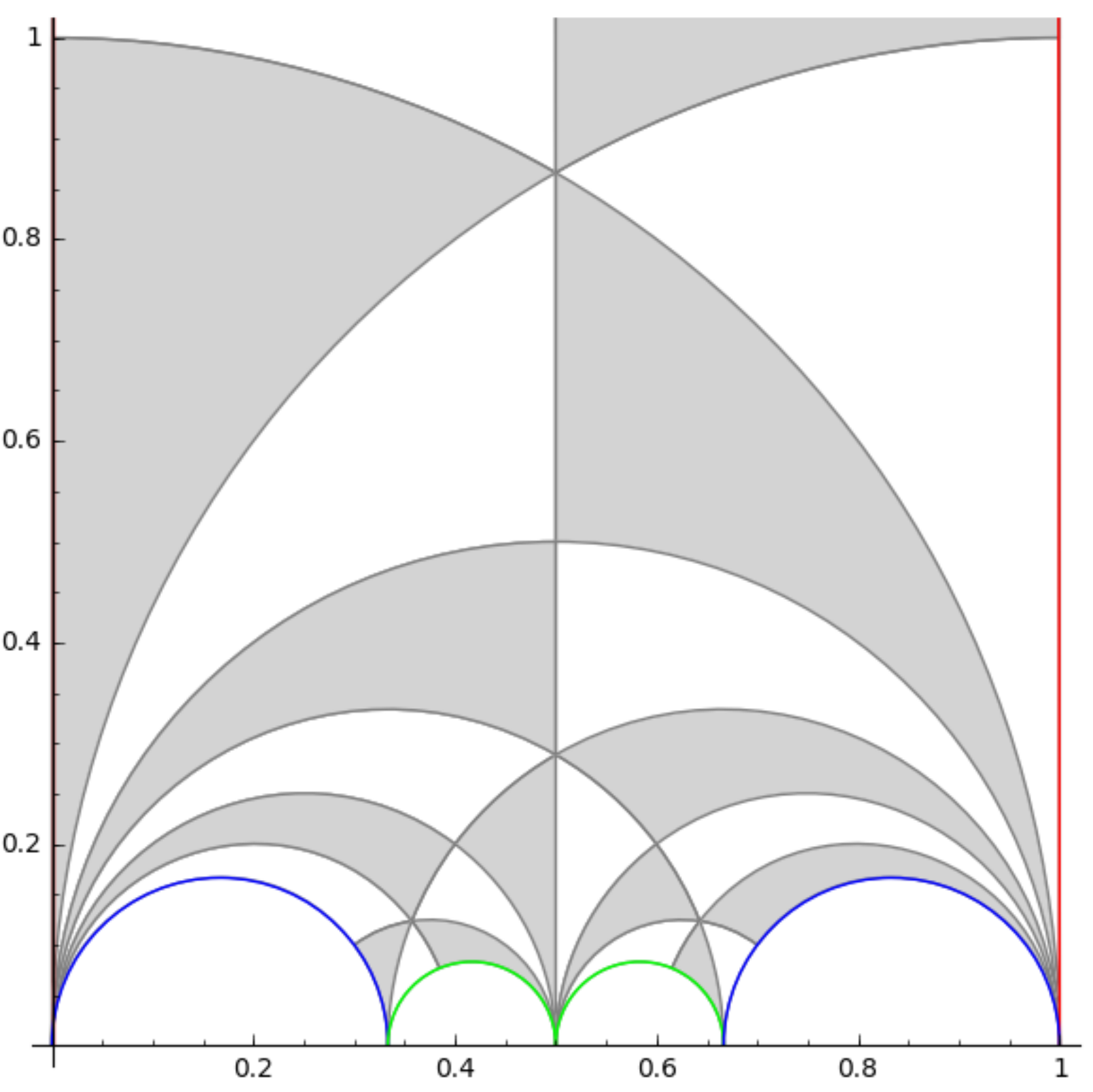}
\end{subfigure}
\begin{subfigure}{.3\textwidth}
\begin{tikzpicture}[scale=1.5]
\draw[fill=black] (0,0) circle (0.1);
\draw[fill=black] (0,-2) circle (0.1);
\draw (0,1) circle (0.1);
\draw (-0.5,-2.5) circle (0.1);
\draw (0.5,-2.5) circle (0.1);
\draw (0,-1) circle (0.1);
\draw (0,-3) circle (0.1);
\draw (0,-4) circle (0.1);
\draw[fill=black] (1,-3) circle (0.1);
\draw[fill=black] (-1,-3) circle (0.1);
\draw (0,0.5) circle(0.5);
\draw (0,-2)--(-1,-3);
\draw (0,-2)--(1,-3);
\draw (0,0)--(0,-0.9);
\draw (0,-1.1)--(0,-2);
\draw (-1,-3)--(1,-3);
\draw (-1,-3) arc(180:360:1);
\node (c) at (-0.5,0.5) {$L_6$};
\node (c) at (0.5,0.5) {$L_{1/6}$};
\node (c) at (0,-1.5) {$L_{1/6,1/6}$};
\node (c) at (-1.3,-2.75) {$L_{2/3,1/3}$};
\node (c) at (1.2,-3.7) {$L_{3/2,1/2}$};
\node (c) at (0.75,-2.25) {$L_{2/3,2/3}$};
\node (c) at (0,-0.5) {$L_{1/6,5/6}$};
\node (c) at (-0.75,-2.25) {$L_{1/6,2/3}$};
\node (c) at (1.3,-2.75) {$L_{1/6,1/3}$};
\node (c) at (-1.2,-3.7) {$L_{1/6,1/2}$};
\node (c) at (-0.6,-3.2) {$L_{3/2}$};
\node (c) at (0.6,-3.2) {$L_{2/3}$};
\end{tikzpicture}
\end{subfigure}
\end{figure}

\begin{figure}[h!]
\centering
\begin{tabular}{|c|c|c|}
\hline
Cusp & Representative & Width \\
\hline
\hline
$([0:1],...,[5:1])=(L_{6},L_{1/6,1/6},L_{2/3,1/3},L_{3/2,1/2},L_{2/3,2/3},L_{1/6,5/6})$ & $0$ & $6$ \\
\hline
$([1:3],[2:3])=(L_{1/6,1/2},L_{2/3})$ & $1/3$ & $2$ \\
\hline
$([1:2],[3:2],[5:2])=(L_{1/6,1/3},L_{3/2},L_{1/6,2/3})$ & $1/2$ & $3$ \\
\hline
$([1:0])=(L_{1/6})$ & $\infty$ & $1$ \\
\hline
\end{tabular}
\end{figure}

\clearpage

\subsection*{$\He(7)$}

The index in $\psl$ is $8$.

\begin{figure}[h!]
\begin{subfigure}{.6\textwidth}
\centering
\includegraphics[scale=0.4]{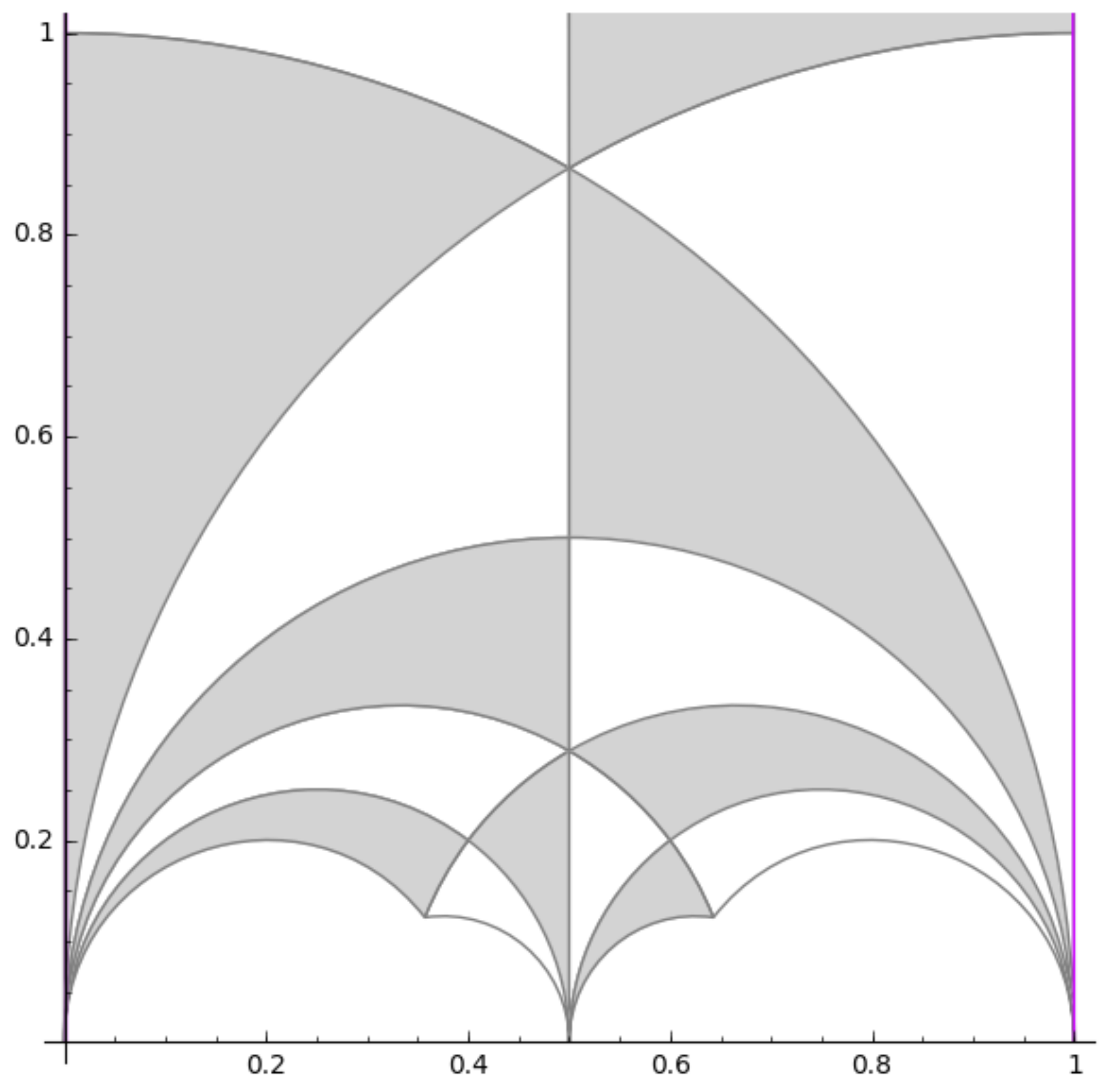}
\end{subfigure}
\begin{subfigure}{.3\textwidth}
\begin{tikzpicture}[scale=1.5]
\draw[fill=black] (0,0) circle (0.1);
\draw[fill=black] (0,-2) circle (0.1);
\draw (0,1) circle (0.1);
\draw (0,-1) circle (0.1);
\draw[fill=black] (1,-3) circle (0.1);
\draw[fill=black] (-1,-3) circle (0.1);
\draw (0,0.5) circle(0.5);
\draw (0.5,-2.5) circle(0.1);
\draw (-0.5,-2.5) circle(0.1);
\draw (0,-2)--(-0.94,-2.94);
\draw (0,-2)--(0.94,-2.94);
\draw (0,0)--(0,-0.9);
\draw (0,-1.1)--(0,-2);
\node (c) at (-0.5,0.5) {$L_7$};
\node (c) at (0.5,0.5) {$L_{1/7}$};
\node (c) at (0,-1.5) {$L_{1/7,1/7}$};
\node (c) at (-1.3,-2.75) {$L_{1/7,4/7}$};
\node (c) at (0.75,-2.25) {$L_{1/7,3/7}$};
\node (c) at (0,-0.5) {$L_{1/7,6/7}$};
\node (c) at (-0.75,-2.25) {$L_{1/7,5/7}$};
\node (c) at (1.3,-2.75) {$L_{1/7,2/7}$};
\end{tikzpicture}
\end{subfigure}
\end{figure}

\begin{figure}[h!]
\centering
\begin{tabular}{|c|c|c|}
\hline
Cusp & Representative & Width \\
\hline
\hline
$([0:1],...,[6:1])=$ & $0$ & $7$ \\
$(L_{7},L_{1/7,1/7},L_{1/7,4/7},L_{1/7,5/7},L_{1/7,2/7},L_{1/7,3/7},L_{1/7,6/7})$ & & \\
\hline
$([1:0])=(L_{1/7})$ & $\infty$ & $1$ \\
\hline
\end{tabular}
\end{figure}

\clearpage

\subsection*{$\He(8)$}

The index in $\psl$ is $12$.

\begin{figure}[h!]
\begin{subfigure}{.6\textwidth}
\centering
\includegraphics[scale=0.5]{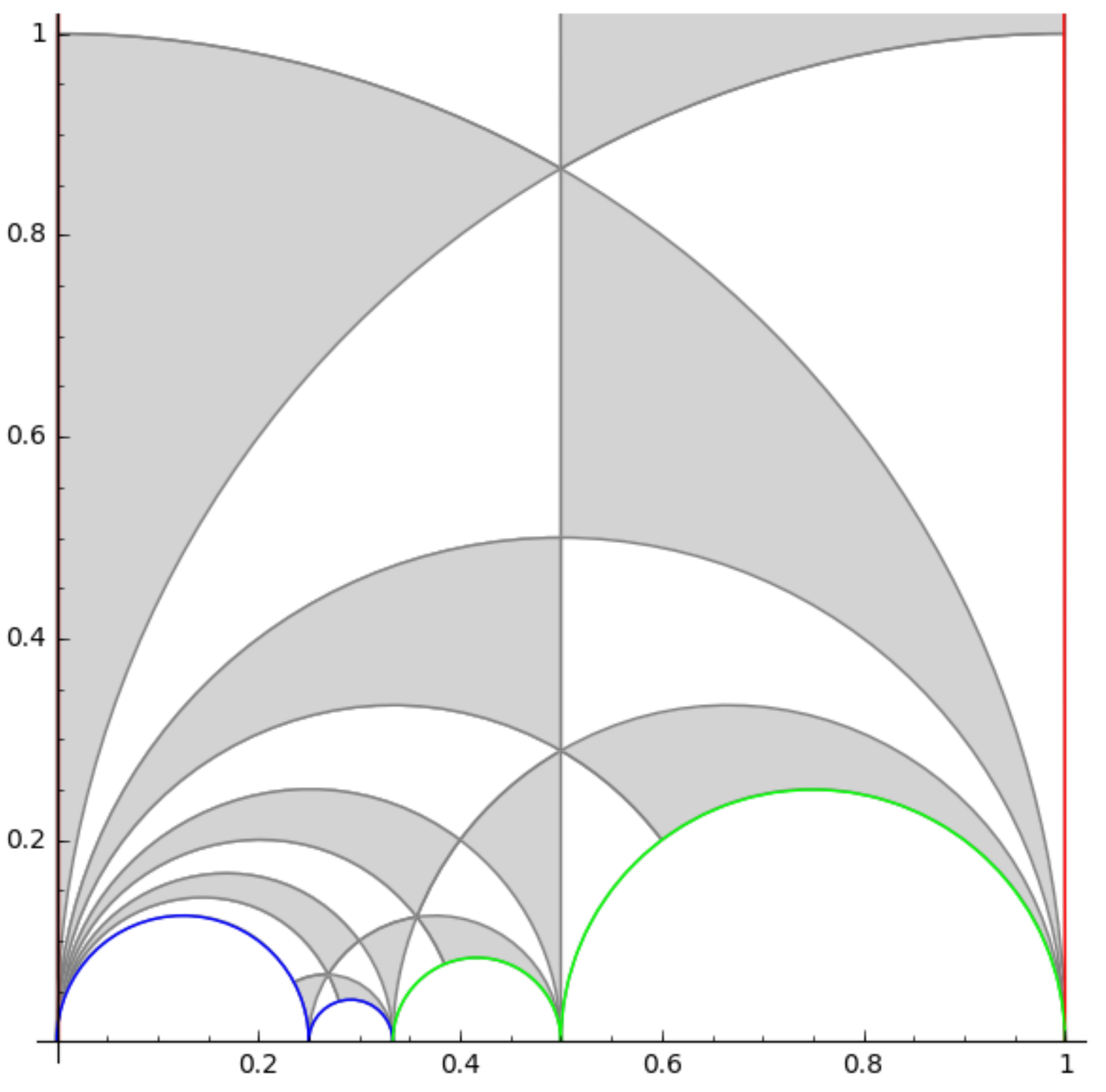}
\end{subfigure}
\begin{subfigure}{.3\textwidth}
\begin{tikzpicture}[scale=1.5]
\draw[fill=black] (0,-1) circle (0.1);
\draw[fill=black] (0,-2) circle (0.1);
\draw (0,0) circle (0.1);
\draw (0,-1.5) circle (0.1);
\draw[fill=black] (0,-3) circle (0.1);
\draw[fill=black] (0,-4) circle (0.1);
\draw (0.5,-2.5) circle (0.1);
\draw (-0.5,-2.5) circle (0.1);
\draw (0,-3.5) circle (0.1);
\draw (0,-5) circle (0.1);
\draw (0,-0.5) circle(0.5);
\draw (0,-2.5) circle(0.5);
\draw (0,-4.5) circle(0.5);
\draw (0,-1)--(0,-2);
\draw (0,-3)--(0,-4);
\node (c) at (-0.5,-0.5) {$L_8$};
\node (c) at (0.5,-0.5) {$L_{1/8}$};
\node (c) at (0,-1.75) {$L_{1/8,1/8}$};
\node (c) at (-0.5,-2.75) {$L_{1/2,1/4}$};
\node (c) at (0,-3.75) {$L_{1/8,3/8}$};
\node (c) at (0.5,-4.5) {$L_{2,1/2}$};
\node (c) at (0,-3.25) {$L_{1/8,5/8}$};
\node (c) at (0.5,-2.25) {$L_{1/2,3/4}$};
\node (c) at (0,-1.25) {$L_{1/8,7/8}$};
\node (c) at (-0.5,-4.5) {$L_{1/8,1/2}$};
\node (c) at (-0.5,-2.25) {$L_{1/8,3/4}$};
\node (c) at (0.5,-2.75) {$L_{1/2,1/4}$};
\end{tikzpicture}
\end{subfigure}
\end{figure}

\begin{figure}[h!]
\centering
\begin{tabular}{|c|c|c|}
\hline
Cusp & Representative & Width \\
\hline
\hline
$([0:1],...,[7:1])=$ & $0$ & $8$ \\
$(L_8,L_{1/8,1/8},L_{1/2,1/4},L_{1/8,3/8},L_{2,1/2},L_{1/8,5/8},L_{1/2,3/4},L_{1/8,7/8})$ & & \\
\hline
$([1:4])=(L_{1/8,1/2})$ & $1/4$ & $1$ \\
\hline
$([1:2],[1:6])=(L_{1/2,1/4},L_{1/8,3/4})$ & $1/2$ & $2$ \\
\hline
$([1:0]=(L_{1/8}))$ & $\infty$ & $1$ \\
\hline
\end{tabular}
\end{figure}

\clearpage

\subsection*{$\He(9)$}

The index in $\psl$ is $12$.

\begin{figure}[h!]
\begin{subfigure}{.5\textwidth}
\centering
\includegraphics[scale=0.4]{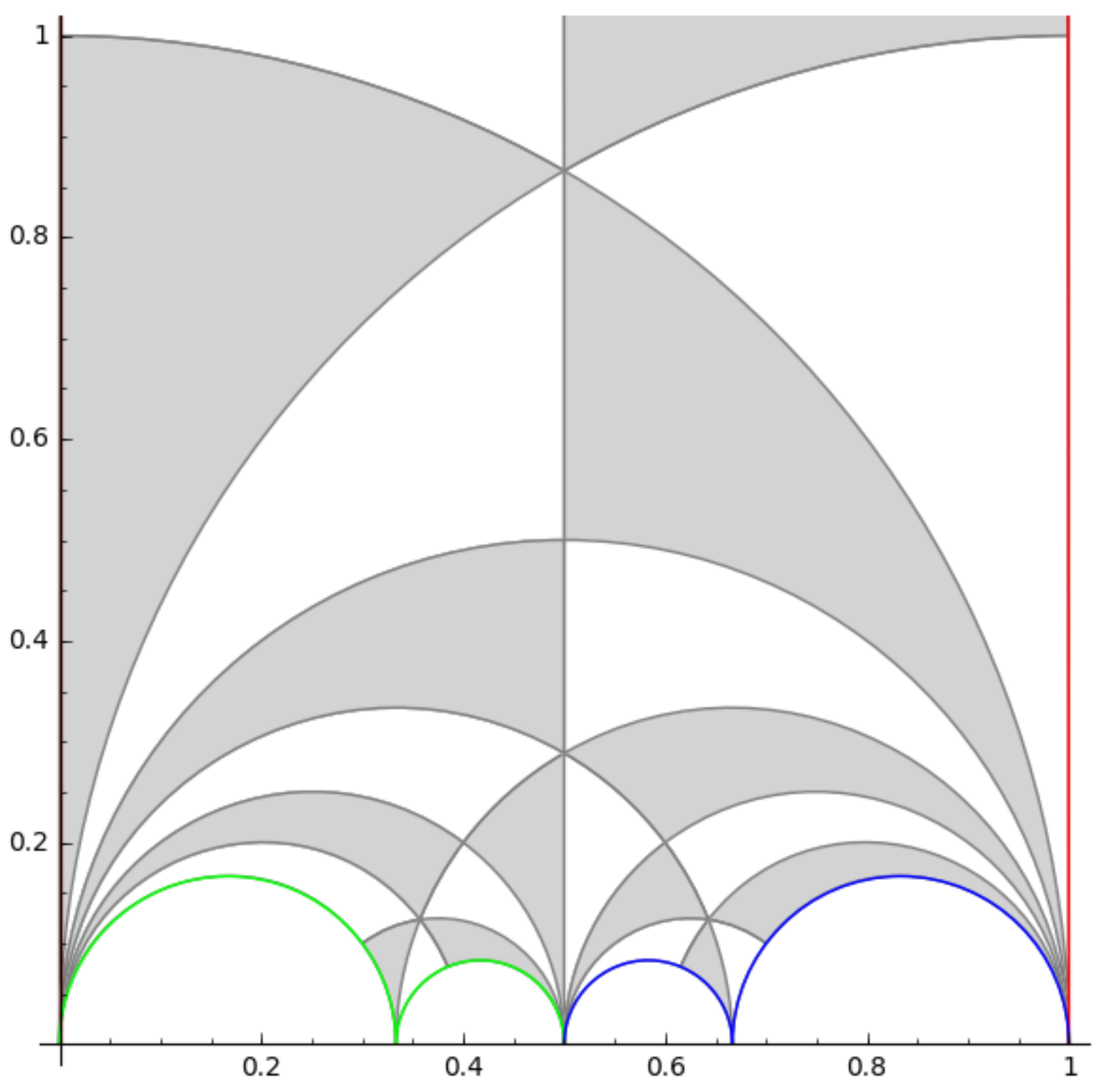}
\end{subfigure}
\begin{subfigure}{.3\textwidth}
\begin{tikzpicture}[scale=2]
\draw[fill=black] (0,1) circle(0.075);
\draw[fill=black] (0,0) circle(0.075);
\draw[fill=black] (-0.87,-0.5) circle(0.075);
\draw[fill=black] (0.87,-0.5) circle(0.075);
\draw (0,0.5) circle(0.075);
\draw (0,2) circle(0.075);
\draw (-0.43,-0.25) circle(0.075);
\draw (0.43,-0.25) circle(0.075);
\draw (-1.72,-1) circle(0.075);
\draw (1.72,-1) circle(0.075);
\draw (0,1.5) circle(0.5);
\draw (-1.3,-0.75) circle(0.5);
\draw (1.3,-0.75) circle(0.5);
\draw (0,0)--(0,1);
\draw (0,0)--(-0.87,-0.5);
\draw (0,0)--(0.87,-0.5);
\node (c) at (-0.5,1.5) {$L_9$};
\node (c) at (0.5,1.5) {$L_{1/9}$};
\node (c) at (0,0.25) {$L_{1/9,1/9}$};
\node (c) at (-0.45,-0.35) {$L_{1/9,5/9}$};
\node (c) at (-0.9,-1) {$L_{1,1/3}$};
\node (c) at (-0.6,-0.05) {$L_{1/9,7/9}$};
\node (c) at (0.45,-0.35) {$L_{1/9,2/9}$};
\node (c) at (1.7,-0.5) {$L_{1,2/3}$};
\node (c) at (0.6,-0.05) {$L_{1/9,4/9}$};
\node (c) at (0,0.75) {$L_{1/9,8/9}$};
\node (c) at (0.9,-1) {$L_{1/9,1/3}$};
\node (c) at (-1.7,-0.5) {$L_{1/9,2/3}$};
\end{tikzpicture}
\end{subfigure}
\end{figure}

\begin{figure}[h!]
\centering
\begin{tabular}{|c|c|c|}
\hline
Cusp & Representative & Width \\
\hline
\hline
$([0:1],...,[8:1])=(L_9,L_{1/9,1/9},L_{1/9,5/9}, L_{1,1/3},$ & $0$ & $9$ \\
$ L_{1/9,7/9}, L_{1/9,2/9}, L_{1,2/3}, L_{1/9,4/9}, L_{1/9,8/9})$ & & \\
\hline
$([1:6])=(L_{1/9,2/3})$ & $1/3$ & $1$ \\
\hline
$([1:3])=(L_{1/9,1/3})$ & $2/3$ & $1$ \\
\hline
$([1:0])=(L_{1/9})$ & $\infty$ & $1$ \\
\hline
\end{tabular}
\end{figure}

\clearpage

\subsection*{$\He(10)$}

The index in $\psl$ is $18$.

\begin{figure}[h!]
\begin{subfigure}{.55\textwidth}
\centering
\includegraphics[scale=0.5]{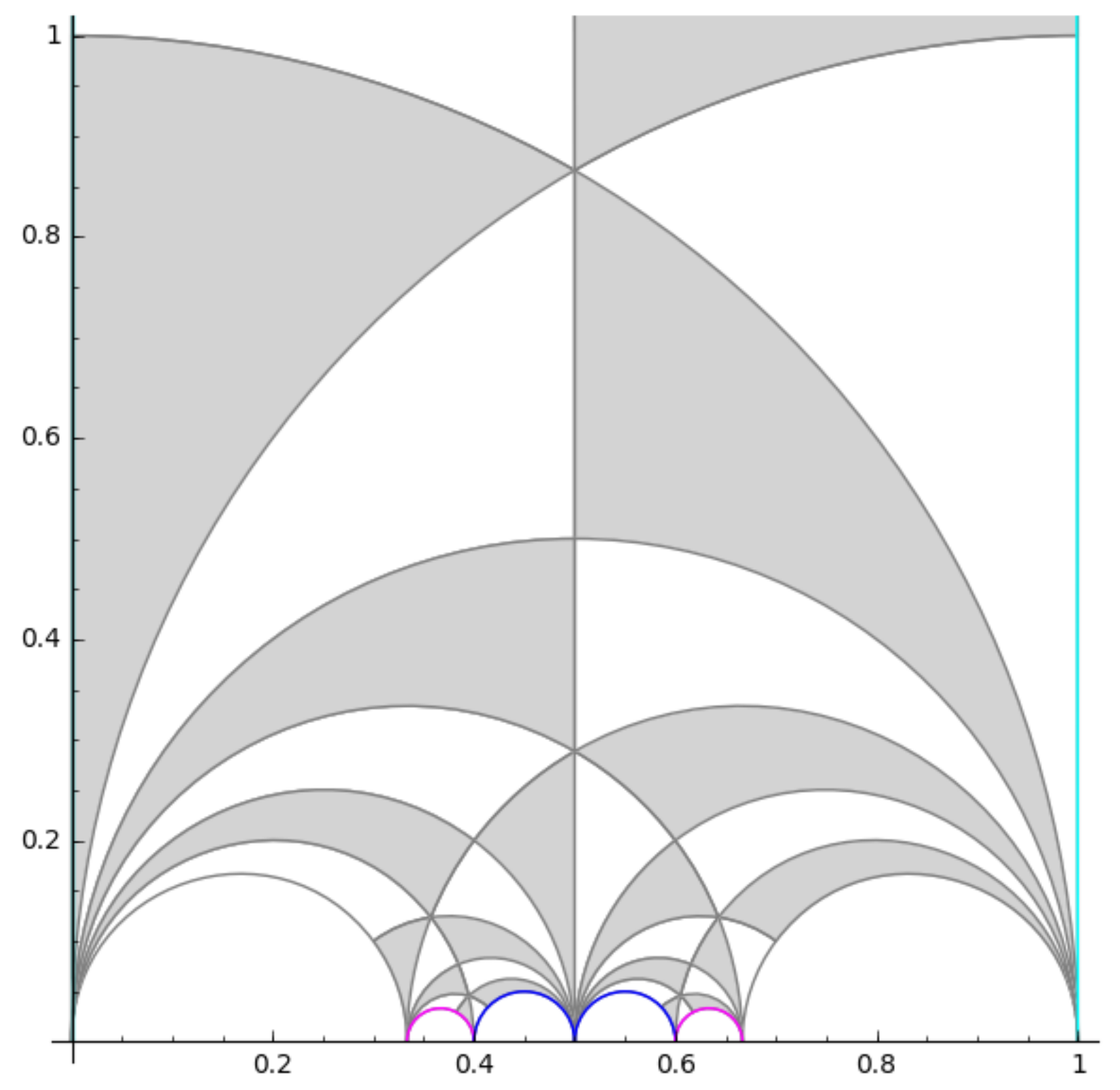}
\end{subfigure}
\begin{subfigure}{.3\textwidth}
\begin{tikzpicture}[scale=2]
\draw[fill=black] (0,1) circle(0.075);
\draw[fill=black] (0,0) circle(0.075);
\draw[fill=black] (-0.87,-0.5) circle(0.075);
\draw[fill=black] (0.87,-0.5) circle(0.075);
\draw[fill=black] (-0.5,-1.5) circle(0.075);
\draw[fill=black] (0.5,-1.5) circle(0.075);
\draw (0,0.5) circle(0.075);
\draw (0,2) circle(0.075);
\draw (-0.43,-0.25) circle(0.075);
\draw (0.43,-0.25) circle(0.075);
\draw (-1.5,-0.5) circle(0.075);
\draw (1.5,-0.5) circle(0.075);
\draw (0,-1.5) circle(0.075);
\draw (0,-2) circle(0.075);
\draw (-0.68,-1) circle(0.075);
\draw (0.68,-1) circle(0.075);
\draw (0,1.5) circle(0.5);
\draw (0,0)--(0,1);
\draw (0,0)--(-0.87,-0.5);
\draw (0,0)--(0.87,-0.5);
\draw (-0.5,-1.5)--(-0.87,-0.5);
\draw (0.5,-1.5)--(0.87,-0.5);
\draw (-0.5,-1.5)--(0.5,-1.5);
\draw (-1.4,-0.5)--(-0.87,-0.5);
\draw (1.4,-0.5)--(0.87,-0.5);
\draw (-0.5,-1.5) arc(180:360:0.5);
\draw (-0.5,1.5) node{$L_{10}$};
\draw (0.5,1.5) node{$L_{1/10}$};
\draw (0.5,-1.2) node{$L_{1/10,2/5}$};
\draw (-0.3,-1.8) node{$L_{1/10,1/2}$};
\end{tikzpicture}
\end{subfigure}
\end{figure}

\begin{figure}[h!]
\centering
\begin{tabular}{|c|c|c|}
\hline
Cusp & Representative & Width \\
\hline
\hline
$([0:1],...,[8:1])=(L_{10},L_{1/10,1/10},L_{2/5,1/5},L_{1/10,7/10},$ & $0$ & $10$ \\
$L_{2/5,3/5},L_{5/2,1/2},L_{2/5,2/5},L_{1/10,3/10},L_{2/5,4/5},L_{1/10,9/10}$ & & \\
\hline
$([1:5],[2:5])=(L_{1/10,1/2},L_{2/5})$ & $1/3$ & $2$ \\
\hline
$([1:2],[1:4],[5:2],[1:6],[1:8])=$ & $1/2$ & $5$ \\
$(L_{1/10,1/5},L_{1/10,2/5},L_{5/2},L_{1/10,3/5},L_{1/10,4/5})$ & & \\
\hline
$([1:0])=(L_{1/10})$ & $\infty$ & $1$ \\
\hline
\end{tabular}
\end{figure}

\clearpage

\subsection*{$\He(12)$}

The index in $\psl$ is $24$.

\begin{figure}[h!]
\begin{subfigure}{.7\textwidth}
\centering
\includegraphics[scale=0.5]{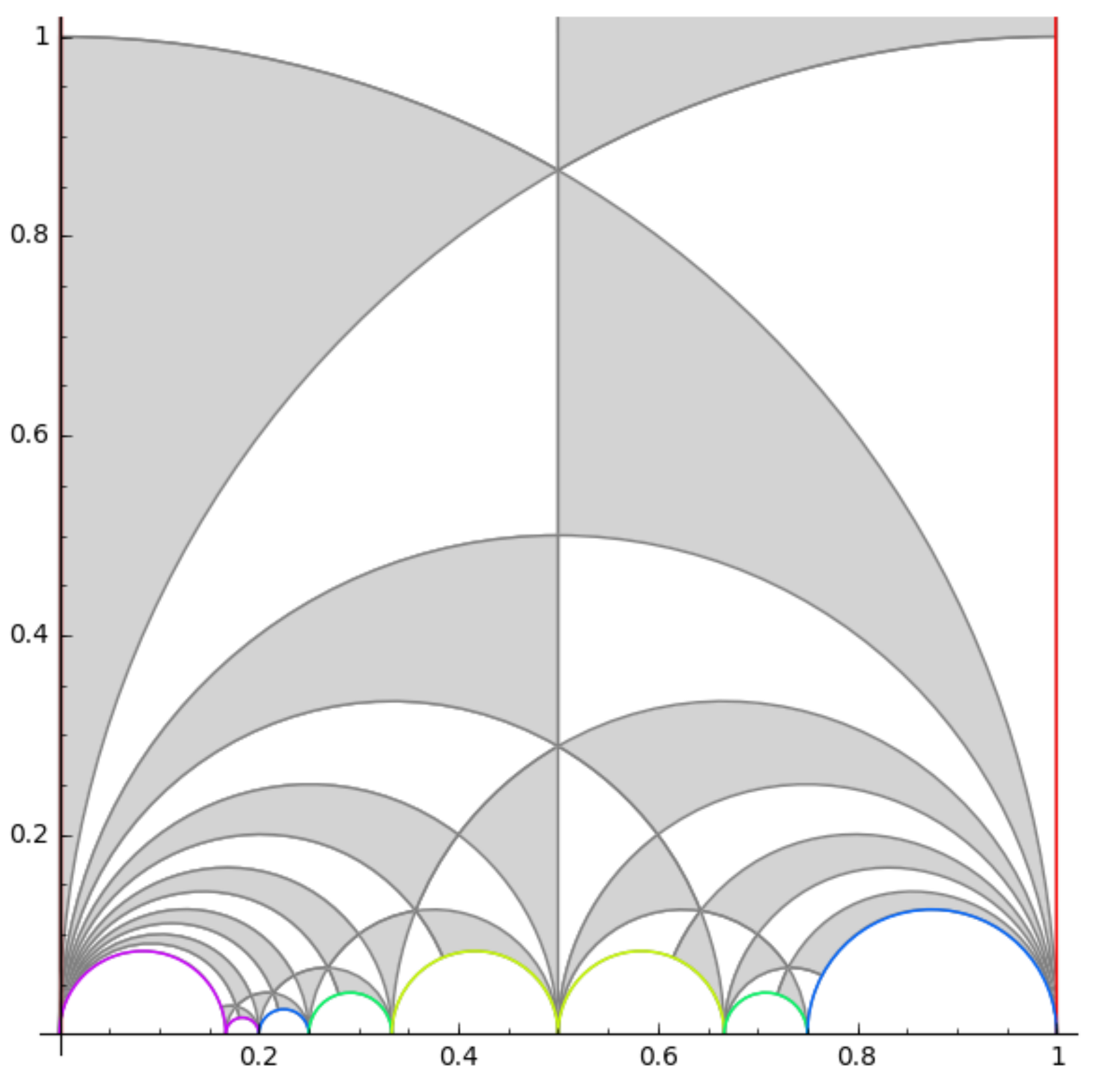}
\end{subfigure}
\begin{subfigure}{.2\textwidth}
\begin{tikzpicture}[scale=2]
\draw[fill=black] (-0.5,0) circle (0.075);
\draw[fill=black] (0.5,0) circle (0.075);
\draw[fill=black] (0,1) circle (0.075);
\draw[fill=black] (0,2) circle (0.075);
\draw[fill=black] (0,-2) circle (0.075);
\draw[fill=black] (0,-3) circle (0.075);
\draw[fill=black] (-0.5,-1) circle (0.075);
\draw[fill=black] (0.5,-1) circle (0.075);
\draw (0,3) circle (0.075);
\draw (0,1.5) circle (0.075);
\draw (-0.25,0.5) circle (0.075);
\draw (0.25,0.5) circle (0.075);
\draw (0,0) circle (0.075);
\draw (-0.5,-0.5) circle (0.075);
\draw (0.5,-0.5) circle (0.075);
\draw (0,-2.5) circle (0.075);
\draw (-0.25,-1.5) circle (0.075);
\draw (0.25,-1.5) circle (0.075);
\draw (0,-1) circle (0.075);
\draw (0,-4) circle (0.075);
\draw (0,2.5) circle (0.5);
\draw (0,-3.5) circle (0.5);
\draw (0,2)--(0,1);
\draw (0,-3)--(0,-2);
\draw (0,1)--(-0.5,0);
\draw (0,1)--(0.5,0);
\draw (0.5,0)--(-0.5,0);
\draw (0.5,0)--(0.5,-1);
\draw (-0.5,0)--(-0.5,-1);
\draw (-0.5,-1)--(0.5,-1);
\draw (-0.5,-1)--(0,-2);
\draw (0.5,-1)--(0,-2);
\draw (-0.5,2.5) node{$L_{12}$};
\draw (0.5,2.5) node{$L_{1/12}$};
\draw (0.4,0.25) node{$L_{1/12,1/6}$};
\draw (0.4,-0.75) node{$L_{1/12,1/4}$};
\draw (-0.3,-1.75) node{$L_{1/12,2/3}$};
\draw (0.5,-3.5) node{$L_{1/12,1/2}$};
\end{tikzpicture}
\end{subfigure}
\end{figure}

\begin{figure}[h!]
\centering
\begin{tabular}{|c|c|c|}
\hline
Cusp & Representative & Width \\
\hline
\hline
$([0:1],...,[11:1])=(L_{12},L_{1/12,1/12},L_{1/3,1/6},L_{3/4,1/4},L_{4/3,1/3}$ & $0$ & $12$ \\
$L_{1/12,5/12}, L_{3,1/2},L_{1/12,7/12},L_{4/3},L_{2/3},L_{3/4,3/4},L_{1/3,5/6},L_{1/12,11/12})$ & & \\
\hline
$([1:6])=(L_{1/12,1/2})$ & $1/6$ & $1$ \\
\hline
$([1:8],[3:8],[1:4])=(L_{1/12,2/3},L_{3/4},L_{1/12,1/3})$ & $1/4$ & $3$ \\
\hline
$([1:9],[2:9],[1:3],[4:3])=(L_{1/12,3/4},L_{1/3,1/2},L_{1/12,1/4},L_{4/3})$ & $1/3$ & $4$ \\
\hline
$([1:2],[3:2],[5:2])=(L_{1/12,1/6},L_{3/4,1/2},L_{1/12,5/6})$ & $1/2$ & $3$ \\
\hline 
$([1:0])=(L_{1/12})$ & $\infty$ & $1$ \\
\hline
\end{tabular}
\end{figure}

\clearpage

\subsection*{$\He(13)$}

The index in $\psl$ is $14$.

\begin{figure}[h!]
\begin{subfigure}{.5\textwidth}
\centering
\includegraphics[scale=0.5]{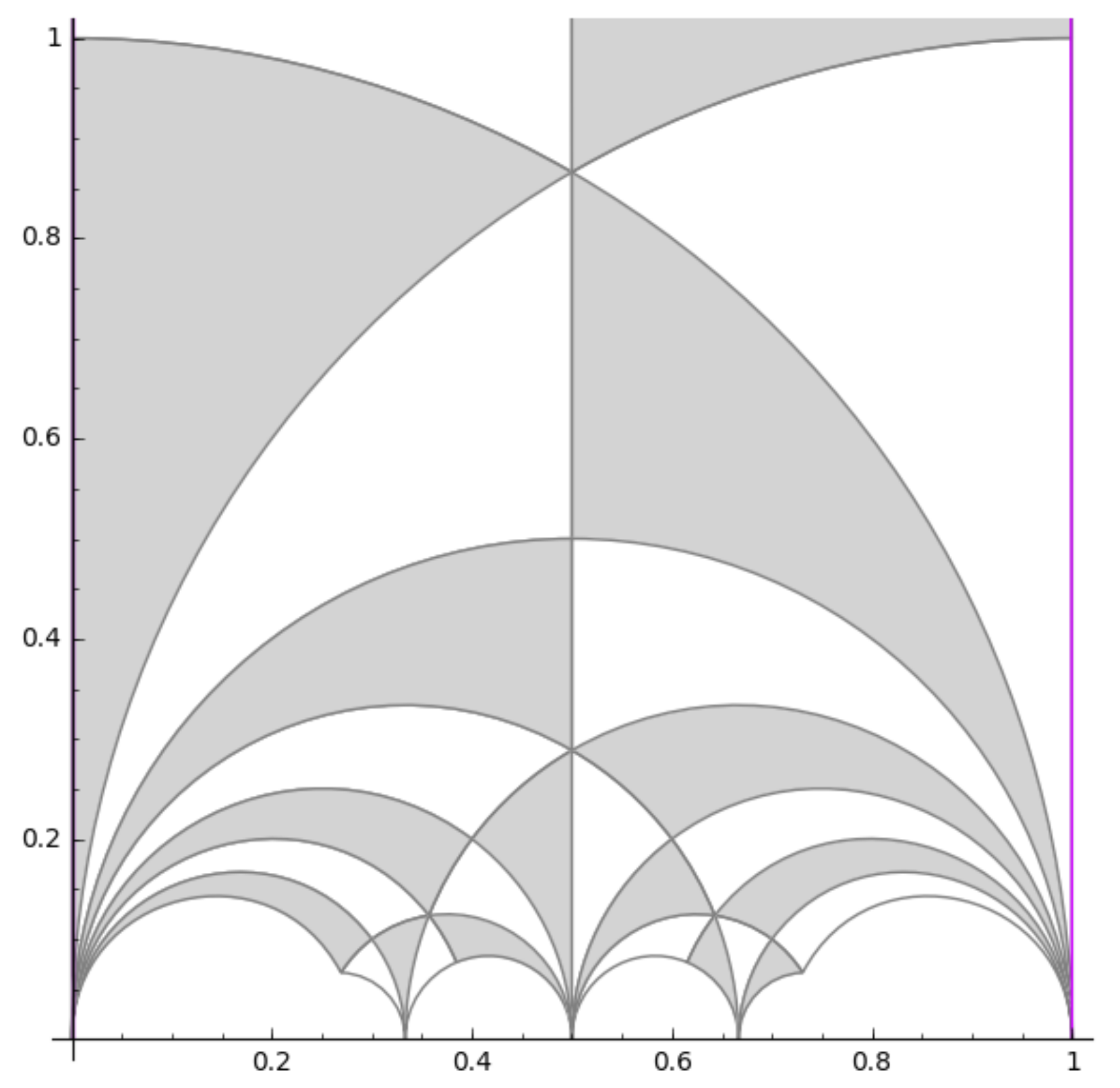}
\end{subfigure}
\begin{subfigure}{.3\textwidth}
\begin{tikzpicture}[scale=2]
\draw[fill=black] (0,1) circle (0.075);
\draw[fill=black] (0,0) circle (0.075);
\draw[fill=black] (-1,-1) circle (0.075);
\draw[fill=black] (1,-1) circle (0.075);
\draw[fill=black] (-2,-2) circle (0.075);
\draw[fill=black] (2,-2) circle (0.075);
\draw (0,2) circle (0.075);
\draw (0,0.5) circle (0.075);
\draw (-0.5,-0.5) circle (0.075);
\draw (0.5,-0.5) circle (0.075);
\draw (-1.5,-1.5) circle (0.075);
\draw (-0.5,-1.5) circle (0.075);
\draw (0.5,-1.5) circle (0.075);
\draw (1.5,-1.5) circle (0.075);
\draw (0,1.5) circle (0.5);
\draw (0,0)--(0,1);
\draw (0,0)--(-1,-1);
\draw (0,0)--(1,-1);
\draw (-1,-1)--(-2,-2);
\draw (-1,-1)--(-0.525,-1.475);
\draw (1,-1)--(2,-2);
\draw (1,-1)--(0.525,-1.475);
\draw (-0.5,1.5) node{$L_{13}$};
\draw (0.5,1.5) node{$L_{1/13}$};
\end{tikzpicture}
\end{subfigure}
\end{figure}

\begin{figure}[h!]
\centering
\begin{tabular}{|c|c|c|}
\hline
Cusp & Representative & Width \\
\hline
\hline
$([0:1],...,[12:1])=(L_{13},L_{1/13,1/13},L_{1/13,7/13},$ & $0$ & $13$ \\
$L_{1/13,9/13},L_{1/13,10/13},L_{1/13,8/13},L_{1/13,11/13},L_{1/13,2/13},L_{1/13,5/13},$ & & \\
$L_{1/13,3/13},L_{1/13,4/13},L_{1/13,6/13},L_{1/13,12/13})$ & & \\
\hline
$([1:0])=(L_{1/13})$ & $\infty$ & $1$ \\
\hline
\end{tabular}
\end{figure}

\subsection*{$\He(16)$}

The index in $\psl$ is $24$.

\begin{figure}[h!]
\centering
\includegraphics[scale=0.5]{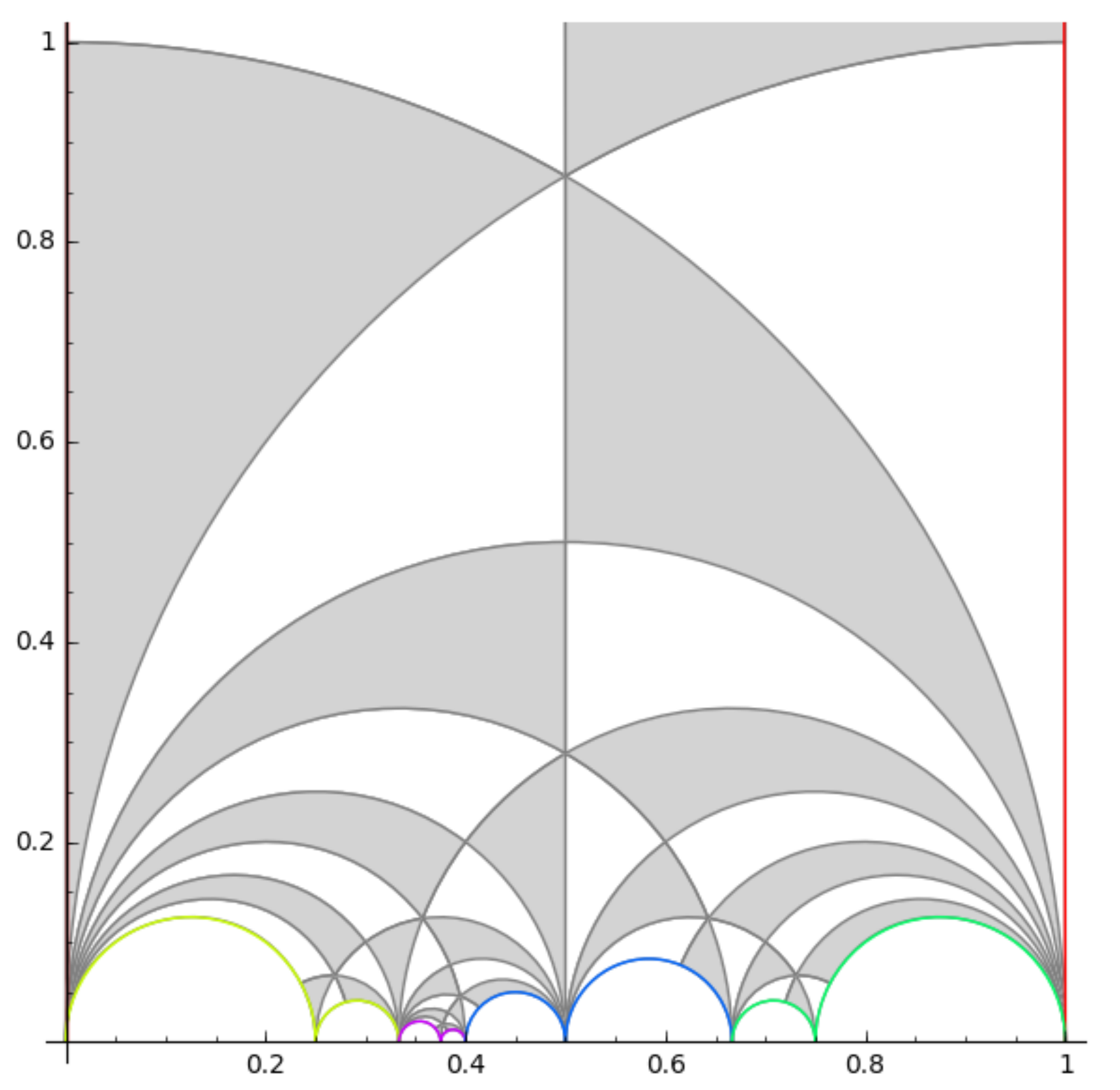}
\end{figure}
\begin{figure}[h!]
\centering
\begin{tikzpicture}[scale=2]
\draw[fill=black] (0,2) circle (0.075);
\draw[fill=black] (0,1) circle (0.075);
\draw[fill=black] (-2,0) circle (0.075);
\draw[fill=black] (-1,0) circle (0.075);
\draw[fill=black] (1,0) circle (0.075);
\draw[fill=black] (2,0) circle (0.075);
\draw[fill=black] (0,-1) circle (0.075);
\draw[fill=black] (0,-2) circle (0.075);
\draw (0,3) circle (0.075);
\draw (0,1.5) circle (0.075);
\draw (-3,0) circle (0.075);
\draw (-1.5,0) circle (0.075);
\draw (1.5,0) circle (0.075);
\draw (3,0) circle (0.075);
\draw (-0.5,0.5) circle (0.075);
\draw (0.5,0.5) circle (0.075);
\draw (-0.5,-0.5) circle (0.075);
\draw (0.5,-0.5) circle (0.075);
\draw (0,-1.5) circle (0.075);
\draw (0,-3) circle (0.075);
\draw (0,2.5) circle (0.5);
\draw (-2.5,0) circle (0.5);
\draw (2.5,0) circle (0.5);
\draw (0,-2.5) circle (0.5);
\draw (0,1)--(0,2);
\draw (0,-1)--(0,-2);
\draw (-2,0)--(-1,0);
\draw (1,0)--(2,0);
\draw (0,1)--(-1,0);
\draw (0,1)--(1,0);
\draw (0,-1)--(-1,0);
\draw (0,-1)--(1,0);
\draw (-0.5,2.5) node{$L_{16}$};
\draw (0.5,2.5) node{$L_{1/16}$};
\draw (2.5,-0.5) node{$L_{1/16,1/4}$};
\draw (-0.5,-2.5) node{$L_{1/16,1/2}$};
\draw (-2.5,0.5) node{$L_{1/16,3/4}$};
\end{tikzpicture}
\end{figure}

\begin{figure}[h!]
\centering
\begin{tabular}{|c|c|c|}
\hline
Cusp & Representative & Width \\
\hline
\hline
$([0:1],...,[15:1])=(L_{16},L_{1/16,1/16},L_{1/4,1/8},L_{1/16,11/16},$ & $0$ & $16$ \\
$L_{1,1/4},L_{1/16,13/16},L_{1/4,3/8},L_{1/16,7/16},L_{4,1/2},L_{1/16,9/16}$ & & \\
$L_{1/4,5/8},L_{1/16,3/16},L_{1,3/4},L_{1/16,5/16},L_{1/4,7/8},L_{1/16,15/16})$ & & \\
\hline
$([1:4])=(L_{1/16,1/4})$ & $1/4$ & $1$ \\
\hline
$([1:8])=(L_{1/16,1/2})$ & $3/8$ & $1$ \\
\hline
$([1:2],[3:2],[5:2],[7:2])=(L_{1/16,1/8},L_{1/16,3/8},L_{1/16,5/8},L_{1/16,7/8})$ & $1/2$ & $4$ \\
\hline
$([1:12])=(L_{1/16,3/4})$ & $3/4$ & $1$ \\
\hline
$([1:0])=(L_{1/16})$ & $\infty$ & $1$ \\
\hline
\end{tabular}
\end{figure}

\clearpage

\subsection*{$\He(18)$}

The index in $\psl$ is $36$.

\begin{figure}[h!]
\begin{subfigure}{.4\textwidth}
\centering
\includegraphics[scale=0.4]{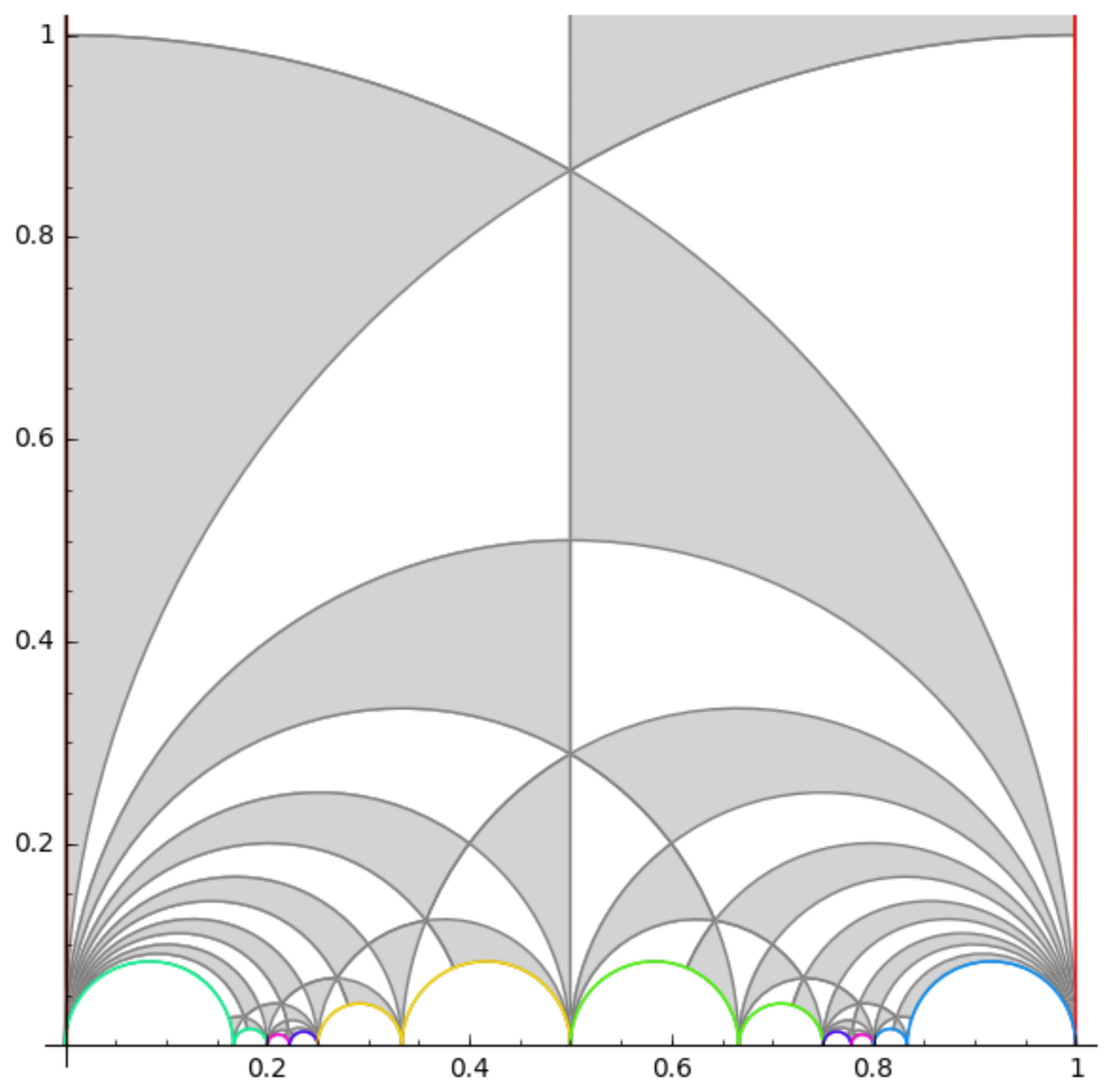}
\end{subfigure}
\begin{subfigure}{.4\textwidth}
\begin{tikzpicture}[scale=2]
\draw[fill=black] (0,1) circle(0.075);
\draw[fill=black] (-0.64,0.77) circle(0.075);
\draw[fill=black] (-0.98,0.17) circle(0.075);
\draw[fill=black] (-0.87,-0.5) circle(0.075);
\draw[fill=black] (-0.34,-0.94) circle(0.075);
\draw[fill=black] (0.34,-0.94) circle(0.075);
\draw[fill=black] (0.87,-0.5) circle(0.075);
\draw[fill=black] (0.98,0.17) circle(0.075);
\draw[fill=black] (0.64,0.77) circle(0.075);
\draw[fill=black] (0,2) circle(0.075);
\draw[fill=black] (-1.74,-1) circle(0.075);
\draw[fill=black] (1.74,-1) circle(0.075);
\draw (0,3) circle (0.075);
\draw (-2.61,-1.5) circle (0.075);
\draw (2.61,-1.5) circle (0.075);
\draw (0,1.5) circle (0.075);
\draw (-1.31,-0.75) circle (0.075);
\draw (1.31,-0.75) circle (0.075);
\draw (-0.32,0.89) circle (0.075);
\draw (0.32,0.89) circle (0.075);
\draw (-0.81,0.47) circle (0.075);
\draw (0.81,0.47) circle (0.075);
\draw (-0.93,-0.17) circle (0.075);
\draw (0.93,-0.17) circle (0.075);
\draw (-0.61,-0.73) circle (0.075);
\draw (0.61,-0.73) circle (0.075);
\draw (0,-0.94) circle (0.075);
\draw (-1.05,0.62) circle (0.075);
\draw (1.05,0.62) circle (0.075);
\draw (0,-1.26) circle (0.075);
\draw (0,2.5) circle (0.5);
\draw (-2.17,-1.25) circle (0.5);
\draw (2.17,-1.25) circle (0.5);
\draw (0,1)--(0,2);
\draw (-0.87,-0.5)--(-1.74,-1);
\draw (0.87,-0.5)--(1.74,-1);
\draw (0,1)--(-0.64,0.77);
\draw (-0.64,0.77)--(-0.98,0.17);
\draw (-0.98,0.17)--(-0.87,-0.5);
\draw (-0.87,-0.5)--(-0.34,-0.94);
\draw (-0.34,-0.94)--(0.34,-0.94);
\draw (0.34,-0.94)--(0.87,-0.5);
\draw (0.87,-0.5)--(0.98,0.17);
\draw (0.98,0.17)--(0.64,0.77);
\draw (0.64,0.77)--(0,1);
\draw (-0.64,0.77) arc(70:250:0.34);
\draw (0.98,0.17) arc(-50:100:0.34);
\draw (-0.34,-0.94) arc(180:360:0.34);
\draw (-0.5,2.5) node{$L_{18}$};
\draw (0.5,2.5) node{$L_{1/18}$};
\draw (1.4,0.3) node{$L_{1/18,1/6}$};
\draw (1.8,-1.5) node{$L_{1/18,1/3}$};
\draw (-0.5,-1.1) node{$L_{1/18,1/2}$};
\draw (-2.2,-0.8) node{$L_{1/18,2/3}$};
\draw (-1.1,0.8) node{$L_{1/18,5/6}$};
\draw (0.9,-0.3) node{$L_{1/18,2/9}$};
\end{tikzpicture}
\end{subfigure}
\end{figure}

\begin{figure}[h!]
\centering
\begin{tabular}{|c|c|c|}
\hline
Cusp & Representative & Width \\
\hline
\hline
$([0:1],...,[18:1])=(L_{18},L_{1/18,1/18},L_{2/9,1/9},L_{1/2,1/6},$ & $0$ & $18$ \\
$L_{2/9,5/9},L_{1/18,11/18},L_{2,1/3},L_{1/18,13/18},L_{2/9,7/9},L_{9/2,1/2},L_{2/9,2/9}$ & & \\
$L_{1/18,5/18},L_{2,2/3},L_{1/18,7/18},L_{2/9,4/9},L_{1/2,5/6},L_{2/9,8/9},L_{1/18,17/18})$ & & \\
\hline
$([1:12])=(L_{1/18,2/3})$ & $1/6$ & $1$ \\
\hline
$([1:9],[2:9])=(L_{1/18,1/2},L_{2/9})$ & $2/9$ & $2$ \\
\hline
$([1:2],[3:2],[5:2],[7:2],[9:2],[11:2],[13:2],$ & $1/4$ & $9$ \\
$[15:2],[17:2])=(L_{1/18,1/9},L_{1/2,1/3},L_{1/18,2/9},L_{1/18,4/9},L_{9/2},$ & & \\
$L_{1/18,5/9},L_{1/18,7/9},L_{1/2,2/3},L_{1/18,8/9})$ & & \\
\hline
$([2:3],[5:3])=(L_{2/9,1/3},L_{1/18,5/6})$ & $1/3$ & $2$ \\
\hline
$([1:3],[4:3])=(L_{1/18,1/6},L_{2/9,2/3})$ & $2/3$ & $2$ \\
\hline
$([1:6])=(L_{1/18,1/3})$ & $5/6$ & $1$ \\
\hline
$([1:0])=(L_{1/18})$ & $\infty$ & $1$ \\
\hline
\end{tabular}
\end{figure}

\clearpage

\subsection*{$\He(25)$}

The index in $\psl$ is $30$.

\begin{figure}[h!]
\centering
\includegraphics[scale=0.37]{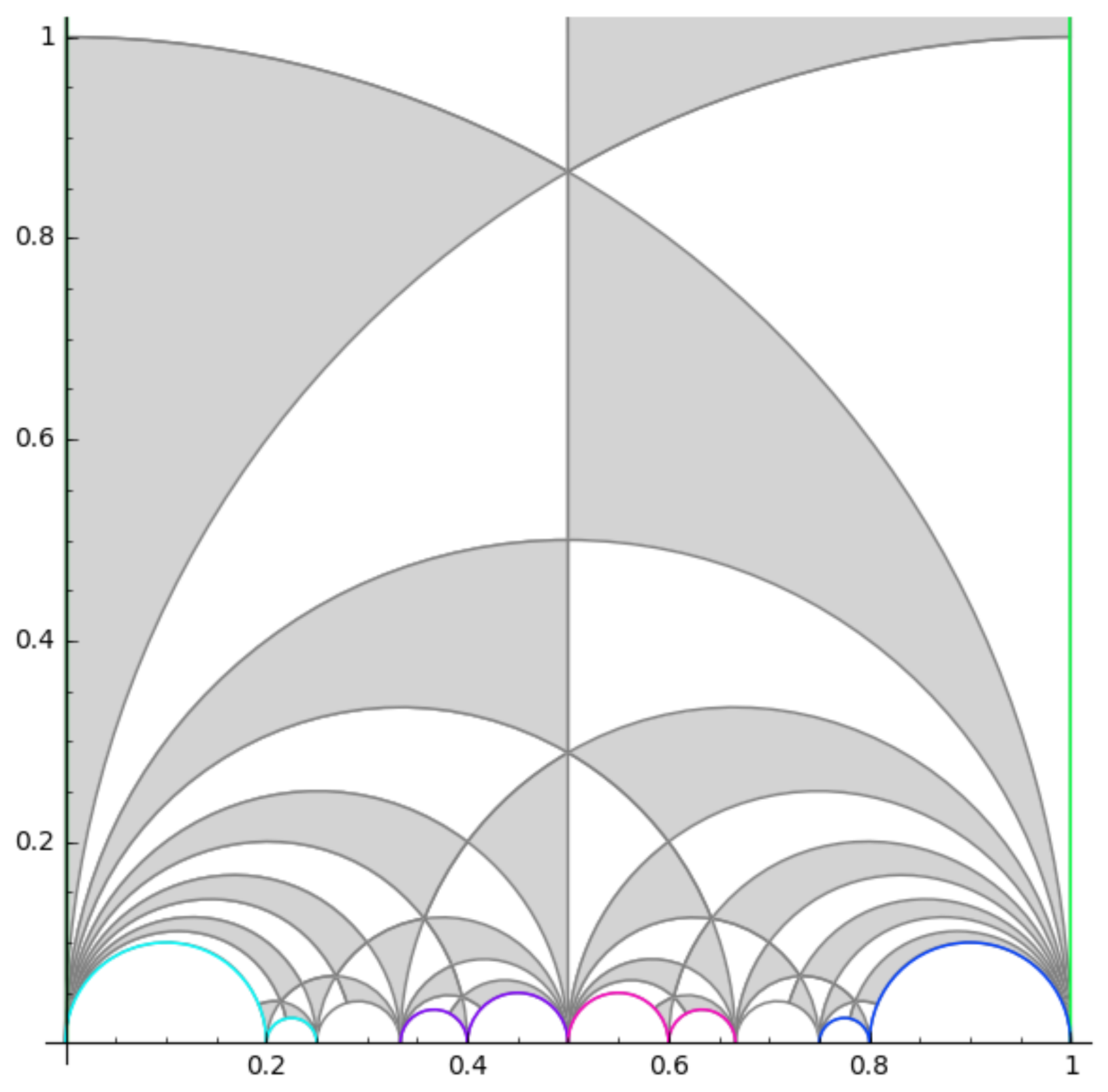}
\end{figure}

\begin{figure}[h!]
\centering
\begin{tikzpicture}[scale=1.5]
\draw[fill=black] (0,1) circle (0.1);
\draw[fill=black] (0,0) circle (0.1);
\draw[fill=black] (-2,-1) circle (0.1);
\draw[fill=black] (2,-1) circle (0.1);
\draw[fill=black] (-3,-2) circle (0.1);
\draw[fill=black] (-1,-2) circle (0.1);
\draw[fill=black] (1,-2) circle (0.1);
\draw[fill=black] (3,-2) circle (0.1);
\draw[fill=black] (-4,-3) circle (0.1);
\draw[fill=black] (4,-3) circle (0.1);
\draw (0,2) circle (0.1);
\draw (0,0.5) circle (0.1);
\draw (-1,-0.5) circle (0.1);
\draw (1,-0.5) circle (0.1);
\draw (-1.5,-1.5) circle (0.1);
\draw (-2.5,-1.5) circle (0.1);
\draw (2.5,-1.5) circle (0.1);
\draw (1.5,-1.5) circle (0.1);
\draw (-3.5,-2.5) circle (0.1);
\draw (-2.5,-2.5) circle (0.1);
\draw (3.5,-2.5) circle (0.1);
\draw (2.5,-2.5) circle (0.1);
\draw (-1,-3) circle (0.1);
\draw (1,-3) circle (0.1);
\draw (-4,-4) circle (0.1);
\draw (4,-4) circle (0.1);
\draw (0,1.5) circle (0.5);
\draw (-4,-3.5) circle (0.5);
\draw (4,-3.5) circle (0.5);
\draw (-1,-2.5) circle (0.5);
\draw (1,-2.5) circle (0.5);
\draw (0,0)--(0,1);
\draw (0,0)--(-2,-1);
\draw (0,0)--(2,-1);
\draw (-2,-1)--(-4,-3);
\draw (2,-1)--(4,-3);
\draw (-2,-1)--(-1,-2);
\draw (2,-1)--(1,-2);
\draw (-3,-2)--(-2.5,-2.5);
\draw (3,-2)--(2.5,-2.5);
\draw (-0.5,1.5) node{$L_{25}$};
\draw (0.5,1.5) node{$L_{1/25}$};
\draw (3.5,-3.5) node{$L_{1/25,1/5}$};
\draw (0.4,-2.5) node{$L_{1/25,2/5}$};
\draw (-1.5,-2.5) node{$L_{1/25,3/5}$};
\draw (-4.5,-3.5) node{$L_{1/25,4/5}$};
\end{tikzpicture}
\end{figure}

\begin{figure}[h!]
\centering
\begin{tabular}{|c|c|c|}
\hline
Cusp & Representative & Width \\
\hline
\hline
$([0:1],...,[24:1])=(L_{25}, L_{1/25,1/25},L_{1/25,13/25}, L_{1/25,17/25},$& & \\
$L_{1/25,19/25}, L_{1,1/5}, L_{1/25,21/25}, L_{1/25,18/25}, L_{1/25,22/25}, L_{1/25,14/25}, L_{1,3/5},$ & $0$ & $25$ \\
$L_{1/25,16/25}, L_{1/25,23/25}, L_{1/25,2/25}, L_{1/25,9/25}, L_{1,2/5}, L_{1/25,11/25}, L_{1/25,3/25},$ & & \\
$ L_{1/25,7/25}, L_{1/25,4/25}, L_{1,4/5}, L_{1/25,6/25}, L_{1/25,8/25}, L_{1/25,12/25}, L_{1/25,24/25}$ & & \\
\hline
$([1:20])=(L_{1/25,4/5})$ & $1/5$ & $1$ \\
\hline
$([1:15])=(L_{1/25,3/5})$ & $2/5$ & $1$ \\
\hline
$([1:10])=(L_{1/25,2/5})$ & $3/5$ & $1$ \\
\hline
$([1:5])=(L_{1/25,1/5})$ & $4/5$ & $1$ \\
\hline
$([1:0])=(L_{1/25})$ & $\infty$ & $1$ \\
\hline
\end{tabular}
\end{figure}

\clearpage

**\appendix

\section{Lattices and Hecke groups}\label{app}

We present here the approach to arithmetic groups developped by Conway in \cite{conway}, in terms of their action on lattices. The modular group $\psl$ and its Hecke congruence subgroups $(\He(N))_{N\geq1}$ naturally appear as stabilisers in $\PGL_2^+(\mathbb{Q})$ of a pair of projective lattices in a $2$-dimensional real vector space. 

We closely follow the first sections \cite{duncan} which fit our purposes well, and even restrict to two dimensions. The article \cite{plazas} ties a link between this and non-commutative geometry systems as developed by Marcolli and Connes.

\subsection{Linear transformations}
Let $V$ be a two-dimensional vector space over $\mathbb{R}$, with a basis $(e^1,e^2)$ refered to as the {\bf reference basis} is what follows, and fixed throughout our paper. 
A vector $v\in V$ is written as a row of two coordinates (generically denoted $v_1$ and $v_2$) with the basis specified when needed. 
For example, in the reference basis
\begin{equation}
v=(v_1\ v_2)=\sum_{i=1}^2v_ie^i \ .
\end{equation}

\begin{definition}
Let $f^1,f^2$ be two vectors in $V$. 
The pair $(f^1,f^2)$ is an oriented basis of $V$ if $f^1\wedge f^2$ is a strictly positive multiple of $e^1\wedge e^2$. 
Let $\Bp \subset V^2$ be the subset of oriented bases. In what follows, oriented bases $(f^1,f^2)\in\Bp$ are written as $2\times1$ matrices of vectors in $V$.
\end{definition}

The ring $\en(V)$ of endomorphisms of $V$ acts naturally on $V$ on the right: 
\begin{equation}
\left.\begin{array}{rcl}V\times\en(V) & \rightarrow & V \\
(v,A) & \mapsto & v\cdot A \end{array}\right. \ .
\end{equation}
This action induces a right-action of $\en(V)$ on $V^n$. 

The reference basis induces an isomorphism $\en(V)\simeq \mathcal{M}_2(\mathbb{R})$. Let $v\in V$ with expression $(v_1\ v_2)$ in the reference basis. A matrix $M\in\mathcal{M}_2(\mathbb{R})$ acts on $V$ as:
\begin{equation}
v\cdot M=(v_1\ v_2)\cdot\left(\begin{array}{cc} a & b \\ c & d \end{array}\right)=(av_1+cv_2\ bv_1+dv_2) \ .
\end{equation}

\begin{definition}
As usual, $\det:\en(V)\rightarrow\mathbb{R}$ is the unique map which satisfies
$$(v^1\cdot A)\wedge(v^2\cdot A)=\det(A)\cdot(v^1\wedge v^2)$$
for $A\in\en(V)$ and $v^1,v^2\in V$.
Let us also set:
\[
\left.\begin{array}{c}\GL(V)=\det^{-1}(\mathbb{R}^*)\\
\SL(V)=\det^{-1}(\{1\})\\ 
\GL^+(V)=\det^{-1}(\mathbb{R}_+^*) \ .
\end{array}\right.
\]
\end{definition}

\begin{remark}
The reference basis induces the isomorphism:
\begin{equation}
\left.\begin{array}{ccc} \Bp & \rightarrow & \GL_2^+(\mathbb{R}) \\ 
\left(\begin{array}{c}f^1\\f^2\end{array}\right) & \mapsto & \left(\begin{array}{cc} f^1_1 & f^1_2 \\ f^2_1 & f^2_2 \end{array}\right) \end{array}\right.
\end{equation}
where $f^i_1$ and $f^i_2$ are the coordinates of $f^i$ in the reference basis ($i=1,2$).
\end{remark}

Let $(v_1\ v_2)_\mathcal{B}$ be the expression of $v$ in coordinates, in a basis $\mathcal{B}$. Then $(v_1,v_2)_\mathcal{B}\cdot M^{-1}$ is the expression in coordinates of the same vector, but in the basis $M\cdot\mathcal{B}$.

\subsection{Lattices}

\begin{definition}
A \textit{lattice} $L$ in $V$ is an additive subgroup of $V$ isomorphic to $\mathbb{Z}^2$ as a $\mathbb{Z}$-module, and such that 
\[
L\otimes_{\mathbb{Z}}\mathbb{R}=V \ .
\]
Let $\La$ be the set of all lattices in $V$. 
\end{definition}

There is a natural surjection:
\begin{equation}
\left.\begin{array}{rcl} \Bp & \rightarrow & \La \\ (v^1,v^2) & \mapsto & \mathbb{Z}v^1+ \mathbb{Z}v^2 \end{array}\right. \ .
\end{equation}
The set $(v^1,v^2)$ is a basis of $L=\mathbb{Z}v^1+ \mathbb{Z}v^2$ as a free $\mathbb{Z}$-module.

\begin{proposition}
Two oriented bases 
$$\left(\begin{array}{c} v^1 \\ v^2 \end{array}\right)=\left(\begin{array}{cc} v^1_1 & v^1_2 \\ v^2_1 & v^2_2 \end{array}\right)\ \mathrm{and}\ \left(\begin{array}{c} w^1 \\ w^2 \end{array}\right)=\left(\begin{array}{cc} w^1_1 & w^1_2 \\ w^2_1 & w^2_2 \end{array}\right)$$
project to the same lattice if the two matrices are related by left-multiplication by an element of $\SL_2(\mathbb{Z})$.
\end{proposition}

\begin{proof}
Let us assume that:
$$\mathbb{Z}v^1+ \mathbb{Z}v^2=\mathbb{Z}w^1+ \mathbb{Z}w^2\ .$$
Then, there exist $m^i_j\in\mathbb{Z}$ and $n^i_j\in\mathbb{Z}$, $i,j=1,2$, such that for all $j=1,2$, one has
$$\left.\begin{array}{c}w^j=m^j_1v^1+m^j_2v^2\\
v^j=n^j_1w^1+n^j_2w^2\end{array}\right.$$
The matrices $M=(m^i_j)$ and $N=(n^i_j)$ are by construction mutually inverse. 
\end{proof}

In other words:
\begin{equation}
\La\simeq \SL_2(\mathbb{Z})\backslash\GL_2^+(\mathbb{R}) \ .
\end{equation}

Let $L_1^{np}\in\La$ be the {\bf (non-projective) reference lattice}, defined as
\begin{equation}\label{refL1}
L_1^{np}=\mathbb{Z}e^1+\mathbb{Z}e^2 \ .
\end{equation}

\subsection{Projective lattices}

The embedding
\begin{equation}
\left.\begin{array}{ccc} \mathbb{R}^\times & \rightarrow & \GL_2^+(\mathbb{R})\\
\alpha & \mapsto & \left(\begin{array}{cc} \alpha & 0 \\ 0 & \alpha \end{array}\right) \end{array}\right. 
\end{equation}
is central, hence there is a well defined left-action of $\mathbb{R}^\times$ on $\La$ given by:
\begin{equation}
\alpha\cdot \left(\SL_2(\mathbb{Z})\cdot(v_1,v_2)\right)= \SL_2(\mathbb{Z})\cdot(\alpha v_1,\alpha v_2)\ .
\end{equation}

\begin{definition}
Let $\PLa=\mathbb{R}^\times\backslash \La$ be the set of \textit{projective lattices} in $V$. 
By definition, a projective lattice is an equivalence class of lattices which are scalar multiples of each other.
\end{definition}

Let also $\PBp$ be the set of projective oriented bases in $V$, that is, $\PBp=\mathbb{R}^\times\backslash \Bp$. Hence
$\PLa\simeq \psl\backslash\PBp$, where $\psl=\{\pm1\}\backslash\SL_2(\mathbb{Z})$.
Once again, the reference basis in $V$ induces an isomorphism:
\begin{equation}
\PLa\simeq\psl\backslash \PGL_2^+(\mathbb{R}) \ .
\end{equation}

\begin{example}
The projective lattice corresponding to the coset
$$\psl\cdot\left[\begin{array}{cc} f^1_1 & f^1_2 \\ f^2_1 & f^2_2 \end{array}\right]$$
is the projective class containing the lattice generated by the vectors $f^1=(f^1_1\ f^1_2)$ and $f^2=(f^2_1\ f^2_2)$, where the coordinates are the ones in the reference basis. 
\end{example}

\subsection{Commensurable lattices}

\begin{definition}
A (non-projective) lattice $L^{np}\in\La$ is said to be {\bf commensurable} with $L_1^{np}$ if the intersection $L^{np}\cap L_1^{np}$ has finite index in both $L^{np}$ and $L_1^{np}$.
\end{definition}
Consider the two-dimensional $\mathbb{Q}$-vector space
\begin{equation}
V_1=L_1^{np}\otimes_\mathbb{Z}\mathbb{Q}\subset V \ .
\end{equation}
It satisfies $V_1\otimes_\mathbb{Q}\mathbb{R}=V$.
\begin{remark}\label{rem:L1}
The lattices in $V$ which are commensurable with $L_1^{np}$ correspond exactly to the additive subgroups of $V_1$ isomorphic to $\mathbb{Z}^2$ as $\mathbb{Z}$-modules. 
\end{remark}

Let $\Bp_1$ be the set of oriented bases of $V_1$, and let
\begin{equation}
\La_1:=\SL_2(\mathbb{Z})\backslash\Bp_1 \ ,
\end{equation}
By Remark \ref{rem:L1}, $\La_1$ is the subset of $\La$ which contains the lattices in $V$ commensurable with $L_1^{np}$. 
The reference basis induces the isomorphism
\begin{equation}
\La_1\simeq \SL_2(\mathbb{Z})\backslash \GL_2^+(\mathbb{Q}) \ .
\end{equation}
Let the rational projectivisation of the set of lattices commensurable with $L_1^{np}$ be the set of rationally projective lattices such that one (equivalently, all) of their representatives is commensurable with $L_1^{np}$:
$$\PLa_1=\mathbb{Q}^\times\backslash \La_1\simeq \psl\backslash\PBp_1 \ , $$ 
where $\PBp_1=\mathbb{Q}^\times\backslash \Bp_1$. The reference basis again induces:
\begin{equation}
\PLa_1\simeq\psl\backslash\PGL_2^+(\mathbb{Q}) \ .
\end{equation}
The rational projectivisation of $L_1^{np}$ is denoted $L_1$ and called the reference projective lattice, or {\bf reference lattice}, for short. We drop the $\mathrm{P}$ (standing for projective) in $L_1$ in order to keep the notation as light as possible. Hopefully, the superscript on $L_1^{np}$ which emphasizes the non-projective nature of the latter will help keeping things clear.

\subsection{Hyperdistance on $\PLa_1$}
Let $M=(m_i^j)$ be a non-zero $2\times2$ matrix with rational coefficients. 
There exists a smallest strictly positive rational number $\alpha_M$ such that 
$$\forall i,j\in\{1,2\},\ \alpha_M m_i^j \in\mathbb{Z}\ . $$

Let us consider the map
\begin{equation}
\left.\begin{array}{rcl}\mathrm{Pdet}: \mathcal{M}_2(\mathbb{Q}) & \rightarrow & \mathbb{Z}\\
M & \mapsto & \det (\alpha_M M)=\alpha_M^2 \det(M)\end{array}\right. \ . 
\end{equation}
For all $x\in\mathbb{Q}^\times$ one has $\mathrm{det}(xM)=\mathrm{det}(M)$, hence this map is well defined on the rational projective space $\mathrm{P}\mathcal{M}_2(\mathbb{Q})$.

\begin{proposition}
Let $A\in\mathrm{SL}_2(\mathbb{Z})$. Then for all $X\in\mathcal{M}_2(\mathbb{Q})$, one has:
$$\mathrm{Pdet}([AX])=\mathrm{Pdet}([X])=\mathrm{Pdet}([XA])$$
where $[X]$ denotes the rational projective class of $X$.
\end{proposition}
\begin{proof}
It suffices to show that $\alpha_{AX}=\alpha_{X}=\alpha_{XA}$. Since $A$ and $A^{-1}$ have integer entries, $\alpha AX$ has integer entries if and only if $\alpha X$ has integer entries, hence 
$$\{\alpha | \alpha X \in\mathcal{M}_2(\mathbb{Z})\}=\{\alpha | \alpha AX \in\mathcal{M}_2(\mathbb{Z})\} \ , $$
and they have the same minimal element.
\end{proof}

\begin{definition}
The projective determinant (still denoted $\mathrm{Pdet}$) is the (induced) function
$$\mathrm{Pdet}:\mathrm{PSL}_2(\mathbb{Z})\backslash\PGL_2^+(\mathbb{Q})\rightarrow \mathbb{N}_{>0} \ .$$
It is invariant under the right-action of $\mathrm{PSL}_2(\mathbb{Z})$.
\end{definition}

Let $L,L'\in\PLa_1$, and let $M,M'$ be representatives in $\GL_2^+(\mathbb{Q})$ of the corresponding elements in $\mathrm{PSL}_2(\mathbb{Z})\backslash\PGL_2^+(\mathbb{Q})$. Set:
\begin{equation}\label{delta}
\delta(L,L')=\mathrm{Pdet}(M(M')^{-1}) \ .
\end{equation}

\begin{proposition}
The function 
$$\delta:\PLa_1\times\PLa_1\rightarrow\mathbb{N}_{>0}$$ 
is symmetric. 
It is called {\bf hyperdistance}.
\end{proposition}
\begin{proof}
Let $M\in\mathcal{M}_2(\mathbb{Z})$ be invertible as a rational matrix. Then, $\det(M)M^{-1}\in\mathcal{M}_2(\mathbb{Z})$. Replacing $M$ in $\det(M)M^{-1}$ with $\det(M)M^{-1}$ implies that if $\det(M)M^{-1}$ is an invertible rational matrix with integer entries, $\det\left(\det(M)M^{-1}\right) \left(\det(M)M^{-1}\right)^{-1}=M\in\mathcal{M}_2(\mathbb{Z})$. 

Thus a $2\times2$ invertible rational matrix $M$ has integer entries if and only if $\det(M)M^{-1}$ does, and hence $\alpha_MM$ has integer entries if and only if  $\alpha_M\det(M)M^{-1}$ does. This implies:
$$\alpha_{M^{-1}}=\alpha_M\det(M)\ .$$
As a consequence of this last equality, one has $\mathrm{Pdet}(M)=\mathrm{Pdet}(M^{-1})$, which proves the claim.
\end{proof}

\begin{remark}
The logarithm of the (judiciously named) hyperdistance is a metric on $\PLa_1$ (see \cite{conway}). Note that in dimension strictly greater than $2$, the function analogous to $\delta$ is not symmetric anymore. 
\end{remark}

Let $N\in\mathbb{N}_{>0}$. The set of projective lattices $N$-hyperdistant from $L_1$ is the set
\begin{equation}
\mathrm{P}\mathcal{L}_1^N=\{L\in\PLa_1|\delta(L,L_1)=N\} \ .
\end{equation}
This particular subset of $\PLa_1$ can be characterised as follows. Let $\tilde{L}^{np}$ be any representative of some $L\in\PLa_1$ such that $\tilde{L}^{np}$ is a subgroup of $L_1^{np}$. The index of $\tilde{L}^{np}$ in $L_1^{np}$ is as usual the order of the finite cyclic abelian group $\tilde{L}^{np}\backslash L_1^{np}$. Then, $\mathrm{P}\mathcal{L}_1^N$ consists of the projective lattices in $\PLa_1$ such that among all their representatives which are subgroups of $L_1^{np}$, the minimum of the index function is $N$.

For exemple, consider any sublattice $\tilde{L}^{np}$ of index $2$ in $L_1^{np}$. Since $2$ is prime, the projective class of $\tilde{L}^{np}$ is always $2$-hyperdistant from $L_1$. However, the representative $2\cdot\tilde{L}^{np}$ of the same projective lattice is of index $8$ $(= 2\times 2^2)$ in $L_1^{np}$.

\subsection{Elements of $\PLa_1$}
Let us describe and label the elements of $\PLa_1=\psl\backslash\PGL_2^+(\mathbb{Q})$ as in \cite{conway}.
Consider the map
\begin{equation}
\GL_2^+(\mathbb{Q})\rightarrow \psl\backslash\PGL_2^+(\mathbb{Q}) \ .
\end{equation}
For each coset $\psl\cdot g$ in its image, let
$g=\left[\begin{array}{cc} a & b \\ c & d \end{array}\right]\in\PGL_2^+(\mathbb{Q})$ denote the projective class of the matrix
$\left(\begin{array}{cc} a & b \\ c & d \end{array}\right)\in\GL_2^+(\mathbb{Q})$.

Let $s,t\in\mathbb{Z}$ be such that $sa+tc=0$, with $s$ and $t$ relatively prime. Since the columns of $g$ are linearly independent, it must be that $sb+td\neq0$. Since $s$ and $t$ have no common factor, there exist $m,n\in\mathbb{Z}$ such that $mt-sn=1$. In other words, there exists $H\in\psl$ such that:
$$H\cdot g=\left[\begin{array}{cc} m & n \\ s & t \end{array}\right]\left[\begin{array}{cc} a & b \\ c & d \end{array}\right]=\left[\begin{array}{cc} a' & b' \\ 0 & d' \end{array}\right]=g' \ , $$
with $b'\in\mathbb{Q}$ and $a',d'\in\mathbb{Q}^\times$. 
Moreover,
$\left[\begin{array}{cc} a' & b' \\ 0 & d' \end{array}\right]=\left[\begin{array}{cc} a'' & b'' \\ 0 & 1 \end{array}\right]$
with $a''=a'/d'$ and $b''=b'/d'$. 
Let $N$ be the unique integer such that $0\leq b''+N < 1$. Then, left-multiplication of the latter element of $\PGL_2^+(\mathbb{Q})$ by 
$\left[\begin{array}{cc} 1 & N \\ 0 & 1 \end{array}\right]\in\psl$
yields some
$\left[\begin{array}{cc} M & b \\ 0 & 1 \end{array}\right]$.
Furthermore, the only element in $\psl$ which maps representatives of projective classes of this form to representatives of the same form is easily shown to be the identity. Hence we have proved the following 

\begin{proposition}\label{representatives}
Let $\mathcal{M}$ be the set of matrices of the form
$\left(\begin{array}{cc} M & b \\ 0 & 1 \end{array}\right)$
with $M\in\mathbb{Q}_+^*$ and $b\in\mathbb{Q}\cap[0,1[$. Then
$$\left.\begin{array}{rcl}\mathcal{M} & \rightarrow & \psl\backslash\PGL_2^+(\mathbb{Q})\\ \left(\begin{array}{cc} M & b \\ 0 & 1 \end{array}\right) & \mapsto & \psl\cdot \left[\begin{array}{cc} M & b \\ 0 & 1 \end{array}\right]\end{array}\right.$$
is a bijection. 
\end{proposition}

\begin{definition}
Let $g_{M,b}$ denote the coset 
$$\psl\cdot\left[\begin{array}{cc} M & b \\ 0 & 1 \end{array}\right]\in\psl\backslash\PGL_2^+(\mathbb{Q}) \ .$$
Let $L_{M,b}:=\psl\cdot g_{M,b}$ be the projective lattice corresponding to the class of $g_{M,b}$. We always shorten $g_{M,0}$ and $L_{M,0}$ to $g_M$ and $L_M$. 
\end{definition}
Note that this definition of $L_1$ coincides with the first one we considered.

\begin{corollary}
This classification of the cosets in $\psl\backslash\PGL_2^+(\mathbb{Q})$ implies that any projective lattice commensurable with $L_1$ has a unique non-projective representative with basis of the form
$$f^1=(M\ b),\ f^2=(0\ 1)\ ,$$
where $M\in\mathbb{Q}_+^*$ and $b\in\mathbb{Q}\cap[0,1[$.
\end{corollary}

\begin{example}
The projective lattice $L_N$ corresponds to the coset
$\psl\cdot\left[\begin{array}{cc} N & 0 \\ 0 & 1 \end{array}\right]$, and
hence to the class of non-projective lattices 
$\{\mathbb{Z}\cdot(\alpha N\ 0)+\mathbb{Z}\cdot(0\ \alpha)|\ \alpha\in\mathbb{Q}^\times\}$.
\end{example}

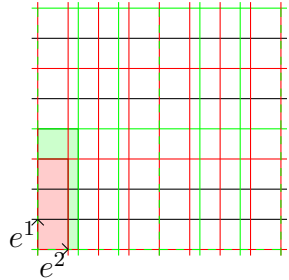
\begin{figure}[h!t!]
\centering
\begin{tikzpicture}[scale=0.8]
\filldraw[fill=green!20!white, draw=green!40!black] (0,0) rectangle (0.667,2);
\filldraw[fill=red!20!white, draw=green!40!black] (0,0) rectangle (0.5,1.5);
\draw[red] (-0.1,0)--(4.1,0);
\draw[dashed, green] (-0.1,0)--(4.1,0);
\draw (-0.1,0.5)--(4.1,0.5);
\draw (-0.1,1)--(4.1,1);
\draw[red] (-0.1,1.5)--(4.1,1.5);
\draw[green] (-0.1,2)--(4.1,2);
\draw (-0.1,2.5)--(4.1,2.5);
\draw[red] (-0.1,3)--(4.1,3);
\draw (-0.1,3.5)--(4.1,3.5);
\draw[green] (-0.1,4)--(4.1,4);

\draw[red] (0,-0.1)--(0,4.1);
\draw[dashed, green] (0,-0.1)--(0,4.1);
\draw[red] (0.5,-0.1)--(0.5,4.1);
\draw[red] (1,-0.1)--(1,4.1);
\draw[red] (1.5,-0.1)--(1.5,4.1);
\draw[red] (2,-0.1)--(2,4.1);
\draw[dashed, green] (2,-0.1)--(2,4.1);
\draw[red] (2.5,-0.1)--(2.5,4.1);
\draw[red] (3,-0.1)--(3,4.1);
\draw[red] (3.5,-0.1)--(3.5,4.1);
\draw[red] (4,-0.1)--(4,4.1);
\draw[dashed, green] (4,-0.1)--(4,4.1);

\draw[green] (0.667,-0.1)--(0.667,4.1);
\draw[green] (1.33,-0.1)--(1.33,4.1);
\draw[green] (2.667,-0.1)--(2.667,4.1);
\draw[green] (3.33,-0.1)--(3.33,4.1);

\draw (0.5,0)--(0.4,0.1);
\draw (0.5,0)--(0.4,-0.1);
\draw (0,0.5)--(0.1,0.4);
\draw (0,0.5)--(-0.1,0.4);

\draw (-0.25,0.25) node{$e^1$};
\draw (0.25,-0.25) node{$e^2$};
\end{tikzpicture}
\caption{{\sf Two non-projective representatives of $L_3$ (in green and red) on $L_1^{np}$.}
\label{f:egcoset}}
\end{figure}

\subsection{Stabilisers and Hecke Congruence Subgroups of $\psl$}
Let $G:=\PGL_2^+(\mathbb{Q})$ and consider its right-action on $\PLa_1$
\begin{equation}
\left[\begin{array}{cc} f^1_1 & f^1_2 \\ f^2_1 & f^2_2 \end{array}\right]\cdot\left[\begin{array}{cc} a & b \\ c & d \end{array}\right]=\left[\begin{array}{cc} af^1_1+cf^1_2 & bf^1_1 + df^1_2 \\ af^2_1+ cf^2_2 & bf^2_1+ df^2_2\end{array}\right] \ .
\end{equation}
Let $G_L:=\fix_G(L)$ be the stabiliser of $L\in\PLa_1$ in $G$. The group $G_{L_1}$ is easily shown to be $\psl$. This is the definition of the modular group we were aiming for. Now, since $G$ acts transitively on $\PLa_1$, the stabiliser of any $L\in\PLa_1$ is a conjugate of $G_1$ in $\PGL_2^+(\mathbb{Q})$. 
For example, and for $M\in\mathbb{Q}_+^*$, one has
\begin{equation}
G_M=g_M^{-1}G_1g_M=\left\{\left[\begin{array}{cc} a & b/M \\ cM & d \end{array}\right] | \ a,b,c,d\in\mathbb{Z},\ ad-bc=1\right\} \ .
\end{equation}
Subsequently, the subgroup of $G$ which stabilizes the pair $(L_1,L_N)$ is $G_{(L_1,L_N)}=G_1\cap G_N$. For $N\in\mathbb{N}_{>0}$ one has:
\begin{equation}
G_{(L_1,L_N)}=\left\{\left[\begin{array}{cc} a & b \\ cN & d \end{array}\right] | \ a,b,c,d \in\mathbb{Z},\ ad-bcN=1\right\} \ .
\end{equation}
\begin{definition}
Let $N$ be a positive integer. 
The {\bf Hecke congruence subgroup} of level $N$ of the modular group is the group
\[
\He(N):=\left\{\left[\begin{array}{cc} a & b \\ c & d \end{array}\right] | \ a,b,c,d \in\mathbb{Z},\ ad-bc=1,\ c\equiv 0 [N]\right\}<\psl \ .
\]
Note that $\He(1)=\psl$. 
\end{definition}

\bibliography{ref}

\end{document}